\documentclass[11pt,letterpaper]{amsart}

\usepackage{color}
\usepackage[linkcolor=blue,citecolor=blue, colorlinks=true,backref=true,bookmarks=true]{hyperref}

\usepackage{amsmath}
\usepackage{amssymb}
\usepackage{mathtools}
\usepackage{todonotes}
\usepackage{leftindex}

\usepackage{scalerel}
\usepackage{stackengine,wasysym}

\numberwithin{equation}{section} 

\allowdisplaybreaks[4]

\theoremstyle{plain}
\newtheorem{theorem}{Theorem}[section]
\newtheorem{corollary}[theorem]{Corollary}

\newtheorem{lemma}[theorem]{Lemma}
\newtheorem{proposition}[theorem]{Proposition}
\newtheorem{definition}[theorem]{Definition}

\theoremstyle{definition}
\newtheorem{remark}[theorem]{Remark}

\renewcommand{\i}{i\mkern1mu}

\newcommand{\N}{\mathbb{N}}
\newcommand{\C}{\mathbb{C}}
\newcommand{\R}{\mathbb{R}}
\newcommand{\Z}{\mathbb{Z}}

\newcommand{\tr}{\operatorname{tr}}
\newcommand{\Tr}{\operatorname{Tr}}

\newcommand{\End}{\mathrm{End}}

\renewcommand{\Im}{\mathfrak{I}}
\newcommand{\Cl}{\mathrm{Cl}}
\renewcommand{\Re}{\operatorname{Re}}
\renewcommand{\Im}{\operatorname{Im}}

\newcommand{\coker}{\operatorname{coker}}

\renewcommand{\epsilon}{\varepsilon}

\newcommand{\inner}[1]{\langle#1\rangle}    
\newcommand{\norm}[1]{\lVert#1\rVert}       
\newcommand{\abs}[1]{\lvert#1\rvert}       
\renewcommand\d {\,\mathrm{d}}              

\newcommand{\spec}{\operatorname{spec}}  

\newcommand{\ind}{\mathrm{ind}}

\newcommand{\vertiii}[1]{{\left\vert\kern-0.25ex\left\vert\kern-0.25ex\left\vert #1 
    \right\vert\kern-0.25ex\right\vert\kern-0.25ex\right\vert}}  

\usepackage[
backend=biber,
style=alphabetic,
sorting=nyt
]{biblatex}
\title{Adiabatic Limit and Analytic Torsion of Vector Bundles}
\author{Xianzhe Dai and Debin Liu}

\begin{document}
\maketitle


\begin{abstract}
    For a vector bundle $E^{n+k}$ over a closed manifold $M^n$ with $k$ even and $n$ odd, we equip the metric with an adiabatic parameter, and prove that the index of $E$ is the same as the index of $M$. We also introduce an analog of analytic torsion on $E$ using the Witten Laplacian. Moreover, we prove that the Quillen metric associated with this analytic torsion coincides with that of $M$.
\end{abstract}

\tableofcontents

\section{Introduction}
The adiabatic limit refers to the geometric degeneration in which the metric is blown up along certain directions. This blow-up phenomenon usually reveals important information about the metric that cannot be seen using regular methods. In particular, it is useful in the study of geometric invariants. Historically, the study of adiabatic limit was initiated by E. Witten in \cite{Witten85}, where he related the adiabatic limit of the $\eta$-invariant to the "global anomaly". 

Another important geometric invariant is the analytic torsion, which is an analytic analog of the R-torsion in topology. It was defined by D. B. Ray and I. M. Singer \cite{RS71} as a combination of the determinants of Hodge Laplacians on differential forms. They also conjectured that the topological R-torsion coincides with the analytic torsion. This conjecture was verified independently by J. Cheeger \cite{Cheeger77},\cite{Cheeger79} and W. M{\" u}ller \cite{Muller78} and this result is now known as Cheeger-M{\" u}ller theorem. To be more precise, let $(M, g)$ be a closed Riemannian manifold. Using supertrace, one can define the torsion zeta function to be
\begin{equation} \label{equa:zeta_func}
    \zeta_T(s) = \frac{1}{\Gamma(s)} \int_0^\infty t^{s-1} \tr_s \left( Ne^{-t \underline{\Delta}}\right) \d t,
\end{equation}
where $\Gamma(s)$ is the Gamma function, $N$ is the number operator (see section \ref{section:prel}), and $\underline{\Delta}$ is the Hodge Laplacian with its null space removed. Intuitively, $\zeta_T$ defines a holomorphic function only when $\mathrm{Re} (s)$ is sufficiently large. But it can be proved, by a careful study of the asymptotic expansion of the heat kernel, that $\zeta_T(s)$ has an analytic continuation as a meromorphic function which is regular near $s=0$. The Ray-Singer analytic torsion is then defined as follows.
\begin{definition}
    The analytic torsion $T(M)$ of $(M, g)$ is defined by
    \begin{equation} \label{equa:analytic_torsion}
        \log T(M) = -\frac{1}{2} \zeta'_T(0).
    \end{equation}
\end{definition}

The adiabatic limit of analytic torsion was studied by R. B. Melrose and the first author in \cite{DM}, where they considered a compact fibration $Y \hookrightarrow E \xrightarrow{\pi} M$ over a closed manifold $M$. Given a submersion metric $g = \pi^* g_M + g_Y$ on $E$, the adiabatic limit refers to the limit as $\epsilon \to 0$ of 
\begin{equation}
    g_\epsilon = \frac{1}{\epsilon^2} \pi^*g_M + g_Y.
\end{equation}
Employing the technique of  microlocal analysis, the adiabatic limit of the analytic torsion of the total space is related to those of the base and the fiber, as well as certain topological invariants of the fibration.

The main purpose of this paper is to study the behavior of the adiabatic limit of the index and analytic torsion of the total space of a vector bundle. The challenge of this problem comes from the noncompactness of the fibers. As a result, the spectrum of some important geometric operators, such as Hodge Laplacian, is usually not discrete, causing many technical difficulties. To overcome them, we introduce a Morse-Bott function $h$ on $E$ which satisfies the strong tame condition in the sense of X. Dai and J. Yan \cite{DY22}. Then the Witten Laplacian $D^2_{\epsilon, \tau h}$ of $g_\epsilon$ possesses a spectrum which is discrete, and the analytic torsion for the total space of a vector bundle can be defined in the same way as in the compact case but using the asymptotic expansion developed in \cite{DY22}.

More precisely, let $Y \hookrightarrow E \xrightarrow{\pi} M$ be a vector bundle over a closed manifold $M$. Denote by $D^2_{\epsilon, \tau h}$ the Witten Laplacian associated with $g_\epsilon$ and $h$ (see section \ref{section:asymtotic_expansion} for details). Let $H^*_{(2)}(E, d_{\tau h})$ denote the $L^2$-cohomology of $E$ with respect to the Witten deformed exterior differentiation. Following the ideas of \cite{DY22} we first prove

\begin{theorem} \label{thm:ind_equal}
    Suppose $E$ is a vector bundle over a compact manifold $M$. Then 
     \begin{equation}
       H^*_{(2)}(E, d_{\tau h}) \simeq H^*(M).
    \end{equation}
In particular, 
    \begin{equation}
        \ind(E) = \ind(M).
    \end{equation}
\end{theorem}

Next, we extend the analytic torsion to the vector bundle case, and study its relationship with the usual analytic torsion of the base manifold.
As in (\ref{equa:zeta_func}) and (\ref{equa:analytic_torsion}), one can formally define the torsion zeta function $\zeta_{\epsilon, \tau h}(s)$ and the analytic torsion $T_{\epsilon, \tau h}(E)$ by
\begin{equation} \label{equa:zeta_func_noncompact}
    \zeta_{\epsilon, \tau h}(s) = \frac{1}{\Gamma(s)} \int_0^\infty t^{s-1} \tr_s \left( Ne^{-t \underline{D}^2_{\epsilon, \tau h}}\right) \d t,
\end{equation}

and 

\begin{equation}\label{equa:analytic_torsion}
    \log T_{\epsilon, \tau h}(E) = -\frac{1}{2} \zeta'_{\epsilon, \tau h}(0).
\end{equation}

First, using the asymptotic expansions given in \cite{DY22}, we show that these definitions are indeed valid.
\begin{theorem} \label{thm:main1_anal_tor}
    For any $\epsilon > 0$, $\zeta_{\epsilon, \tau h}(s)$ admits an analytic continuation as a meromorphic function on $\C$, and it is analytic near $0$. Consequently, $\log T_{\epsilon, \tau h} (E)$ is well-defined.
\end{theorem}

We further assume that $\dim M = n$ is odd, and $\dim Y = k$ is even. When $\epsilon = 1$, we will write $\log T_{\tau h}(E)$ for $\log T_{\epsilon, \tau h}(E)$. The main result of our paper is 
\begin{theorem} \label{thm:main2_limit}
    Suppose $M^n$ is a closed manifold, and $E^{n+k} \to M^n$ is a vector bundle of rank $k$. Assume $n$ is odd and $k$ is even. Then for any $\tau \in (0, \infty)$, 
    \begin{equation}
        \log T_{\tau h}(E) - \log \left(\frac{\leftindex_\infty{\abs{~}}}{\abs{~}}\right) = \log T(M).
    \end{equation}
\end{theorem}

By Theorem \ref{thm:ind_equal}, the determinant line bundle of $E$ can be identified with that of $M$. We can equip each determinant line bundle with a Quillen metric $\norm{~}$, which is the product of the standard $L^2$ metric and the analytic torsion. In this way, Theorem \ref{thm:main2_limit} can be interpreted as $\norm{~}_{E, \tau h} = \norm{~}_M$.

Intuitively, analyzing the behavior of the analytic torsion under the adiabatic limit may link $\log T_{\epsilon, \tau h} (E)$ to $\log T(M)$. Unfortunately, there are two main issues. First, since the torsion zeta function is an integral from $0$ to infinity, when the time $t$ is large, the supertrace behaves by nature differently than when $t$ is small. This forces us to deal with large and small time separately. If $t$ is large, we decompose $\frac{1}{\epsilon}D_{\epsilon, \tau h}$ under the splitting $L^2(\Lambda^* E) = L^2(\ker \tilde{D}_Y)^\perp \oplus L^2(\ker \tilde{D}_Y)$, where $\tilde{D}_Y$ is the Witten-Dirac operator on the fiber $Y$. Then a careful study of the resolvent yields the desired estimate. On the other hand, local index theory plays an important role in the small time case. Motivated by the rescaling technique of Getzler \cite{G86}, we introduce several different rescalings that fit into various situations, and use them to derive the corresponding local index theorems as well as some small time estimates.

Second, due to the lack of a uniform estimate, the two limits $\lim_{\epsilon \to 0}$ and $\lim_{t \to \infty}$ may not commute. To solve this issue, we follow the idea of A. Berthomieu and J.-M. Bismut \cite{BB1994} (see also \cite{BL}) to construct a closed $1$-form $\alpha_{t, T}$ on $\R_+ \times \R_+$ and integrate it over a closed rectangular contour. We analyze the asymptotic behavior of the integral of $\alpha_{t, T}$ over each side as we push two sides to infinity and one side to the $x$-axis. Then Theorem \ref{thm:main2_limit} follows from matching the divergent terms that appear in the above limiting process.

This paper is organized as follows: Section \ref{section:prel} collects results about fiber bundle, Clifford multiplication, and functional calculus. In section \ref{section:asymtotic_expansion}, we introduce a Witten Laplacian with a particular Morse-Bott function as the potential function, and calculate the asymptotic expansion of the supertrace. Then we give a proof of Theorem \ref{thm:main1_anal_tor}. In section \ref{section:inter_result} we collect some technical results. Some of the proofs are postponed to section \ref{sect:large_time}-\ref{sect:proof_last_thm}. In section \ref{section:proof_main}, we use results in the previous section to prove Theorem \ref{thm:main1_anal_tor}.

In section \ref{sect:large_time}, we study the large time adiabatic limit of the heat operator for both the compact fiber case and the noncompact fiber case. We also prove Theorem \ref{thm:ind_equal}. In section \ref{sect:local_ind_thm}, we prove several local index theory type results, and use them to obtain small time estimates. In section \ref{sect:proof_last_thm}, we study the uniform behavior of the heat kernel carefully, and provide a uniform estimate for the supertrace when the time is large.

\section*{Acknowledgments}
The first author was partially supported by the Simons Foundation. 

\section{Preliminaries}\label{section:prel}

\subsection{Elementary geometry of fiber bundles}

In this section, we collect some preliminaries about fiber bundles and give a brief review of some constructions in \cite{B1_family} and \cite{BC}.

Let $Y^k \hookrightarrow E^{n+k} \xrightarrow{\pi} M^n$ be a fiber bundle over a compact manifold $M$, with (possibly noncompact) fiber $Y$. Then there is a naturally defined subbundle $VE \subset TE$, called the vertical subbundle, given by 
\[ VE = \ker \pi_* = \{ X \in \Gamma(TE) \mid \pi_*(X) = 0 \}. \]
For any $x \in M$, we denote by $E_x = \pi^{-1}(x)$ the fiber of $E$ at $x$. Then for any $p \in E_x$, $V_p E$ consists of vectors tangent to the fiber $E_x$ at $p$, i.e. $V_pE \cong T_p E_x$. 

Now we fix a connection $\nabla$ on the vector bundle $E$. Recall that this amounts to a splitting of its tangent bundle:
\[ TE = HE \oplus VE, \]
where $HE$ is called the horizontal subbundle of the connection.
Let $g_M$ be a smooth metric on $M$. For convenience, we still denote by $g_M = \pi^* g_M$ the pullback metric on $HE$. Suppose $g_Y$ is a smooth family of flat metrics on the fibers, or equivalently a metric on the vector bundle $VE$.
Then we equip $TE$ with a metric
\begin{equation} \label{equa:metric}
    g_E = g_M + g_Y
\end{equation}
 by requiring that $g_E$ preserves the splitting. As a result, the splitting $TE = HE \oplus VE$ is orthogonal with respect to $g_E$.

Let $P^H$, $P^V$ denote the projections onto $HE$, $VE$ respectively. Throughout the paper, we will use $X \in \Gamma(TE)$ to denote a tangent vector on $E$, with $X^H \in \Gamma(HE)$ and $X^V \in \Gamma(VE)$ being its horizontal and vertical components. Also, we denote by $U\in \Gamma(TM)$ a tangent vector on the base manifold $M$ as well as its horizontal lift to $E$. $W \in \Gamma(VE)$ will denote a vertical vector field.

Now we introduce a natural connection that faithfully reflects the splitting structure of the tangent bundle (see  \cite[section 1(b)]{B1_family}, \cite[section 4(a)]{BC}). Let $\nabla^{L}$ denote the Levi-Civita connection on $TE$ with respect to $g_E$.

\begin{definition} \label{def:conn}
    The connection $\nabla^E$ on $TE$ is defined by 
    \begin{equation} \label{equa:conn}
        \begin{split}
            &\nabla^E_{U_1} U_2 = \nabla^M_{U_1} U_2, ~~~ \nabla^E_U W = P^V(\nabla^L_U W) \\
            &\nabla^E_W U = 0,  ~~~~~~~~~~~~ \nabla^E_{W_1}W_2 = P^V(\nabla^L_{W_1} W_2),
        \end{split}
    \end{equation}
    where $U_i, U \in \Gamma(TM)$, $W_i, W \in \Gamma(VE)$.
    Let $S$ be the difference tensor $S = \nabla^L - \nabla^E$.
\end{definition}

The following proposition summarizes some properties of this connection.

\begin{proposition} \label{prop:S}
    Let $U_i, U \in \Gamma(TM)$, $W_i, W \in \Gamma(VE)$. Then the only nonvanishing components of $S$ are 
    \begin{equation}
        \inner{S(W_1)W_2, U} = - \inner{S(W_1)U, W_2} = - \inner{\nabla^L_{W_1}U, W_2}
    \end{equation} 
    and
    \begin{equation}
        \begin{split}
            \inner{S(U_1)W,  U_2} = - \inner{S(U_1)U_2, W} = \inner{S(W)U_1, U_2} = -\frac{1}{2} \inner{[U_1, U_2], W}.
        \end{split}
    \end{equation}
    Moreover, we have 
    \begin{enumerate}
        \item $S(U_1) U_2 \in \Gamma(VE)$,  $S(\cdot)W \in \Gamma(HE)$.
        \item $S(W)U \in \Gamma(HE) \oplus \Gamma(VE)$.
    \end{enumerate}
\end{proposition}

Next, we consider an adiabatic metric $g_\epsilon = \frac{1}{\epsilon^2} g_M + g_Y$. Recall that  $\nabla^L$ denotes the Levi-Civita connection for $g_E = g_1$. Let $S^\epsilon = \nabla^{L,\epsilon} - \nabla^E$ be the difference tensor. 
\begin{proposition} \label{prop:S_epsilon}
    \begin{enumerate}
    \item $S^\epsilon = \epsilon^2 P^H S + P^V S$.
        \item $\nabla^{L,\epsilon} = \nabla^E + S^\epsilon =  \nabla^E + \epsilon^2 P^H S + P^V S$
        \item $g_E(S(\cdot)\cdot, \cdot)$ is independent of $g^M$. In particular, $g_\epsilon(S^\epsilon(X)Y, Z) = g_E(S(X)Y, Z)$ for any $X, Y, Z \in \Gamma(TE)$.
    \end{enumerate}
\end{proposition} 

\subsection{Clifford multiplication} \label{subsection:Clifford multiplication}
In this section we recall several facts concerning Clifford multiplications (see \cite{BZ92} and \cite{Freed_notes} for references).

We start with some algebraic preliminaries. Let $V$ be an $n$-dimensional real vector space, equipped with an inner product $g_V$. Let $\{e_1, \cdots, e_n\}$ be an orthonormal basis of $V$, and $\{e^1, \cdots, e^n\}$ be its dual basis. Then the Clifford algebra $\Cl(V^*)$ is the algebra spanned by $V^*$, subject to the relations
\begin{equation}
    e^i \cdot e^j + e^j \cdot e^i = -2 \delta_{ij}.
\end{equation}
The canonical isomorphism between $\Cl(V^*)$ and $\Lambda V^*$ is given by identifying $e^{i_1} e^{i_2} \cdots e^{i_k}$ with $e^{i_1} \wedge e^{i_2} \wedge \cdots \wedge e^{i_k}$, provided that $i_j$'s are mutually distinct. Under this isomorphism, the left Clifford multiplication by $e^i$ becomes $c(e^i) = e^i \wedge - \iota_{e_i}$. In particular, this equips $\Lambda V^*$ with a Clifford module structure $c: \Cl(V^*) \to \End(\Lambda V^*)$. 

\begin{remark}
    Suppose we rescale the metric by a factor $\epsilon^{-2}$, i.e. we consider $\epsilon^{-2} g_V$. 
Using the isometry $(V, \epsilon^{-2}g_V) \to (V,  g_V)$ induced by $\epsilon e_i := \tilde{e}_i \mapsto e_i$, we can identify $\Cl_\epsilon(V^*)$ with $\Cl(V^*)$ as algebras. In this way, we see that the natural Clifford multiplication of $\Cl_\epsilon(V^*)$ on $\Lambda V^*$ is given by 
\[ c(\tilde{e}^i) = \tilde{e}^i \wedge - \iota_{\tilde{e}_i}, \]
and it is independent of $\epsilon$. 
\end{remark}

Now suppose $Y \hookrightarrow E \xrightarrow{\pi} M$ is  a fibration equipped with the same metric $g_E$ as in (\ref{equa:metric}).
Let $\{e_1, \cdots, e_n\}$ and $\{f_1, \cdots, f_k\}$ denote local orthonormal frames for $TM$ and $TY$ respectively. The readers should notice that the notations we adopt here are opposite to most of the conventions in the literature, but are more natural in our setting. 
Let $c(e^i) = e^i \wedge - \iota_{e_i}$ and $\hat{c}(e^i) = e^i + \iota_{e_i}$ denote the left and "right" Clifford multiplications (for $f_i$ they are defined similarly). It is not hard to verify that
\begin{proposition} For convenience, we temporarily denote $f_i$ by $e_{n + i}$.
    \begin{enumerate}
        \item As a $C^\infty(E)$-algebra, $\Gamma\big(\End (\Lambda^* E)\big)$ is generated by $1$ and the $c(e^i), c(f^i)$, $ \hat{c}(e^j), \hat{c}(f^j)$'s.
        \item $c(e^i) c(e^j) + c(e^j) c(e^i) = -2 \delta_{ij}$
        \item $\hat{c}(e^i) \hat{c}(e^j) + \hat{c}(e^j) \hat{c}(e^i) = 2 \delta_{ij}$
        \item $c(e^i) \hat{c}(e^j) + \hat{c}(e^j) c(e^i) = 0 $, i.e. $c$ and $\hat{c}$ anti-commute.
    \end{enumerate}
\end{proposition}

Now we introduce a supertrace defined on the endomorphism bundle, and illustrate a fantastic cancellation concerning the supertrace of Clifford multiplications. We start with a natural $\Z_2$-grading on the space of differential forms:
\begin{equation} \label{equa:exterior_bundle_decomp}
    \Omega^* E = \Omega^+ E \oplus \Omega^- E,
\end{equation}
where $\Omega^+ E$ (or $\Omega^- E$) consists of even (or odd) forms. On the section level, this leads to a $\Z_2$-grading on the endomorphism bundle:
\begin{equation}
    \Gamma\big(\End(\Lambda^* E)\big) = \Gamma\big(\End^+(\Lambda^* E)\big) \oplus \Gamma\big(\End^-(\Lambda^* E)\big).
\end{equation}
In particular, $A \in \Gamma\big(\End^+(\Lambda^* E)\big)$ if $A: \Omega^\pm E \to \Omega^\pm E$, and $A \in \Gamma\big(\End^-(\Lambda^* E)\big)$ if $A: \Omega^\pm E \to \Omega^\mp E$. 
\begin{definition}
    The (global) supertrace $\tr_s : \Gamma\big(\End(\Lambda^* E)\big) \to \C^\infty (M)$ is defined as follows: If $A = \begin{pmatrix}
        A_{11} & A_{12} \\
        A_{21} & A_{22}
    \end{pmatrix}$ with respect to (\ref{equa:exterior_bundle_decomp}), then 
    \begin{equation}
        \tr_s A = \tr A_{11} - \tr A_{22}.
    \end{equation}
\end{definition}

For $I = \{i_1, \cdots, i_p\} \subset \{1, 2, \cdots, n\}$, let $c(e^I)$ denote $c(e^{i_1}) c(e^{i_2}) \cdots c(e^{i_p})$. We have the following well known fact.
\begin{proposition}  \label{prop:tr_s_cliff_multip}
    \begin{equation}
        \tr_s c(e^I f^J) \hat{c}(e^{I'} f^{J'}) = \begin{cases}
            (-1)^{\frac{(n+k)(n+k+1)}{2}} 2^{n+k}, & \text{if } \begin{aligned}[t]
            &I = I' = \{1, \cdots, n\},\\
            &J = J' = \{1, \cdots, k\}; \end{aligned}\\
            0, \text{ if otherwise }.
        \end{cases}
    \end{equation}
\end{proposition}

Following \cite[section 4(d)]{BZ92}, we separate the contributions of left and right Clifford multiplications. Consider a vector bundle $T^*(E) \oplus T^*(\hat{E})$ over $E$, where $T^*(\hat{E})$ denotes another copy of $T^*(E)$. Let $e^1, \cdots, e^n, f^1$, $\cdots, f^k$ denote the local orthonormal basis of the first copy $T^*(E)$, and $\hat{e}^1, \cdots, \hat{e}^n, \hat{f}^1, \cdots, \hat{f}^k$ denote the corresponding orthonormal basis of the second copy $T^*(\hat{E})$. Set $c(e^i) = e^i \wedge - \iota_{e_i}$, $c(f^i) = f^i \wedge - \iota_{f_i}$, and $\hat{c}(e^i) = \hat{e}^i + \iota_{\hat{e_i}}$, $\hat{c}(f^i) = \hat{f}^i + \iota_{\hat{f_i}}$. Then in view of $\Lambda^*(E \oplus \hat{ E}) =  \Lambda^*E \hat{\otimes} \Lambda^* \hat{E}$, $c(e^i)$ and $c(f^i)$ only act on $\Lambda^*E$, while $\hat{c}(f^i)$ and $\hat{c}(f^i)$ only act on $\Lambda^* \hat{E}$. Clearly Proposition \ref{prop:tr_s_cliff_multip} still holds with respect to the induced $\Z_2$-grading on $\End\left( \Lambda^*E \hat{\otimes} \Lambda^* \hat{E} \right)$. 

For any $a = \sum a_{I,J,I'J'} c(e^I f^J) \hat{c}(e^{I'} f^{J'}) \in \End\left( \Lambda^*E \hat{\otimes} \Lambda^* \hat{E} \right)$, let $a' := \sum a_{I,J,I'J'} e^I \wedge f^J \wedge e^{I'} \wedge f^{J'}$ be the associated differential form obtained by replacing the Clifford multiplications with $e^I \wedge f^J \wedge e^{I'} \wedge f^{J'}$. Moreover, we define the Berezin integral $\int^B a'$ to be the coefficient of the top form term. Then we have 

\begin{proposition} \label{prop:tr_s_Berezin_int}
    \begin{equation}
        \tr_s a = (-1)^{\frac{(n+k)(n+k+1)}{2}} 2^{n+k} \int^B a' .
    \end{equation}
\end{proposition}



\subsection{Riesz-Dunford functional calculus}
In this section, we review the so-called Riesz-Dunford functional calculus for closed operators. This serves as a powerful tool for analyzing various properties of operators constructed out of (local-)holomorphic functions. For the purpose of this paper, we will present only a baby version of it. For more details, we refer to \cite[Chapter VII]{dunford1988linear}. 

To begin with, let $H$ be a Hilbert space, $T: H \to H$ be a closed linear operator (possibly unbounded). Let $\spec(T)$ denote the spectrum of $T$. Suppose $f$ is a complex function over $\C$, which is analytic near $\spec(T) \cup \{\infty\}$. 
\begin{theorem} \label{thm:funcional_calculus}
    Suppose 
    \begin{enumerate}
        \item $V \subset \C$ is an open subset containing $\spec(T)$;
        \item The boundary $\Gamma$ of $V$ consists of a finite number of Jordan curves, equipped with positive orientation with respect to $V$;
        \item $f$ is holomorphic on $V \cup \Gamma$.
    \end{enumerate}
    Then $f(T)$ is a well-defined linear operator over $H$, and we have the following Cauchy-type integral formula 
    \begin{equation}
        f(T) = \frac{1}{2\pi i} \int_\Gamma f(\lambda) (\lambda  - T)^{-1} \d \lambda,
    \end{equation}
    provided that the integral on the right-hand side is absolutely convergent.
    Moreover, $\spec\big(f(T)\big) = f\big(\spec(T)\cup \{\infty\}\big)$.
\end{theorem}

As an example, suppose $D$ is an unbounded self-adjoint operator. Then from basic functional analysis theory, we know that $\spec(D) \subset \R$. Consider the function $f(z) = e^{- t z^2}$, where $t>0$ is any fixed number. If we choose $\Gamma$ properly such that $(\lambda - D)^{-1}$ is uniformly bounded, Theorem \ref{thm:funcional_calculus} implies that 
\begin{equation}
    e^{-t D^2} = f(D) = \frac{1}{2\pi i} \int_\Gamma e^{-t \lambda^2} (\lambda - D)^{-1} \d \lambda.
\end{equation}
Moreover, if $\{D_n\}$ is a sequence of such operators that converges to $D$ in the norm resolvent sense (i.e. $(\lambda - D_n)^{-1}$ converges to $(\lambda - D)^{-1}$ under operator norm topology), then it is easy to see that $e^{-t D_n^2}$ and $e^{-t D^2}$ are all bounded operators, and 
\[  \lim_{n \to \infty} e^{-t D^2_n} = e^{-t D^2} \]
in operator norm.

\section{Torsion zeta function and analytic torsion} \label{section:asymtotic_expansion}

In this section, we will apply results in \cite{DY23} and \cite{DY22} to calculate the asymptotic expansion of $\tr_s\Big(N e^{-t D^2_{\epsilon, \tau h}}\Big)$ and provide a proof of theorem \ref{thm:main1_anal_tor}. 

To begin with, consider a Morse-Bott function $h$ locally defined by 
    \[h(x, y) = \frac{1}{2}(y^1)^2 + \frac{1}{2}(y^2)^2 + \cdots  + \frac{1}{2}(y^k)^2,\]
where $y^1, \cdots, y^k$ are the coordinates of the fibers. We will always choose a local chart so that $\abs{\nabla h}^2 = \abs{y^2}$. By choosing the transition functions to be $O(k)$-valued, $h$ can be defined globally.  Moreover, one can easily check that $h$ satisfies the strong tame property, namely:
    \begin{enumerate}
        \item $\limsup_{p \to \infty} \frac{\abs{\nabla^2 h}}{ \abs{\nabla h}^2}(p) = 0$
        \item $\lim_{p \to \infty} \abs{\nabla h} = \infty$.
    \end{enumerate}

\begin{definition} \label{def:Witten_Lap}
    We define the Witten-Dirac operator to be 
    \begin{equation}
        D_{\epsilon, \tau h} = \d_{\tau h} + \d^{*, \epsilon}_{\tau h},
    \end{equation}
    where $\d_{\tau h} = \d + \tau \d h \wedge$ is the Witten deformation of de Rham operator, and $\d^{*, \epsilon}_{\tau h} = \d^{*, \epsilon} + \tau \iota_{\nabla h}$ is its dual operator with respect to the metric $g_{ \epsilon}$. The Witten Laplacian $D_{\epsilon, \tau h}^2$ is defined as the square of the Witten-Dirac operator.
\end{definition}

An observation is since $h$ is strong tame, it follows from \cite[Theorem 2.2]{DY23} that the spectrum of $D^2_{\epsilon,\tau h}$ is discrete.

Let $K_{\epsilon, \tau}$ denote the heat kernel of $e^{-t D^2_{\epsilon, \tau h}}$. Then \cite[Theorem 1.1]{DY22} shows that, for $t \in (0, 1]$ and $ \tau \in (0, \sqrt{t}]$, we have the following asymptotic expansion for $K_{\epsilon, \tau}$ along the diagonal:
\begin{align*}
    K_{\epsilon, \tau}(t, p, p) &\sim \frac{1}{(4 \pi t)^{\frac{n+k}{2}}} e^{-t\tau^2 \abs{y}^2} \sum_{j = 0}^\infty t^j A_{\epsilon, \tau, j}(p),
\end{align*}
where $A_{\epsilon, \tau, j}$'s are smooth sections of $\End (\Lambda^* E)$. Also, for $r$ large enough we have the following remainder estimate:
\[  \Big| K_{\epsilon, \tau} - \frac{1}{(4 \pi t)^{\frac{n+k}{2}}} e^{-t\tau^2 \abs{y}^2} \sum_{j = 0}^r t^j A_{\epsilon, \tau, j} \Big| \le C t^{\frac{r + 1}{3} - n-k} \tau^{\frac{-2r + 4}{3}}  e^{-a \tilde{d}_\tau(t, y)}, \]
where $\tilde{d}_\tau$ is a parabolic distance satisfying $\tilde{d}_\tau(t, y) \ge \min \{ c_1 \tau \abs{y}, c_2 \tau^2 \abs{y}^2\}$.

From section \ref{subsection:Clifford multiplication} we know that each $A_{\epsilon, \tau, j}$ can be written as 
\begin{equation}
    A_{\epsilon, \tau, j} = \sum_{\substack{I, I' \subset \{1, 2, \cdots, n\} \\ J, J' \subset \{1, 2, \cdots, k\}} } A_{\epsilon, \tau, j}^{I, J, I', J'} \tilde{c}(e^I f^J) \tilde{\hat{c}}(e^{I'} f^{J'}),
\end{equation}
and $A_{\epsilon, \tau, j}^{I, J, I', J'}$'s are algebraic combinations of derivatives of the curvature of $g_\epsilon$ and the function $h$. In particular, we have
\[ A_{\epsilon, \tau, j}^{I, J, I', J'} = \sum_{0 \le l \le j}\sum_{\abs{\alpha_1} + \cdots + \abs{\alpha_l} \le 2j} c^{I, J, I', J'}_{\alpha_1, \cdots, \alpha_l}(\epsilon, \tau) \partial^{\alpha_1}h \cdots \partial^{\alpha_l}h.\]

Let $N$ be the number operator, namely $N: \Lambda^p (E) \to \Lambda^p(E), N(\omega) = p \, \omega$. Then clearly $N K_{\epsilon, \tau}$ is the heat kernel of $N e^{-t D^2_{\epsilon, \tau h}}$. Let $V(y) = \abs{\nabla h}^2$. Although as a homogeneous function, $h$ is degenerate (i.e. $0$ is not the only critical point), one can check directly that $t V(y)  = V(t^{1/2} y)$. So by a similar argument as in \cite[Section 6]{DY22}, we conclude that $N e^{-t D^2_{\epsilon, \tau h}}$ is of trace class. Thus the McKean-Singer formula implies 
\[ \tr_s\Big(N e^{-t D^2_{\epsilon, \tau h}}\Big) = \int_E \tr_s\Big(N K_{\epsilon, \tau}(t, p, p)\Big)\d\mathrm{vol}_{g_\epsilon} .\]
Combining this with the asymptotic expansion of the heat kernel $K_{\epsilon, \tau}$ gives
\begin{align*}
    &\tr_s\Big(N e^{-t D^2_{\epsilon, \tau h}}\Big) \sim \\
    &\frac{1}{(4 \pi t)^{\frac{n+k}{2}}}  \sum_{j = 0}^\infty \sum_{I, J, I', J'} t^j \Big(\int_E e^{-t\tau^2 \abs{y}^2} A^{I, J, I', J'}_{\epsilon, \tau, j}(p) \d\mathrm{vol}_{g_\epsilon}\Big) \tr_s\Big(N c(\tilde{e}^I f^J) c(\tilde{e}^{I'} f^{J'})\Big).
\end{align*}

The supertrace terms are always real numbers, so it remains to evaluate the integrals above. First, recall that each $\tr A^{I, J}_{\epsilon, \tau, j}$ is a linear combination of $\partial^{\alpha_1}h \cdots \partial^{\alpha_l}h$'s with $l \le j$ and $\sum \abs{\alpha_i} \le 2j$. One can check directly that the factors satisfy  $(\partial^{\alpha_i}h)(y) = t^{\frac{\abs{\alpha_i}}{2}-1}(\partial^{\alpha_i}h)(t^{1/2}y) $. Then under the substitution $\tilde{y} = t^{1/2} y$, we have 
\begin{align*}
    \int_{R^k}e^{-t \tau^2 \abs{y}^2} &\partial^{\alpha_1}h \cdots \partial^{\alpha_l}h \d y \\
    &= \int_{R^k}e^{-t \tau^2 \abs{y}^2} t^{\frac{\abs{\alpha_1} + \cdots + \abs{\alpha_l}}{2}-l} \partial^{\alpha_1}h(t^{1/2}y) \cdots \partial^{\alpha_l}h(t^{1/2}y) \d y  \\
    &= \int_{R^k}e^{-\tau^2 \abs{\tilde{y}}^2} t^{\frac{\abs{\alpha_1} + \cdots + \abs{\alpha_l}}{2}-l} \partial^{\alpha_1}h(\tilde{y}) \cdots \partial^{\alpha_l}h(\tilde{y}) t^{-k/2} \d {\tilde{y}} \\
    &= C_{\alpha_1, \cdots, \alpha_l}(\tau) t^{\frac{\abs{\alpha_1} + \cdots + \abs{\alpha_l}}{2}-l-k/2}.
\end{align*}
Therefore, by Fubini's theorem and the assumption that $M$ is compact, we obtain 
\begin{align*}
    &\int_E e^{-t\tau^2 \abs{y}^2} A^{I, J, I', J'}_{\epsilon, \tau, j}(p) \d vol_{g_\epsilon} \\
    &= \int_M \Big( \int_{E_x}e^{-t\tau^2 \abs{y}^2} A^{I, J, I', J'}_{\epsilon, \tau, j}(x, y) \d y \Big) \epsilon^{-\frac{n}{2}} \d vol_x \\
    &= \sum_{0 \le l \le j}\sum_{\alpha_1, \cdots, \alpha_l} \int_M  c^{I, J, I', J'}_{\alpha_1, \cdots, \alpha_l}(\epsilon, \tau) \Big(\int_{R^k}e^{-t \tau^2 \abs{y}^2} \partial^{\alpha_1}h \cdots \partial^{\alpha_l}h \d y\Big) \epsilon^{-\frac{n}{2}}\d vol_x \\
    &= C^{I, J, I', J'}_{j}(\epsilon, \tau) \sum_{0 \le l \le j} \sum_{\abs{\alpha_1} + \cdots + \abs{\alpha_l} \le 2j} t^{\frac{\abs{\alpha_1} + \cdots + \abs{\alpha_l}}{2} -l - \frac{k}{2}}.
\end{align*}

This leads to the following asymptotic expansion for the supertrace: 
\begin{theorem} \label{thm:tr_s_asymp_expan}
    For any $\tau \in (0, \sqrt{t}]$, we have 
    \begin{equation*}
        \tr_s\Big(N e^{-t D^2_{\epsilon, \tau h}}\Big) \sim \frac{1}{(4 \pi t)^{\frac{n+k}{2}}} \sum_{j = 0}^\infty \sum_{l \le j} \sum_{\sum\abs{\alpha_i} \le 2j} C_{\alpha_1, \cdots, \alpha_l}(\epsilon, \tau) t^{j + \frac{\sum \abs{\alpha_i} }{2} -l - \frac{k}{2}}
    \end{equation*}
    as $t \to 0$. Moreover, for $r$ large enough, we have the corresponding remainder estimate:
    \begin{align*}
        &\left|\tr_s\Big(N e^{-t D^2_{\epsilon, \tau h}}\Big) - \frac{1}{(4 \pi t)^{\frac{n+k}{2}}} \sum_{j = 0}^r \sum_{l \le j} \sum_{\sum\abs{\alpha_i} \le 2j} C_{\alpha_1, \cdots, \alpha_l}(\epsilon, \tau) t^{j + \frac{\sum \abs{\alpha_i} }{2} -l - \frac{k}{2}}\right| \\
        &\le C(\epsilon, \tau, r) t^{\frac{r+3}{1} -3n - 3k}.
    \end{align*}
\end{theorem}

\begin{remark} \label{rmk:simple_asymp_expan}
    Obviousely $j + \frac{\sum\abs{a_i}}{2} - l \in [0, 3j]$. The asymptotic expansion above can be written as 
    \[ \tr_s\Big(N e^{-t D^2_{\epsilon, \tau h}}\Big) \sim \frac{1}{(4 \pi t)^{\frac{n+k}{2}}} \sum_{p = -k}^\infty a_p t^{\frac{p}{2}}. \]
    In particular, the leading term is $a_{-k}t^{-\frac{n+2k}{2}}$. Moreover, for $r$ sufficiently large, the remainder estimate implies 
    \[ \left|\tr_s\Big(N e^{-t D^2_{\epsilon, \tau h}}\Big) - \frac{1}{(4 \pi t)^{\frac{n+k}{2}}} \sum_{p = -k}^r a_p t^{\frac{p}{2}}\right| \le C t. \]
\end{remark}

As an upshot of this, we prove the main theorem in this section.
\begin{proof}[Proof of Theorem \ref{thm:main1_anal_tor}]
    Let $D^2_{q}$ denote the restriction of $D^2_{\epsilon, \tau h}$ on $\Lambda^q E$. Also, let $\Sigma_0 =  \frac{1}{(4 \pi t)^{\frac{n+k}{2}}} \sum_{p = -k}^r a_p t^{\frac{p}{2}}$ be a partial sum of the series in the asymptotic expansion, with $r$ large enough such that Remark \ref{rmk:simple_asymp_expan} applies. 
    
    We consider two intervals $[0, 1]$ and $[1, +\infty]$ separately. We first deal with the second interval. Let $\lambda_0(\epsilon)$ be the smallest nonzero eigenvalue of $N e^{-t D^2_{\epsilon, \tau h}}$. Then it is easy to see 
    \[ \abs{\tr_s (N e^{-t \underline{D^2}_{\epsilon, \tau h}})} \le C e^{-\frac{\lambda_0(\epsilon)}{2}t} \]
    for some constant $C$. Therefore integrating over the second interval yields a holomorphic function on $\C$.
    
    For the first interval, it follows from the definition that 
    \begin{align*}
        \int_0^1 &t^{s-1} \tr_s(N e^{-t \underline{D^2}_{\epsilon, \tau h}}) \d t \\
        = &\int_0^1 t^{s-1} \left(\tr_s( N e^{-t D^2_{\epsilon, \tau h}}) - \Sigma_0\right) \d t + \int_0^1 t^{s-1} \Sigma_0 \d t \\
        & - \sum_{q = 1}^{n + k} (-1)^q q \int_0^1 t^{s-1}  \dim(\ker D^2_{q})  \d t.
    \end{align*}
    By Remark \ref{rmk:simple_asymp_expan}, the first integral above is holomorphic provided that $\mathrm{Re}(s) > - 1$. Next, notice that each term in 
    \[\int_0^1 t^{s-1} \Sigma_0 \d t = \sum_{p = -k}^{r} \int_0^1 t^{s-1} \frac{1}{(4 \pi t)^{\frac{n+k}{2}}}  a_p \, t^\frac{p}{2} \d t\]
    creates a single pole at $s = -\frac{k}{2}, -\frac{k+1}{2}, \cdots$. Also, we see that 
    \begin{align*}
        - \sum_{q = 1}^{n + k} (-1)^q q \int_0^1 t^{s-1}  \dim(\ker D^2_{q}) \d t = - \frac{1}{s} \sum_{q = 1}^{n + k} (-1)^q q \dim(\ker D^2_{q})
    \end{align*}
    has a simple pole at $s = 0$. Thus the proof is completed by noticing that $\frac{1}{\Gamma(s)}$ is an entire function with a simple zero at $s = 0$.
\end{proof}

\section{Several intermediate results} \label{section:inter_result}

For reader's convenience, in this section we collect some technical results which will be crucial in proving the main theorem. Some of the theorems will be proved in later sections. In particular, the proof of Theorem \ref{thm:interm_result_1} will be postponed to section \ref{subsect:noncpt_fiber}, Theorem \ref{thm:a=b} to section \ref{subsect:a=b}, Theorem \ref{thm:interm_result_I_4_1} to section \ref{subsect:proof_int_res_I_4_1}, Theorem \ref{thm:interm_result_I_4_2} to section \ref{subsect:proof_int_res_I_4_2}, and Theorem \ref{thm:interm_result_I_4_3} to section \ref{sect:proof_last_thm}.

Recall that we equip $E$ with the metric $g_\epsilon = \frac{1}{\epsilon^2} g_M + g_Y$.
For $a>0$, we use $P_\epsilon^{(0,a]}$ to denote the orthogonal projection from $L^2(\Lambda^*E)$ onto the eigenspaces $\oplus_{0 < \abs{\lambda'_\epsilon} \le a}E(\lambda'_\epsilon)$ of $\frac{1}{\epsilon^2}D_{\epsilon, \tau h}^2$ corresponding to small eigenvalues. We define $P_\epsilon^{[0, a]}$ and $P_\epsilon^{(a, \infty)}$ accordingly.

\begin{theorem} 
\label{thm:interm_result_spec}
    \begin{enumerate}
        \item There exists $\lambda'_0>0$ such that for any $\epsilon>0$ sufficiently small, we have the following uniform bound for non-zero eigenvalues $\lambda'_\epsilon$ of $\frac{1}{\epsilon}D_{\epsilon, \tau h}$: 
        \[ \abs{\lambda'_\epsilon} \ge \lambda'_0, ~~~~ \forall \lambda'_\epsilon \ne 0 \in \spec\left(\frac{1}{\epsilon}D_{\epsilon, \tau h}\right). \]
        \item Let $a = \frac{\lambda'_0}{2}$. Then $\lim_{\epsilon \to 0} \tr_s \left(\frac{1}{\epsilon^2}D_{\epsilon, \tau h}^{2}P_\epsilon^{[0,a]}\right) = 0$.
    \end{enumerate}
\end{theorem}
\begin{proof}
    We will follow \cite[section 4.2]{Dai91} and prove (1) by comparing two spectral sequences. For references of spectral sequences, we refer to \cite[section 4.2]{Dai91}, \cite[section 1]{BB1994} as well as \cite[chapter 3.5]{GHbook}.
        
    It is clear that every $\lambda'_\epsilon$ corresponds to a non-zero eigenvalue $\lambda_\epsilon = \epsilon \lambda'_\epsilon$ of $D_{\epsilon, \tau h}$. To prove the uniform lower bound, it suffices to consider $\lambda_\epsilon$ which decays to $0$ as $\epsilon \to 0$. Let $G_{\Lambda^0_r}$ be the direct sum of eigenspaces of $D_{\epsilon, \tau h}$, such that the corresponding eigenvalues decay at least like $\epsilon^r$, i.e. 
    \[ G_{\Lambda^0_r} = \bigoplus_{\lambda_\epsilon = o(\epsilon^r)} E(\lambda_\epsilon). \]
    Define $E_r := \lim_{\epsilon \to 0} G_{\Lambda^0_r}$. Then it is easy to check that $(E_r, \epsilon^{-r}d_{\tau h})$ forms a spectral sequence. A direct observation is that the sequence terminates at $E_k$, i.e. $E_{k-1} \ne E_k = E_{k+1} = \cdots = E_\infty$, if and only if every $\lambda_\epsilon \ne 0$ decays at most like $\epsilon^{k-1}$. 
        
    Next, we consider Serre's filtration 
    \begin{align*}
        F^i :&= \left\{ a(z, y) \d x^\alpha \wedge \d y^\beta \mid \abs{\alpha} \ge i \right\} \\
        &= \bigoplus_{p \ge i} \Lambda^pM \otimes \Lambda^*Y.
    \end{align*}
    This induces the Leray spectral sequence $(\bar{E}_{r}, \bar{d}_{r, \tau h})$. It is well known that 
    \begin{equation}
        \begin{split}
            \bar{E}_1 &\cong \Lambda^*M \otimes \ker(\tilde{D}_Y|_{\Lambda^*Y}) \cong \Lambda^*M \otimes \mathrm{span}_\R\left\{e^{-\tau \frac{\abs{y}^2}{2}}\right\} \\
            \bar{E}_2 &\cong H^*(\bar{E}_1) \cong H_{dR}^*M \otimes \mathrm{span}_\R\left\{e^{-\tau \frac{\abs{y}^2}{2}}\right\} \\
            \bar{E}_\infty &\cong \ker (D_{\epsilon, \tau h}).
        \end{split}
    \end{equation} 
    By Theorem 1.3 in \cite{DY23}, for $\epsilon>0$ small enough, we have 
    \begin{equation}
        \ker (D_{\epsilon, \tau h}) \cong H_{dR}^*(E, U_c) = H_{dR}^*(E) \cong H^*_{dR}(M),
    \end{equation}
    where the second equality comes from the fact that $U_c := \{p \in E \mid h(p) < -c < 0\}$ is an empty set, and the last equality holds true because de Rham cohomology groups are homotopy invariant. As a consequence, we see that $\bar{E}_2 \cong \bar{E}_\infty$, and therefore the Leray spectral sequence terminates at $\bar{E}_2$. 

    Now notice that the proof of Theorem 0.2 in \cite{Dai91} extends to our situation trivially. In particular, we have $E_r \cong \bar{E}_r$. As a result, every $\lambda_\epsilon$ has exactly linear decay. Thus we can find a constant $\lambda'_0 > 0$ such that 
    \[ \abs{\lambda'_\epsilon} = \frac{1}{\epsilon} \abs{\lambda_\epsilon} \ge \lambda'_0.  \]
    This proves item (1). Since (2) is a direct consequence of (1), the proof is completed.
\end{proof}

The next three theorems describe the large time behavior of the supertrace and the $L^2$-norm on the determinant line bundle. 

\begin{theorem} \label{thm:interm_result_1}
    For any $t> 0$ fixed, 
    \begin{enumerate}
        \item $\lim_{\epsilon \to 0} \tr_s \left(N e^{-\frac{t}{\epsilon^2}D_{\epsilon, \tau h}^2}\right) = \tr_s \left(N_M e^{-t D_M^2}\right)$.
        \item $\left|\tr_s \left(N_Y e^{-\frac{t}{\epsilon^2}D_{\epsilon, \tau h}^2}\right)\right| < C\epsilon$.
        \item $\left|\tr_s \left( e^{-\frac{t_0}{\epsilon^2}D_{\epsilon, \tau h}^2}\right)\right| < C$, for any $t_0 \ge t$ and $0<\epsilon<1$.
    \end{enumerate}
\end{theorem}
\begin{proof}
     See section \ref{subsect:noncpt_fiber}.
\end{proof}

\begin{definition}
    We define the secondary Euler characteristics of $M$ and $E$ to be 
    \begin{equation}
        \begin{split}
            \chi_2(M) := \sum_{p = 0}^n (-1)^p p b_p(M),\\
            \chi_2(E) := \sum_{p = 0}^{n+k} (-1)^p p b_p(E),
        \end{split}
    \end{equation}
    where $b_p(M)$ and $b_p(E)$ denote the $p$-th betti numbers of $M$ and $E$ respectively.
\end{definition}

Since the total space $E$ is homotopically equivalent to the image of the zero section, which is diffeomorphic to $M$, we see that 
\begin{equation} \label{eq:chi_2}
    b_p(E) = b_p(M), ~~~ \chi_2(E) = \chi_2(M).
\end{equation}

\begin{theorem} \label{thm:interm_result_2}
     Let $a$ be the constant given in Theorem \ref{thm:interm_result_spec}. For any $t > 0$ fixed, 
    \begin{enumerate}
        \item $\lim_{\epsilon \to 0} \tr_s \left(N e^{-\frac{t}{\epsilon^2}D_{\epsilon, \tau h}^2}P_\epsilon^{(a, \infty)}\right) = \tr_s \left(N_M e^{-t \underline{D}_M^2}\right)$.
        \item $\left|\tr_s \left(N e^{-\frac{t}{\epsilon^2}D_{\epsilon, \tau h}^2}P_\epsilon^{(a, \infty)}\right)\right| < Ce^{-C t}$ .
    \end{enumerate}
    Here $\underline{D}_M^2$ denotes the Hodge Laplacian with the null space removed.
\end{theorem}
\begin{proof}
    It follows from the proof of Theorem \ref{thm:interm_result_spec} that for $\epsilon$ small enough, the only eigenvalue of $\frac{1}{\epsilon^2}D_{\epsilon, \tau h}^2$ that lies in $[0,a]$ is $0$. From this and (\ref{eq:chi_2}), we deduce 
    \begin{align*}
        \tr_s \left(N e^{-\frac{t}{\epsilon^2}D_{\epsilon, \tau h}^2}P_\epsilon^{(a, \infty)}\right) &= \tr_s \left(N e^{-\frac{t}{\epsilon^2}D_{\epsilon, \tau h}^2}\right) - \sum_{p = 0}^{n+k} (-1)^p p \dim \ker D_{\epsilon, \tau h}^2 \\
        &= \tr_s \left(N e^{-\frac{t}{\epsilon^2}D_{\epsilon, \tau h}^2}\right) - \chi_2(E) \\
        &= \tr_s \left(N e^{-\frac{t}{\epsilon^2}D_{\epsilon, \tau h}^2}\right) - \chi_2(M).
    \end{align*}
    Therefore (1) follows from a similar argument as Theorem \ref{thm:interm_result_1}.

    Since $a$ is the uniform lower bound of all non-zero eigenvalues of $\frac{1}{\epsilon^2}D_{\epsilon, \tau h}^2$, 
    we obtain the exponential decay in (2).
\end{proof}

Using \cite{DY23},  more specifically arguing as in the proof of \cite[Theorem 4.1]{DY23}, 
one can show that the Witten deformed $L^2$-cohomology of $E$ is isomorphic to the (absolute) de Rham cohomology: 
\begin{equation} \label{eq:cohom_isom}
    H_{(2)}^* (E, d_{\tau h}) \cong H^*_{dR}(E) \cong H^*_{dR}(M).
\end{equation} 
Hence we have the following isomorphisms between determinant line bundles: 
\begin{equation}
    (\det E_1)^{-1} \cong \left[\det (\ker \tilde{D}_{T, \tau h}^2)\right]^{-1} \cong (\det E_2)^{-1} = (\det E_\infty)^{-1} := (\det E)^{-1}.
\end{equation}

Before stating the rest of the results, we need to introduce a new metric which is essentially conformal to $g_\epsilon$.
\begin{definition}
    For any $T > 0$, we define a metric on $E$ by 
    \[ \tilde{g}_{T} = g_M + \frac{1}{T^2}g_Y. \]
    Let $\tilde{D}_{T}$ be the corresponding Dirac operator, and 
    \[ \tilde{D}_{T, \tau h} := \tilde{D}_T + \tau \hat{c}(\tilde{\nabla}^T h) \]
    be the Witten-Dirac operator, where $\tilde{\nabla}^L$ is the Levi-Civita connection of $\tilde{g}_T$.
\end{definition}

Let $\abs{~}_T$ denote the $L^2$-norm on $(\det E)^{-1}$ induced by $\tilde{g}_T$. When $T=1$, we write $\abs{~}:= \abs{~}_1$ for convenience. On the other hand, the metric $\tilde{g}_1 = g_1$ induces an $L^2$-norm $\leftindex_1{\abs{~}}$ on $(\det E_1)^{-1} \cong (\det E)^{-1}$, and hence a norm $\leftindex_\infty{\abs{~}}$ on $(\det E_\infty)^{-1} \cong (\det E)^{-1}$. 

\begin{theorem}  \label{thm:interm_result_L^2_norm}
    \[ \lim_{T \to \infty } \log \left(\frac{\abs{~}_{T}}{\abs{~}}\right) = \log \left(\frac{\leftindex_\infty{\abs{~}}}{\abs{~}}\right) .\]
\end{theorem}
\begin{proof}
    See \cite[section 6(e)]{BB1994} for the proof.
\end{proof}

In the end, we discuss the contribution of supertrace when the time is small. Recall that by assumption, $M$ is compact and odd-dimensional. It follows from \cite[Theorem 5.1]{DM} that 
\begin{equation} \label{equ:tr_s_expantion_on_M}
    \tr_s \left(N_M e^{-t D_M^2}\right) \sim b_{-1/2}t^{-\frac 12} + \sum_{j \ge 1, odd} b_{j/2} t^{\frac{j}{2}},
\end{equation}
and
\begin{equation} \label{equ:b_-1/2}
    b_{-1/2} e^1 \wedge \cdots \wedge e^n =  2\i (-1)^{\frac{n+1}{2}} (16\pi)^{-\frac{n}{2}} \sum_{k=1}^n (-1)^k \mathrm{Pf}(R_k) \wedge e^k,
\end{equation}
where $R_k = \left( \sum_{p, q \ne k} R_{ijkl} e^p \wedge e^k \right)_{i, j \ne k}$ is the matrix $2$-form obtained by deleting the $k$-th row and column from the curvature matrix $2$-form. 

\begin{theorem} \label{thm:a=b}
    The asymptotic expansion of $\tr_s \left(N e^{-t D^2_{\tau h}}\right)$ as $t \to 0$ starts with $a_{-1/2} t^{-\frac{1}{2}}$ and does not contain a constant term. Moreover, for the leading coefficient we have 
    \begin{equation}
        a_{-1/2} = b_{-1/2}.
    \end{equation}
\end{theorem}
\begin{proof}
    See section \ref{subsect:a=b}.
\end{proof}

\begin{theorem} \label{thm:interm_result_I_4_1}
    For any $T >0$ fixed, 
    \[ \lim_{\sigma \to 0}\tr_s \left(N_Y e^{ - \sigma^2 \tilde{D}_{T/\sigma,\tau h}^2}\right) = 0. \]
\end{theorem}
\begin{proof}
    See section \ref{subsect:proof_int_res_I_4_1}.
\end{proof}

\begin{theorem}  \label{thm:interm_result_I_4_2}
    There exists a constant $C > 0$ such that for any $\sigma\in (0, 1]$ and $T \in [\sigma, 1]$,
    \[ \left|\frac{1}{T}\tr_s\left(N_Y e^{-\sigma^2\tilde{D}_{T/\sigma, \tau h}^2}\right) \right| \le C.\]
\end{theorem}
\begin{proof}
    See section \ref{subsect:proof_int_res_I_4_2}.
\end{proof}

\begin{theorem}  \label{thm:interm_result_I_4_3}
    There exist $\delta \in (0,1]$ and $C>0$ such that for any $\sigma \in (0,1]$ and $T\in [1, \infty)$, 
    \[ \left|\frac{1}{T}\tr_s\left(N_Y e^{-\sigma^2 \tilde{D}_{T/\sigma, \tau h}^2}\right)\right| \le \frac{C}{T^{1+\delta}} .\]
\end{theorem}
\begin{proof}
    See section \ref{sect:proof_last_thm}.
\end{proof}

\section{Proof of main theorem} \label{section:proof_main}
We will proceed as in \cite[Section 4]{BB1994} to prove Theorem \ref{thm:main2_limit}. For convenience, let $T = \frac{1}{\epsilon}$.

\begin{definition}
\begin{enumerate}
    \item For any $t, T > 0$, we define a Riemannian metric $\tilde{g}_{t,T}$ on $E$ by 
    \begin{align*}
        \tilde{g}_{t, T} = \frac{1}{t^2} \tilde{g}_T =  \frac{1}{t^2} \left(g_M + \frac{1}{T^2}g_Y \right).
    \end{align*}
    Let $\tilde{*}_{t, T}$ and $\tilde{D}_{t,T}$ denote the corresponding Hodge star operator and Dirac operator. Moreover, let 
    \begin{align*}
        \tilde{D}_{t, T, \tau h} := \tilde{D}_{t, T} + \tau \hat{c}(\nabla^{t, T} h) 
    \end{align*}
    denote the associated Witten-Dirac operators.
    \item Let $\alpha_{t, T, \tau h}$ be a $1$-form  on $\R_+^* \times \R_+$ defined by  
    \begin{equation*}
        \alpha_{t, T} = \d t \tr_s\left[\tilde{*}_{t, T}^{-1} \partial_t\tilde{*}_{t, T} e^{-\tilde{D}_{t, T, \tau h}^2}\right] + \d T \tr_s\left[\tilde{*}_{t, T}^{-1} \partial_T\tilde{*}_{t, T} e^{-\tilde{D}_{t, T, \tau h}^2}\right].
    \end{equation*}
\end{enumerate}
\end{definition}

\begin{remark} \label{rmk:metrics}
    Under the identification $T = \frac{1}{\epsilon}$, we have the following relations between metrics:
    \begin{align*}
        \tilde{g}_{T} = \epsilon^2 g_\epsilon, ~~~  \tilde{g}_{t, T} = \left(\frac{\epsilon}{t}\right)^2 g_\epsilon.
    \end{align*}
\end{remark}

\begin{proposition} \label{prop:alpha}
    \begin{enumerate}
        \item $\alpha_{t, T}$ is closed.
        \item $\tilde{*}_{t, T}^{-1} \partial_t\tilde{*}_{t, T} = \frac{1}{t}\left(2N - n-k\right)$, $\tilde{*}_{t, T}^{-1} \partial_T\tilde{*}_{t, T} = \frac{1}{T}\left(2N_Y-k\right)$. Therefore 
        \begin{equation} \label{eq:alpha_new}
            \alpha_{t, T} = \frac{2\d t}{t} \tr_s\left[N e^{-\tilde{D}_{t, T, \tau h}^2}\right] + \frac{2\d T}{T} \tr_s\left[N_Y e^{-\tilde{D}_{t, T, \tau h}^2}\right].
        \end{equation}
        \item If $\tilde{g} = \frac{1}{\sigma^2} g$, then $\sigma^N \tilde{D} \sigma^{-N} = \sigma D$, and 
        \begin{align*}
            \tr_s\left(N e^{-\tilde{D}^2}\right) = \tr_s\left(N e^{-\sigma^2D^2}\right).
        \end{align*}
        In particular, 
        \begin{equation*} 
            \tr_s\left(N e^{-\tilde{D}_{t, T, \tau h}^2}\right) = \tr_s\left(N e^{-\left(\frac{t}{\epsilon}\right)^2D_{\epsilon, \tau h}^2}\right).
        \end{equation*}
    \end{enumerate}
\end{proposition}
\begin{proof}
    \begin{enumerate}
        \item It follows from \cite[Theorem 4.3]{BB1994}. 
        \item Recall that $\{e^1,\cdots,e^n,f^1,\cdots,f^k\}$ denotes an orthonormal coframe of $g = g_1$. Let $\{\bar{e}^1,\cdots,\bar{e}^n,\bar{f}^1,\cdots,\bar{f}^k\}$ be the corresponding orthonormal basis with respect to $g_{t, T}$. Then clearly we have $\bar{e}^i = \frac{1}{t}e^i$ and $\bar{f}^j = \frac{1}{tT}f^i$. Now if 
        \[\tilde{*}_{t, T}\left(\bar{e}^{i_1} \cdots \bar{e}^{i_p} \wedge \bar{f}^{j_1} \cdots \bar{f}^{j_q}\right) = \bar{e}^{i_{p+1}} \cdots  \bar{e}^{i_n} \wedge \bar{f}^{j_{q+1}}  \cdots \bar{f}^{j_k}, \]
        then equavilently we have 
        \begin{align*}
            \tilde{*}_{t, T} &\left(e^{i_1} \cdots e^{i_p} f^{j_1} \cdots f^{j_q}\right) \\
            &= t^{2(p+q) - (n+k)} T^{2q - k} e^{i_{p+1}} \cdots e^{i_n} f^{j_{q+1}} \cdots f^{j_k}.
        \end{align*}
        This proves the first line of (2). (\ref{eq:alpha_new}) follows from that fact that the dimension of $E$ is odd and thus 
        \[ \tr_s \left(e^{-\tilde{D}^2_{t, T, \tau h}}\right) = \chi(E, d_{\tau h}) = 0. \]
        
        \item The first part follows from a straightworward calculation. Using Remark \ref{rmk:metrics} we obtain the last line.
    \end{enumerate}
\end{proof}

Consider the following loop $\Gamma=\Gamma_1 \cup \Gamma_2 \cup \Gamma_3 \cup \Gamma_4$ in $\R^2$ with counterclockwise orientation, where 
\begin{align*}
    \Gamma_1 &= \{(T_0, t) \mid \sigma \le t \le A\}, \\
    \Gamma_2 &= \{(T, A) \mid 1 \le T \le T_0\}, \\
    \Gamma_3 &= \{(1, t) \mid \sigma \le t \le A\}, \\
    \Gamma_4 &= \{(T, \sigma) \mid 1 \le T \le T_0\}. 
\end{align*}

Denote $I_k^0 := \int_{\Gamma_k} \alpha_{t, T}$ for $k = 1, 2, 3, 4$. Then Proposition \ref{prop:alpha} implies that 
\begin{equation}
    I_1^0 + I_2^0 + I_3^0 + I_4^0 = 0.
\end{equation}

The main theorem will then follow from letting $A \to \infty$, $T_0 \to \infty$ (or equivalently $\epsilon_0 \to 0$) and $\sigma \to 0$ sequentially for each $I_k^0$ (subtracting divergent terms if necessary), and comparing the limits.

\subsection{$\mathbf{I_1^0}$} \label{sect:I_1}
Observe that 
\begin{align*}
    I_1^0 &= 2 \int_\sigma^A \tr_s\left(N e^{-\tilde{D}_{t, T}^2}\right) \frac{\d t}{t} \\
    &= 2 \int_\sigma^A \tr_s\left(N e^{-\left(\frac{t}{\epsilon_0}\right)^2D_{\epsilon_0, \tau h}^2}\right) \frac{\d t}{t} \\
    &= \int_{\sigma^2}^{A^2} \tr_s\left(N e^{-\frac{\tilde{t}}{\epsilon_0^2}D_{\epsilon_0, \tau h}^2}\right) \frac{\d \tilde{t}}{\tilde{t}},
\end{align*}
where in the last equality we made a substitution $\tilde{t} = t^2$. We will still write the last line as $\int_{\sigma^2}^{A^2} \tr_s\left(N e^{-\frac{t}{\epsilon_0^2}D_{\epsilon_0, \tau h}^2}\right) \frac{\d t}{t}$ when there is no confusion of notation.
\begin{enumerate}
    \item $\underline{A \to \infty}$: To cancel the term which diverges at infinity, we subtract 
\end{enumerate}
a term $2 \chi_2(E) \log A$ to obtain
\begin{equation}
    \begin{split}
        &I^0_1 - 2 \chi_2(E) \log A \to \\
        &I^1_1 := \int_{\sigma^2}^1 \tr_s\left(N e^{-\frac{t}{\epsilon_0^2}D_{\epsilon_0, \tau h}^2}\right) \frac{\d t}{t} + \int_1^\infty \tr_s\left(N e^{-\frac{t}{\epsilon_0^2}\underline{D}_{\epsilon_0, \tau h}^2}\right) \frac{\d t}{t}.
    \end{split}
\end{equation}
\begin{enumerate}
    \item[(2)] $\underline{\epsilon_0 \to 0}$: Theorem \ref{thm:interm_result_1} shows that 
\end{enumerate}
    \begin{equation*}
    \int_{\sigma^2}^1 \tr_s\left(N e^{-\frac{t}{\epsilon_0^2}D_{\epsilon_0, \tau h}^2}\right) \frac{\d t}{t} \to \int_{\sigma^2}^1 \tr_s\left(N_M e^{-tD_M^2}\right) \frac{\d t}{t}. 
\end{equation*}
On the other hand, 
\begin{equation*}
    \begin{split}
        &\int_1^\infty \tr_s\left(N e^{-\frac{t}{\epsilon_0^2}\underline{D}_{\epsilon, \tau h}^2}\right) \frac{\d t}{t} \\
        &= \int_1^\infty \tr_s\left(N e^{-\frac{t}{\epsilon_0^2}D_{\epsilon_0, \tau h}^2}P_{\epsilon_0}^{(a, \infty)}\right) \frac{\d t}{t} + \int_1^\infty \tr_s\left(N e^{-\frac{t}{\epsilon_0^2}D_{\epsilon_0, \tau h}^2}P_{\epsilon_0}^{(0, a]}\right) \frac{\d t}{t}.
    \end{split}
\end{equation*}

For the first term, it follows from Theorem \ref{thm:interm_result_2} that 
\begin{equation*}
    \int_1^\infty \tr_s\left(N e^{-\frac{t}{\epsilon_0^2}D_{\epsilon_0, \tau h}^2}P_{\epsilon_0}^{(a, \infty)}\right) \frac{\d t}{t} \to \int_1^\infty \tr_s\left(N_M e^{-t\underline{D}_{M}^2}\right) \frac{\d t}{t}
\end{equation*}
For the second term, Theorem \ref{thm:interm_result_spec} implies that $P^{(0,a]}_{\epsilon_0} = 0$ for $\epsilon_0$ small enough. Therefore if we let $u = \frac{t}{\epsilon_0^2}D_{\epsilon_0, \tau h}^2$, then 
\begin{equation}
    \begin{split}
        &\int_1^\infty \tr_s\left(N e^{-\frac{t}{\epsilon_0^2}D_{\epsilon_0, \tau h}^2}P_{\epsilon_0}^{(0,a]}\right) \frac{\d t}{t} = \tr_s\left(N\int_{\frac{1}{\epsilon_0^2}D_{\epsilon_0, \tau h}^2}^\infty e^{-u} \frac{\d u}{u} P_{\epsilon_0}^{(0,a]}\right) \\
        &= \tr_s\left(N\int_{\frac{1}{\epsilon_0^2}D_{\epsilon_0, \tau h}^2}^1 (e^{-u} - 1) \frac{\d u}{u} P_{\epsilon_0}^{(0,a]}\right) - \tr_s\left(N (\log \frac{1}{\epsilon_0^2}D_{\epsilon_0, \tau h}^2) P_{\epsilon_0}^{(0,a]} \right) \\
        &~~~~+ \tr_s\left(N\int_{1}^\infty e^{-u} \frac{\d u}{u} P_{\epsilon_0}^{(0,a]}\right) \\
        &\to 0.
    \end{split}
\end{equation}
Consequently, 
\begin{equation}
    \begin{split}
        I^1_1 \to I^2_1 := \int_{\sigma^2}^1 \tr_s\left(N_M e^{-tD_M^2}\right) \frac{\d t}{t} + \int_1^\infty \tr_s\left(N_M e^{-t\underline{D}_{M}^2}\right) \frac{\d t}{t}.
    \end{split}
\end{equation}
\begin{enumerate}
    \item[(3)] $\underline{\sigma \to 0}$: Recall that for $t$ small, 
\end{enumerate}
\begin{equation*}
    \tr_s\left(N_M e^{-t D_M^2}\right) = b_{-\frac{1}{2}} t^{-\frac{1}{2}} + o(t^{\frac{1}{2}}).
\end{equation*}
A straightforward calculation shows 
\begin{equation}
    \begin{split}
        I^2_1 - 2b_{-\frac{1}{2}} \frac{1}{\sigma} \to I^3_1 := &\int_{0}^1 \left[\tr_s\left(N_M e^{-tD_M^2}\right) - b_{-\frac{1}{2}} t^{-\frac{1}{2}}\right] \frac{\d t}{t} \\
       &  + \int_1^\infty \tr_s\left(N_M e^{-t\underline{D}_{M}^2}\right) \frac{\d t}{t} - 2b_{-\frac{1}{2}}.
    \end{split}
\end{equation}

\subsection{$\mathbf{I_2^0}$} \label{sect:I_2}
Let $P_\epsilon: L^2(\Lambda^*E) \to L^2(\ker D_{\epsilon, \tau h}^2) $ be the orthogonal projection onto the kernel. Notice that 
\begin{align}
    I^0_2 = - \int_1^{T_0} \tr_s\left(\tilde{*}_T^{-1}\partial_T \tilde{*}_T e^{-A^2T^2 D_{\epsilon, \tau h}^2}\right) \frac{\d T}{T}.
\end{align}

\begin{enumerate}
    \item $\underline{A \to \infty}$: Clearly we have 
\end{enumerate}
\begin{equation}
    I^0_2 \to I^1_2 := - \int_1^{T_0} \tr_s\left(\tilde{*}_T^{-1}\partial_T \tilde{*}_T P_\epsilon\right) \frac{\d T}{T}.
\end{equation}
 Then by \cite[Proposition 4.18]{BB1994}, we have 
\begin{equation}
    I^1_2 = \log \left(\frac{\abs{~}_{T_0}}{\abs{~}}\right)^2.
\end{equation}

\begin{enumerate}
    \item[(2)] $\underline{T_0 \to \infty}$: It follows from Theorem \ref{thm:interm_result_L^2_norm} that  
\end{enumerate}
\begin{equation}
    I^1_2  \to I^2_2:= \log \left(\frac{\leftindex_\infty{\abs{~}}}{\abs{~}}\right).
\end{equation}

\begin{enumerate}
    \item[(3)] $\underline{\sigma \to 0}$: $I^3_2 = I^2_2= \log \left(\frac{\leftindex_\infty{\abs{~}}}{\abs{~}}\right)$ .
\end{enumerate}

\subsection{$\mathbf{I_3^0}$} \label{sect:I_3}
Proposition \ref{prop:alpha} implies 
\begin{equation}
    \begin{split}
        I^0_3 &= - 2\int_\sigma^A \tr_s\left(N e^{-\tilde{D}_{t, 1, \tau h}^2}\right) \frac{\d t}{t} \\
        &= - \int_{\sigma^2}^{A^2} \tr_s\left(Ne^{-tD_{\tau h}^2}\right) \frac{\d t}{t}.
    \end{split}
\end{equation}

\begin{enumerate}
    \item $\underline{A \to \infty}$: As before, 
\end{enumerate}
\begin{equation}
    \begin{split}
        &I^0_3 + 2\chi_2(E) \log A \to \\
    &I^1_3:= -\int_{\sigma^2}^1 \tr_s\left(Ne^{-tD_{\tau h}^2}\right) \frac{\d t}{t} - \int_1^\infty \tr_s\left(Ne^{-t\underline{D}_{\tau h}^2}\right) \frac{\d t}{t}
    \end{split}
\end{equation}

\begin{enumerate}
    \item[(2)] $\underline{T_0 \to \infty}$: Since $I^1_3$ is independent of $T_0$, we have $I^2_3 = I^1_3$.
\end{enumerate}

\begin{enumerate}
    \item[(3)] $\underline{\sigma \to 0}$: Recall that by Theorem \ref{thm:a=b},   
\end{enumerate}
\[ \tr_s\left(Ne^{-t D_{\tau h}^2}\right) = a_{-1/2} t^{-\frac{1}{2}} + o(t^{\frac{1}{2}}). \]

Thus 
\begin{equation}
    \begin{split}
        I^2_3 + 2 a_{-1/2} \frac{1}{\sigma} \to 
        I^3_3:= &-\int_{0}^1 \left[\tr_s\left(Ne^{-tD_{\tau h}^2}\right) - a_{-1/2} t^{-\frac{1}{2}} \right]\frac{\d t}{t} \\
        &- \int_1^\infty \tr_s\left(Ne^{-t\underline{D}_{\tau h}^2}\right) \frac{\d t}{t} + 2a_{-1/2}.
    \end{split}
\end{equation}

\subsection{$\mathbf{I_4^0}$}  \label{sect:I_4}
To begin with, 
\begin{equation}
    I^0_4 = 2\int_1^{T_0} \tr_s\left( N_Ye^{-\sigma^2 \tilde{D}_{T,\tau h}^2} \right) \frac{\d T}{T}.
\end{equation}

\begin{enumerate}
    \item $\underline{A \to \infty}$: Obviously $I^1_4 = I^0_4$.
\end{enumerate}

\begin{enumerate}
    \item[(2)] $\underline{T_0 \to \infty}$: It follows from Proposition \ref{thm:interm_result_1} that 
\end{enumerate}
\begin{equation} \label{eq:I_4^2}
    \begin{split}
        I^1_4 \to  
        I^2_4 &:= 2\int_1^\infty \tr_s\left( N_Ye^{-\sigma^2 \tilde{D}_{T,\tau h}^2} \right) \frac{\d T}{T} \\
        &= 2\int_\sigma^{1} \tr_s\left( N_Ye^{-\sigma^2 \tilde{D}_{T/\sigma,\tau h}^2} \right) \frac{\d T}{T} + 2\int_1^\infty \tr_s\left( N_Ye^{-\sigma^2 \tilde{D}_{T/\sigma,\tau h}^2} \right) \frac{\d T}{T},
    \end{split}
\end{equation}
where in the last equality we make a change of variable $\tilde{T} = \sigma T$ and still write the new variable as $T$.

\begin{enumerate}
    \item[(3)] $\underline{\sigma \to 0}$:
For the first term in the last line of (\ref{eq:I_4^2}), using Theorem 
\end{enumerate}
\ref{thm:interm_result_I_4_1}, Theorem \ref{thm:interm_result_I_4_2}, and the dominated convergence theorem, we obtain
\begin{equation}
    2\int_\sigma^{1} \tr_s\left( N_Ye^{-\sigma^2 \tilde{D}_{T/\sigma,\tau h}^2} \right) \frac{\d T}{T} \to 0.
\end{equation}

For the second term, similarly we use Theorem \ref{thm:interm_result_I_4_1}, Theorem \ref{thm:interm_result_I_4_3} and dominated convergence theorem to deduce that 

\begin{equation}
    2\int_1^\infty \tr_s\left( N_Ye^{-\sigma^2 \tilde{D}_{T/\sigma,\tau h}^2} \right) \frac{\d T}{T} \to 0.
\end{equation}

Therefore 
\begin{equation}
    I_4^2 \to I_4^3 = 0.
\end{equation}

\subsection{Matching the divergences}
By Theorem \ref{thm:a=b}, we have the following cancelation of  the divergent terms appearing in section \ref{sect:I_1}-\ref{sect:I_4}:
\begin{equation}
    - 2 \chi'(E) \log A - 2b_{-1/2} \frac{1}{\sigma} + 2 \chi'(E) \log A + 2a_{-1/2} \frac{1}{\sigma} = 0.
\end{equation}

Therefore 
\begin{equation}
    \begin{split}
        \sum_{i=0}^4 I_i^3 = \lim_{\substack{A \to \infty \\ T_0 \to \infty \\ \sigma \to 0}} &\left[ \sum_{i=0}^4 I_i^0 - 2 \chi'(E) \log A - 2b_{-1/2} \frac{1}{\sigma} + 2 \chi'(E) \log A + 2a_{-1/2} \frac{1}{\sigma} \right] \\
        = \lim_{\substack{A \to \infty \\ T_0 \to \infty \\ \sigma \to 0}} &\left[I_1^0 + I_2^0 + I_3^0 + I_4^0\right] = 0.
    \end{split} 
\end{equation}

This completes the proof of Theorem \ref{thm:main2_limit}. \qed

\section{Adiabatic limit of heat operator: large time behavior} \label{sect:large_time}

This section is devoted to the study of the adiabatic limit of the heat operator when the time is large. As a result, we will prove Theorem \ref{thm:interm_result_1}. We start with the compact fiber case and provide a slightly different proof of some results in \cite{Dai91}. By introducing a Morse-Bott function and considering the Witten Laplacian, we show that these results remain true when the fiber is complete and noncompact. As an application, we prove that the index of the total space $E$ coincides with the index of $M$.

In the whole section, we denote by $\norm{~}_0 = \norm{~}_{L^2(\Lambda^*E)}$ the standard $L^2$-norm. If $A: L^2(\Lambda^*E) \to L^2(\Lambda^*E)$ is a bounded linear operator, we use $\norm{~}_{0,0}$ to denote the $L^2$-operator norm. Besides, there is no parity restriction on the dimensions of $E$ and $M$ in this section except in the proof of Theorem \ref{thm:interm_result_1}.

\subsection{Compact fiber case} \label{subsect:cpt_fiber}
Suppose $Y \hookrightarrow E \to M$ is a fibration of closed manifolds, where the fiber $Y$ is a compact manifold. We equip $TE$ with a metric $g_E$ as in (\ref{equa:metric}), and a connection $\nabla^E$ defined by (\ref{equa:conn}).
We further consider a family of metrics with an adiabatic parameter along the horizontal direction:
\begin{equation}
    g_{\epsilon} = \frac{1}{\epsilon^2}g_M + g_Y .
\end{equation}

For any $p \in M$, let $x^1, \cdots, x^n$ be a normal coordinate centered at $p$. Let $\{e_1, \cdots, e_n\}$ be a locally orthonormal frame around $p \in M$ of $TM$ such that $e_i(p) = \frac{\partial}{\partial x^i}(p)$ and $\nabla_{e_i}^L e_j (p) = 0$. Let $\{f_1, \cdots, f_k\}$ be an orthonormal frame for $TY$.
The connection $\nabla^E$ respects the horizontal-vertical decomposition, but may not be orthogonal in general.  To fix this, we consider the following connection
\[ \nabla^u = \nabla^E - \frac 12 \sum_{i=1}^k \inner{S(f_i)f_i, \cdot}. \]



We denote by $\d^{*, \epsilon}$ the formal dual of the exterior differential operator $d$ with respect to the metric $g_{\epsilon}$. Let $D_\epsilon = \d + \d^{*, \epsilon}$ be the corresponding Dirac operator. Then
\begin{equation} \label{eq:D_epsilon}
    \begin{split}
        D_\epsilon &= \epsilon c(\tilde{e}^i) \nabla^u_{e_i} + c(f^j) \nabla^E_{f_j} - \frac{\epsilon^2}{4} c(\tilde{e}^\alpha) c(\tilde{e}^\beta) c(T) (e_\alpha, e_\beta) \\
        & := \epsilon \tilde{D}_M + D_Y - \frac{\epsilon^2}{4} T,
    \end{split}
\end{equation}
where $c(T)$ is the Clifford multiplication of the torsion tensor. We refer to \cite[P. 56]{BC} for the proof.

Although we are eventually interested in the heat operator corresponding to $D_\epsilon$, a detailed analysis of the resolvent of $\frac{1}{\epsilon}D_\epsilon$ is necessary to understand the spectrum of $D_\epsilon$. Formally, since $\frac{1}{\epsilon}D_\epsilon$ is self-adjoint, we have the following integration formula for the heat operator
\begin{equation}
    e^{-\frac{t}{\epsilon^2}D^2_{\epsilon}} = \frac{1}{2\pi \i} \int_\Gamma e^{-t \lambda^2} \left(\lambda - \frac{1}{\epsilon}D_\epsilon\right)^{-1} \d \lambda,
\end{equation}
provided that the integral on the right-hand side is uniformly convergent. Here $\Gamma$ is any contour surrounding the spectrum of $\frac{1}{\epsilon}D_\epsilon$. Given that $\frac{1}{\epsilon}D_\epsilon$ is unbounded, $\Gamma$ should be taken such that $\Gamma$, together with the infinity, surrounds its spectrum.

Now, as in \cite[Section 2.2]{Dai91}, by assuming that $\ker D_Y$ is a vector bundle, we can consider the following decomposition:
\begin{equation} \label{eq:L_2_decomp}
    L^2(\Lambda^* E) = L^2(\ker D_Y)^\perp \oplus L^2(\ker D_Y).
\end{equation}
Let $p$ and $p^\perp$ be the projections onto $L^2(\ker D_Y)$ and $L^2(\ker D_Y)^\perp$, respectively.  Using (\ref{eq:L_2_decomp}), we can write $\frac{1}{\epsilon}D_\epsilon$ as $\frac{1}{\epsilon}D_\epsilon = \begin{pmatrix}
    A_1 & A_2 \\
    A_2^* & A_3
\end{pmatrix}$, where $A_2^*$ is the dual operator of $A_2$, and
\begin{equation}
    \begin{split}
        A_1 &= p^\perp \tilde{D}_M p^\perp + \frac{1}{\epsilon} D_Y - \epsilon p^\perp \frac T4 p^\perp, \\
    A_2 &= p^\perp \tilde{D}_M p - \epsilon p^\perp \frac T4 p, \\
    A_3 &= D_0 - \epsilon p \frac T4 p.
    \end{split}
\end{equation}
Here we denote $D_0 := p \tilde{D}_M p$ for simplicity. 

The following lemma plays a crucial role in proving Theorem \ref{thm:large_time}. The proof essentially follows from {\cite[Proposition 4.41]{BC}}, {\cite[Lemma 2.6]{Dai91}}, so we omit it here.

\begin{lemma}\label{lem:spec}
    ~

    \begin{enumerate}
        \item There exists a positive constant $\lambda_0$ such that $\spec(\abs{A_1}) \subset [\frac{\lambda_0}{\epsilon}, +\infty)$.
        \item The operators $A_2$ and $A_2^*$ are bounded.
    \end{enumerate}
\end{lemma}

For any $\lambda \notin \spec(D_0)$, we consider the operator 
\begin{equation}
    S_\epsilon= \lambda - \frac{1}{\epsilon} D_\epsilon = \begin{pmatrix}
        \lambda - A_1 & -A_2 \\
        -A_2^* & \lambda - A_3 
    \end{pmatrix}.
\end{equation}
It is obvious that $S_\epsilon$ is invertible, and its inverse is given by 
\begin{equation}
    S_\epsilon^{-1} = \begin{pmatrix}
        T_1 & T_2 \\
        T_2^* & T_3
    \end{pmatrix},
\end{equation}
where 
\begin{align*}
    T_1 &= -\left(A_1 - \lambda - A_2 (A_3 - \lambda)^{-1} A_2^*\right)^{-1} \\
    &= (\lambda - A_1)^{-1} \left(I - (A_1 - \lambda)^{-1} A_2 (A_3 - \lambda)^{-1} A_2^*\right)^{-1}, \\
    T_2 &= (A_1 - \lambda)^{-1} A_2 \left(A_3 - \lambda - A_2^*(A_1 - \lambda)^{-1}A_2\right)^{-1} \\
    &= (A_1 - \lambda)^{-1} A_2 (A_3 - \lambda)^{-1} \left(I - (A_3 - \lambda)^{-1} A_2^* (A_1 - \lambda)^{-1} A_2\right)^{-1}, \\
    T_3 &= -\left(A_3 - \lambda - A_2^*(A_1 - \lambda)^{-1}A_2\right)^{-1} \\
    &= (\lambda - A_3)^{-1} \left(I - (A_3 - \lambda)^{-1} A_2^* (A_1 - \lambda)^{-1} A_2\right)^{-1}.
\end{align*}
We introduce the following notations: 
\begin{equation}
    \begin{split}
        B_1 = \left(A_1-\lambda\right)^{-1} A_2\left(A_3-\lambda\right)^{-1} A_2^*, \\
    B_2 = \left(A_3-\lambda\right)^{-1} A_2^*\left(A_1-\lambda\right)^{-1} A_2.
    \end{split}
\end{equation}
Therefore 
\begin{equation}
    \begin{split}
         T_1 &= (\lambda - A_1)^{-1} (I - B_1)^{-1}, \\
         T_2 &= (\lambda - A_1)^{-1} A_2 (\lambda - A_3)^{-1} (I - B_2)^{-1}, \\
         T_3 &= (\lambda - A_3)^{-1}(I - B_2)^{-1}.
    \end{split}
\end{equation}


Since $\frac{1}{\epsilon} D_\epsilon$ and $D_0$ are self-adjoint, they have real spectra. Since the fiber is compact, we can pick a constant $b > 0$ such that $20 b > \max\{ \lambda_1(\frac{1}{\epsilon} D_\epsilon),\lambda_1(D_0)\}$, where $\lambda_1(\frac{1}{\epsilon} D_\epsilon)$ and $\lambda_1(D_0)$ are the smallest nonzero eigenvalues of $\frac{1}{\epsilon} D_\epsilon$ and $D_0$, respectively. Then $\spec(\frac{1}{\epsilon}D_\epsilon) \subset \{0\} \cup [20b, \infty)$. 

To define the integral, let $\Gamma = \Gamma' \cup \Gamma'' \subset \C$, where 
\begin{equation}
    \begin{split}
        \Gamma' &= \partial B_{b}(0), \\
        \Gamma'' &= \{\abs{\Re z} \ge 4b, \abs{\Im z} = 2b \} \cup \{\abs{\Re z} = 4b, \abs{\Im z} \le 2b \}.
    \end{split}
\end{equation}
Here $\Gamma_i$ and $\Gamma$ take the counterclockwise orientations. The following lemma plays a crucial role in proving Theorem \ref{thm:large_time}.
\begin{lemma} \label{lem:T_i}
    For any $\lambda \in \Gamma$, we have
    \begin{enumerate}
        \item $\lim_{\epsilon \to 0} \norm{T_1}_{0,0} = \lim_{\epsilon \to 0} \norm{T_2}_{0,0} = 0$.
        \item $\lim_{\epsilon \to 0} \norm{T_3 - (\lambda - D_0)^{-1}}_{0,0} = 0$.
    \end{enumerate}
\end{lemma}
\begin{proof}
    First we analyze the norms of operators that appear in the definitions of $T_1$, $T_2$ and $T_3$. For any $\lambda \in \Gamma$ fixed, we choose $\epsilon > 0$ sufficiently small such that $\abs{\lambda} \le \frac{\lambda_0}{2 \epsilon}$.
    It follows from lemma \ref{lem:spec} that 
\begin{align*}
    \left\|\left(A_1-\lambda\right) u\right\|_0 \geqslant\left\|A_1 u\right\|_0-|\lambda|\|u\|_0 \geqslant\left(\frac{\lambda_0}{\epsilon}-|\lambda|\right)\|u\|_0,
\end{align*}
which implies 
\begin{equation} \label{ineq:A_1}
    \norm{(A_1 - \lambda)^{-1}}_{0,0} \le \frac{1}{\frac{\lambda_0}{\epsilon} - \abs{\lambda}} \le \frac{2\epsilon}{\lambda_0}.
\end{equation}
Similarly, since $D_0$ is self-adjoint, we have
\begin{align*}
        \left\|\left(A_3-\lambda\right) u\right\|_0 &\geqslant\left\|\left(D_0-\lambda\right) u\right\|_0-\left\|\epsilon p \frac{T}{4} p u\right\|_0 \\
        & \geqslant b\|u\|_0-\epsilon C_1 \|u\|_0 =\left(b-\epsilon C_1\right)\|u\|_0.
\end{align*}
Thus 
\begin{equation} \label{ineq:A_3}
    \norm{(A_3 - \lambda)^{-1}}_{0,0} \le \frac{1}{b-\epsilon C_1}.
\end{equation}
By combining (\ref{ineq:A_1}), (\ref{ineq:A_3}), and the fact that $\norm{A_2}_{0,0} \le \sqrt{C_2}$ is bounded, we obtain
\begin{align} \label{ineq:B_1B_2} 
    \norm{B_1}_{0,0} =  \norm{B_2}_{0,0} \le \frac{2 C_2}{\lambda_0} \frac{\epsilon}{b-\epsilon C_1}  < 1, 
\end{align}
provided that $\epsilon$ is small enough. In fact, $ \lim_{\epsilon \to 0}\norm{B_i} = 0$ for $i = 1, 2$. As a result,
\begin{align}
    \norm{(I-B_i)^{-1}}_{0,0} \le \frac{1}{1-\norm{B_i}_{0,0}} \le \frac{1}{1 - \frac{2C_2}{\lambda_0} \frac{\epsilon}{b - \epsilon C_1}} \to 1, ~~~~ \text{ as } \epsilon \to 0.
\end{align}

Using (\ref{ineq:A_1}), (\ref{ineq:A_3}) and (\ref{ineq:B_1B_2}), we obtain the following estimates on the norms of the $T_i$'s. 
\begin{align}
        \norm{T_1}_{0,0} &= \norm{(\lambda - A_1)^{-1}(I - B_1)^{-1}}_{0,0} \le \frac{2\epsilon}{\lambda_0} \frac{1}{1 - \frac{2C_2}{\lambda_0} \frac{\epsilon}{b - \epsilon C_1}},  \label{ineq:T_1}\\
        \norm{T_2}_{0,0} &= \norm{(\lambda - A_1)^{-1} A_2 (\lambda - A_3)^{-1} (I - B_2)^{-1}}_{0,0} \notag \\
        &\le \frac{2\epsilon}{\lambda_0} \frac{C_2}{b - \epsilon C_1} \frac{1}{1 - \frac{2C_2}{\lambda_0} \frac{\epsilon}{b - \epsilon C_1}}. \label{ineq:T_2}
\end{align}
In particular, $\lim_{\epsilon \to 0} \norm{T_1}_{0,0} = \lim_{\epsilon \to 0} \norm{T_2}_{0,0} = 0$.

On the other hand, by the second resolvent identity, one has
\[ (\lambda - A_3)^{-1} - (\lambda - D_0)^{-1} = -(\lambda - A_3)^{-1} \left(\epsilon p \frac{T}{4} p \right) (\lambda - D_0)^{-1}, \]
which implies 
\[ \norm{(\lambda - A_3)^{-1} - (\lambda - D_0)^{-1}}_{0,0} \le  \frac{C_3}{b^2 - \epsilon b C_1} \sqrt{\epsilon} . \]
Therefore a direct computation shows
\begin{equation}
    \begin{split}
        &\norm{T_3 - (\lambda - D_0)^{-1}}_{0,0} \\
        &\le \norm{T_3 - (\lambda - D_0)^{-1}(I- B_2)^{-1}}_{0,0} + \norm{(\lambda - D_0)^{-1}(I- B_2)^{-1} - (\lambda - D_0)^{-1}}_{0,0} \\
        & \le \norm{(\lambda - A_3)^{-1} - (\lambda - D_0)^{-1}}_{0,0} \cdot \norm{(I - B_2)^{-1}}_{0,0} \\
        &+ \norm{(\lambda - D_0)^{-1}}_{0,0} \cdot \norm{(I - B_2)^{-1} - I}_{0,0} \\
        &\le o(\epsilon) + C \norm{(I - B_2)^{-1} B_2}_{0,0} = o(\epsilon). \label{ineq:T_3}
    \end{split}
\end{equation}
So we conclude that as $\epsilon \to 0$, $T_3$ approaches $(\lambda - D_0)^{-1}$ under the operator norm. 
\end{proof}

In the end, we obtain the following key technical result.
\begin{theorem} \label{thm:large_time}
    Suppose $\ker D_Y$ is a vector bundle over $M$. Let $\delta > 0$ and $t \in [\delta, +\infty)$ be fixed. Then the adiabatic limit of the heat operator is
    \begin{equation*}
        \lim_{\epsilon \to 0}\bigg\| e^{- \frac{t}{\epsilon^2}D_\epsilon^2} - \begin{pmatrix}
            0 & 0 \\
            0 & e^{-t D_0^2}
        \end{pmatrix} \bigg\|_{0,0}  = 0.
    \end{equation*} 
\end{theorem}
\begin{proof}
    First we show that the heat operator can be expressed as an integral. It follows from the proof of lemma \ref{lem:T_i} that
\begin{align*}
    \left| \int_\Gamma e^{-t \lambda^2} \left(\lambda - \frac{1}{\epsilon}D_\epsilon\right)^{-1} \d \lambda  \right| & \le \int_\Gamma \left|e^{-t \lambda^2}\right| \left\|\left(\lambda - \frac{1}{\epsilon}D_\epsilon\right)^{-1}\right\|_{0,0} \d \lambda \\
    &= 2 \int_\R e^{-t (x^2 - b^2)} \left\|\left(\lambda - \frac{1}{\epsilon}D_\epsilon\right)^{-1}\right\|_{0,0} \d x \\
    &\le 2(C + o(\epsilon)) \int_\R e^{-t (x^2 - b^2)}  \d x \\
    &= (C' + o(\epsilon)) e^{t b^2}.
\end{align*}
In particular, the integral is bounded for fixed $\epsilon$ and $t$. Therefore Theorem \ref{thm:funcional_calculus} implies  
\begin{equation}
     e^{- \frac{t}{\epsilon^2}D^2_{\epsilon}} = \int_\Gamma e^{-t \lambda^2} \left(\lambda - \frac{1}{\epsilon}D_\epsilon\right)^{-1} \d \lambda
\end{equation}
for any $\epsilon > 0$. To distinguish the contributions of different $\lambda$'s to the integral, we split the region as $\Gamma = \{\lambda \in \Gamma : \abs{\lambda} \le \frac{\lambda_0}{2\epsilon}\} \cup  \{\lambda \in \Gamma : \abs{\lambda} > \frac{\lambda_0}{2\epsilon}\}$.

For the first part, we consider the substitution $\tau = \epsilon \lambda$. Let $\Gamma_\epsilon = \{ x + \i \frac{b}{\epsilon} \mid x \in \R \}$. A direct computation shows that the following estimate is uniform in $\epsilon$:
\begin{equation}
    \left|\int_{\abs{\lambda} \le \frac{\lambda_0}{2\epsilon}, \lambda \in \Gamma} e^{-t \lambda^2} \epsilon \d \lambda\right| = \left|\int_{\abs{\tau} \le \frac{\lambda_0}{2}, \tau \in \Gamma_\epsilon} e^{- \frac{t}{\epsilon^2} \tau^2} \d \tau \right| \le C < \infty
\end{equation}
This, together with (\ref{ineq:T_1}), (\ref{ineq:T_2}) and (\ref{ineq:T_3}), implies
\begin{equation}   \label{limit:lambda_small}
    \begin{split}
    \lim_{\epsilon \to 0} \Bigg| \int_{\abs{\lambda} \le \frac{\lambda_0}{2\epsilon}, \lambda \in \Gamma}  &e^{-t \lambda^2} \begin{pmatrix}
            T_1 & T_2 \\
            T_2^* & T_3 - (\lambda - D_0)^{-1}
    \end{pmatrix} \d \lambda \Bigg|\\
    &\le \lim_{\epsilon \to 0} \int_{\abs{\lambda} \le \frac{\lambda_0}{2\epsilon}, \lambda \in \Gamma} \left| e^{-t \lambda^2} \right| o(\epsilon) \d x \\
    &\le \lim_{\epsilon \to 0} \int_\Gamma \left| e^{-t \lambda^2} \right| o(\epsilon)\d x \\
    &= 0.
    \end{split}
\end{equation}

For the second part, one has
\begin{equation}
    \begin{split}
        \label{limit:lambda_big}
        \lim_{\epsilon \to 0} \int_{\{x^2 > \lambda_0^2/4\epsilon^2 - b^2\}} &e^{-t(x^2 - b^2)} \left\|(\lambda - \frac{1}{\epsilon} D_\epsilon)\right\|^{-1}_{0,0} \d x \\
        &\le \lim_{\epsilon \to 0} \int_{\{x^2 > \lambda_0^2/4\epsilon^2 - b^2\}}  e^{-t(x^2 - b^2)} \frac{1}{b} \d x \\
        &= \lim_{\epsilon \to 0} \int_{\{y^2 > \lambda_0^2/4 - \epsilon^2 b^2\}}  e^{t b^2} e^{-t \frac{y^2}{\epsilon^2}} \frac{1}{b\epsilon} \d y \\
        &= 0,
    \end{split}
\end{equation}

Combining (\ref{limit:lambda_small}) and (\ref{limit:lambda_big}), one obtains 
\begin{align*}
    \lim_{\epsilon \to 0} e^{- \frac{t}{\epsilon^2} D_\epsilon^2} &= \lim_{\epsilon \to 0} \int_{\Gamma} e^{-t \lambda} \left(\lambda - \frac{1}{\epsilon} D_\epsilon\right)^{-1} \d \lambda \\
        &= \int_{\Gamma} e^{-t \lambda} \begin{pmatrix}
            0 & 0 \\
            0 & (\lambda - D_0)^{-1}
        \end{pmatrix} \d \lambda \\
        & = \begin{pmatrix}
            0 & 0 \\
            0 & e^{-t D_0^2}
        \end{pmatrix}.
\end{align*}
This completes the proof.
\end{proof}


\subsection{Noncompact fiber case and Proof of Theorem \ref{thm:interm_result_1}} \label{subsect:noncpt_fiber}

Now we deal with the case that the fiber is complete but not compact. Assume $Y = \R^k$ is the flat Euclidean space, i.e. $E$ is a real vector bundle of rank $k$. One should notice that since the fiber is not compact, we shall no longer expect that the spectrum of $\abs{A_1}$ has a positive lower bound. 

Consider the Witten Dirac operator $D_{\epsilon, \tau h}^2$ defined in the previous section. As an analogue of (\ref{eq:D_epsilon}), we have
\begin{equation} \label{eq:D_epsilon_Witten}
    D_{\epsilon, \tau h} = \epsilon \tilde{D}_M + D_Y + \tau \hat{c}(d h) - \frac{\epsilon^2}{4}T.
\end{equation}

Let $\tilde{D}_Y := D_Y + \tau \hat{c}(d h)$.
Suppose for now that $\ker \tilde{D}_Y$ is a vector bundle. Consider the following decomposition of the $L^2$-space: 
    \begin{equation} \label{equ:decomp_noncompact}
        L^2(\Lambda^* E) = L^2(\ker\tilde{D}_Y)^\perp \oplus L^2(\ker\tilde{D}_Y).
    \end{equation}
Then we have $\frac{1}{\epsilon}D_{\epsilon, \tau h} = \begin{pmatrix}
        \tilde{A}_1 & \tilde{A}_2 \\
        \tilde{A}_2^* & \tilde{A}_3
    \end{pmatrix}$, where $\tilde{A}_2^*$ is the adjoint of $\tilde{A}_2$, and 
    \begin{align*}
        \tilde{A}_1 &= \frac{1}{\epsilon}D_Y + \frac{1}{\epsilon} \tau \hat{c}(\d h) + p^\perp \tilde{D}_M p^\perp + \epsilon p^\perp \frac{T}{4} p^\perp, \\
        \tilde{A}_2 &= p^\perp \tilde{D}_M p + \epsilon p^\perp \frac{T}{4} p, \\
        \tilde{A}_3 &= D_0 + \epsilon p \frac{T}{4} p .
    \end{align*}

A similar argument as in the proof of Lemma \ref{lem:spec} shows that  $\spec(\abs{\tilde{A}_1}) \subset [\frac{\tilde{\lambda}_0(\tau)}{\epsilon}, +\infty)$, and $\tilde{A}_2$ is a bounded operator. This time we take $b > 0$ so that $20 b > \max\{ \lambda'_0,\lambda_1(D_0)\}$, where $\lambda'_0$ is the constant given in Theorem \ref{thm:interm_result_spec}.
Then the rest of the argument in section \ref{subsect:cpt_fiber} can be extended to the noncompact case without any change. 

It turns out that, in the noncompact setting, the assumption that $\ker \tilde{D}_Y$ is a vector bundle is redundant. To see this, we observe that since $T^* E = \pi^* T^*M \oplus T^* Y$, we have 
\begin{equation} \label{equa:split_of_exterior_bundle}
    \Lambda^* E = \pi^* \left(\Lambda^* M\right) \otimes \Lambda^* Y.
\end{equation}
For simplicity, in the following we will still denote $\pi^* \left(\Lambda^* M\right)$ by $\Lambda^* M$. 

Since $\tilde{D}_Y$ only acts on the $\Lambda^* Y$ part of the exterior algebra bundle, it is easy to see 
\begin{equation} \label{equa:split_of_kernel}
    \ker \tilde{D}_Y = \Lambda^* M \otimes \ker (\tilde{D}_Y|_{\Lambda^* Y}).
\end{equation}

A direct computation in local coordinates shows 
\begin{align*}
    (\tilde{D}_Y|_{\Lambda^* Y})^2 &= - \sum_{j = 1}^k \frac{\partial^2}{\partial (y^j)^2} - k \tau + \tau \abs{y}^2 + \tau \cdot 2 \sum_{j = 1}^k f^j \wedge \iota_{f_j}.
\end{align*}
It is a standard fact (see, for example,  \cite[Section 5.4]{Zhang_book}) that locally 
\begin{equation} \label{eq:ker_D_Y}
    \ker \tilde{D}_Y|_{\Lambda^* Y} = \ker (\tilde{D}_Y|_{\Lambda^* Y})^2 = \mathrm{span}_\R \big\{ e^{- \frac{\tau \abs{y}^2}{2}} \big\}.
\end{equation}

As we have seen, $e^{- \frac{\tau \abs{y}^2}{2}}$ can be extended to a global function on $E$. As a result, $\ker \tilde{D}_Y|_{\Lambda^* Y} \cong E \times \R^1$ is a trivial line bundle over $E$. Therefore we end up with the following noncompact version of theorem \ref{thm:large_time}.

\begin{theorem} \label{thm:large_time_noncompact}
    Let $E \to M$ be a vector bundle over a compact manifold $M$. For any $t > 0$ fixed, we have 
    \begin{equation}
        \lim_{\epsilon \to 0}\bigg\| e^{- \frac{t}{\epsilon^2}D_{\epsilon, \tau h}^2} - \begin{pmatrix}
            0 & 0 \\
            0 & e^{-t D_0^2}
        \end{pmatrix} \bigg\|_{0,0}  = 0.
    \end{equation} 
\end{theorem}

\begin{remark} \label{rmk:T_tilde_i}
    If we decompose $(\lambda - \frac{1}{\epsilon}D_{\epsilon, \tau h})^{-1}$ with respect to (\ref{equ:decomp_noncompact}) as 
\begin{equation}
    \left(\lambda - \frac{1}{\epsilon} D_{\epsilon, \tau h} \right)^{-1} = \begin{pmatrix}
        \widetilde{T}_1 & \widetilde{T}_2 \\
        \widetilde{T}_2^* & \widetilde{T}_3 
    \end{pmatrix},
\end{equation}
    then from the proof of Lemma \ref{lem:T_i} one can actually deduce that for $\abs{\lambda} \le \frac{\lambda_0}{2 \epsilon}$, 
    \begin{equation}
        \norm{\widetilde{T}_1}_{0,0}, \norm{\widetilde{T}_2}_{0,0}, \norm{\widetilde{T}_3 - (\lambda - D_0^{-1})}_{0,0} \le C \epsilon,
    \end{equation}
\end{remark}

\begin{proof}[Proof of Theorem \ref{thm:interm_result_1}]
    Recall that we assume $M$ is odd-dimensional, thus $\tr_s (e^{-t_0 D_M^2}) = 0$. The proof is completed by using Remark \ref{rmk:T_tilde_i} and an argument similar to \cite[Section 5(c)-5(e)]{BB1994}.
\end{proof}


\subsection{Index of vector bundle}
As an interesting and important application of the discussion above, we show that the index of the total space $E$ of a vector bundle coincides with that of the base manifold $M$.

\begin{theorem} \label{thm:tr_s_limit}
    For any $t >0$, 
    \begin{equation} 
        \lim_{\epsilon \to 0} \tr_s \left(e^{-\frac{t}{\epsilon^2}D_{\epsilon, \tau h}^2}\right) = \tr_s \left(e^{-t D_0^2}\right).
    \end{equation}
\end{theorem}
\begin{proof}
    The proof is essentially the same as Theorem \ref{thm:interm_result_1}, so we omit it here.
\end{proof}

\begin{proof}[Proof of Theorem \ref{thm:ind_equal}]
    The isomorphism between cohomologies is given in (\ref{eq:cohom_isom}).
    
    Recall that $D_0 = p \tilde{D}_M p$ maps $\ker \tilde{D}_Y$ to itself. Using the splittings (\ref{equa:split_of_exterior_bundle}) and (\ref{equa:split_of_kernel}), we can write $D_0$ as a twisted Dirac operator 
    \begin{equation*}
        D_0 = D_M \otimes 1 + 1 \otimes D^u: \Lambda^*M \otimes \ker (\tilde{D}_Y|_{\Lambda^* Y}) \to \Lambda^* M \otimes \ker (\tilde{D}_Y|_{\Lambda^* Y}),
    \end{equation*}
    where $D^u$ is the induced Dirac operator on $\ker (\tilde{D}_Y|_{\Lambda^* Y})$. With respect to the natural $\Z_2$ grading 
    \begin{align*}
        \Lambda^+ E &= \Big( \Lambda^+ M \otimes \Lambda^+ Y \Big) \oplus \Big( \Lambda^- M \otimes \Lambda^- Y \Big), \\
        \Lambda^- E &= \Big( \Lambda^- M \otimes \Lambda^+ Y \Big) \oplus \Big( \Lambda^+ M \otimes \Lambda^- Y \Big),
    \end{align*}
    the "positive" part of $D_0$ is $ D^+_0 = D_M^+ \otimes 1 + 1 \otimes D^u$. One should notice that $D_M^+ \otimes 1$ and $1 \otimes D^u$ map $\Lambda^+_M \otimes \ker (\tilde{D}_Y|_{\Lambda^* Y})$ to $\Lambda^- M \otimes \Lambda^+ Y$ and $\Lambda^+ M \otimes \Lambda^- Y$ separately. When coupling with the projection to the kernel of $\tilde{D}_Y|_{\Lambda^* Y}$, $1 \otimes D^u$ is the zero map.  Therefore the kernel and cokernel of $D_0^+$ split as 
    \begin{equation}
        \begin{split}
            \ker D^+_0 &= \ker D^+_M \otimes \ker (\tilde{D}_Y|_{\Lambda^* Y}), \\
        \coker D^+_0 &= \coker D^+_M \otimes \ker (\tilde{D}_Y|_{\Lambda^* Y}).
        \end{split}
    \end{equation}
    Since $\ker (\tilde{D}_Y|_{\Lambda^* Y})$ is $1$-dimensional, we obtain 
    \begin{equation}
        \begin{split}
            \ind D_0^+ &= \dim \ker D_0^+ - \dim \coker D_0^+ \\
        &= \dim \ker D_M^+ - \dim \coker D_M^+ \\
        &= \ind (M).
        \end{split}
    \end{equation}
    On the other hand, Theorem \ref{thm:tr_s_limit} and the heat equation formulation for index imply 
    \begin{equation}
        \begin{split}
            \ind D_0^+ = \tr_s e^{-t D_0^2} 
        = \tr_s \begin{pmatrix}
            0 & 0 \\
            0 & e^{-t D_0^2}
        \end{pmatrix} 
        = \lim_{\epsilon \to 0} \tr_s e^{- \frac{t}{\epsilon^2} D^2_{\epsilon, \tau h}} 
        = \ind(E),
        \end{split}
    \end{equation}
    which proves the theorem.
\end{proof}

\section{Local index theorems and small time behavior} 
From now on, we will assume $n$ is odd and $k$ is even.

\label{sect:local_ind_thm}
To evaluate $\int_0^\infty t^{s-1}\tr_s(N e^{-t \underline{\Delta}}) \d t$, we still need to understand the behavior of the supertrace when $t$ is small. 
The key to this is the local index theory, which is usually obtained by applying a rescaling technique introduced by Getzler \cite{G86}, and studying the limit of the supertrace of the heat kernel.

The standard Getzler's rescaling technique amounts to a rescale operator $G$ acting on functions, differential forms, as well as differential operators over $M$. More precisely: 
\begin{enumerate}
    \item For $f(t, x) \in C^\infty(\R^1_+ \times M)$, we define $(G f)(t ,x) = f(\epsilon^2 t, \epsilon x)$. One can easily verify that 
    \[  G  \circ \frac{\partial}{\partial x_i} \circ G^{-1} = \frac{1}{\epsilon} \frac{\partial}{\partial x_i}, ~~~ G \circ \frac{\partial}{\partial t} \circ G^{-1} = \frac{1}{\epsilon^2} \frac{\partial}{\partial t}. \]
    \item Suppose $e^1, \cdots, e^n$ is a local orthonormal frame for $T^*M$. Then $G$ acts on $\Lambda^p M$ as a rescale by its degree:
    \[ G (e^{i_1} \wedge \cdots \wedge e^{i_p}) = \epsilon^{-p} e^{i_1} \wedge \cdots \wedge e^{i_p}.  \]
    Thus it acts on Clifford multiplications by conjugation as 
    \[ G \circ c(e_i) \circ G^{-1} = \frac{1}{\epsilon} e^i \wedge - \epsilon \iota_{e_i}. \]
\end{enumerate}

In this section, we will introduce several modified rescalings that fit into different settings. Then we will use them to prove Theorem \ref{thm:a=b}, Theorem \ref{thm:interm_result_I_4_1}, and Theorem \ref{thm:interm_result_I_4_2}.

\subsection{Matching the leading terms: Proof of Theorem \ref{thm:a=b}} \label{subsect:a=b}

By Theorem \ref{thm:tr_s_asymp_expan}, for $t$ small enough $\Tr_s \left(N e^{-t D_{\tau h}^2}\right)$ admits an asymptotic expansion as 
\begin{equation} \label{equ:tr_s_expantion_on_E}
    \Tr_s \left(N e^{-t D_{\tau h}^2}\right) \sim \sum_{p=- n-2k}^\infty \sum_{q\in \Z} a_{p/2,q} ~ t^{\frac p2} \tau^q.
\end{equation}
By the construction in \cite{DY22}, we immediately see that $a_{p/2,q} = 0$ for $p$ even, i.e. the expansion only consists of terms with half-integer powers in $t$. 

Now, we study the convergence of the coefficients in the asymptotic expansion. We adopt the construction in \cite{BZ92} and consider a rescaling $G_1$ defined as follows:
\begin{enumerate}
    \item For $f(t, x, y) \in C^{\infty}(\R^1 \times E)$, $(G_{1}f)(t, x, y) = f(t, \sqrt{t} x, \sqrt{t} y)$. It can be checked easily that as operators acting on functions, 
    \[ G_{1} \circ e_i \circ G_{1}^{-1} = \frac{1}{\sqrt{t}} e_i, ~~ G_{1} \circ f_j \circ G_{1}^{-1} = \frac{1}{\sqrt{t}} f_j.\]
    \item The action on differential forms is defined accordingly. For our purpose, we rescale both the left and right Clifford multiplications. As a result, we have 
    \begin{align*}
        G_1 \circ c(e^i) \circ G_1^{-1} &= t^{-\frac{1}{4}} \left(e^i \wedge - t^{\frac{1}{2}} \iota_{e_i}\right) := t^{-\frac{1}{4}} c_t(e^i), \\
        G_1 \circ \hat{c}(e^i) \circ G_1^{-1} &= t^{-\frac{1}{4}} \left(e^i \wedge + t^{\frac{1}{2}} \iota_{e_i}\right) := t^{-\frac{1}{4}} \hat{c}_t(e^i)).
    \end{align*}
    We define the actions on $c(f^j)$ and $\hat{c}(f^j)$ similarly.
    A useful observation is that $c_t$ and $\hat{c}_t$ converge to wedge products as $t \to 0$.
    \item In analogy to the treatment in \cite{DY22}, we rescale the time parameter of Witten deformation by $G_1(\tau) = t^{- \frac 12} \tau$.
\end{enumerate}

\begin{remark}
    In fact, $G_1$ amounts to localizing the calculation from $(x, y)$ to any fixed point $(x_0, y_0)$. In the definition above we set $(x_0, y_0) = (0,0)$ for convenience.
\end{remark}

Next, we derive the corresponding Bochner-Lichnerowicz-Weitzenbock formula (or BLW formula for short). In what follows, we will write the index corresponding to $f_i$ as $\bar{i}$, and use the notation $\overline\sum_{ijkl}$ to denote the summation over all possible combinations of $\{i, j, k, l\}$, while allowing each index to be either with or without a "bar". For instance,  
\begin{align*}
        \overline{\sum_{i, j, k, l}} &R_{ijkl} c(e_i) c(e_j) \hat{c}(e_k) \hat{c}(e_l) \\
        := &\sum_{i, j, k, l}R_{ijkl}c(e_i) c(e_j) \hat{c}(e_k) \hat{c}(e_l) + \sum_{\bar{i}, j, k, l}R_{\bar{i}jkl}c(f_i) c(e_j) \hat{c}(e_k) \hat{c}(e_l) + \cdots  \\
        &+ \sum_{\bar{i}, \bar{j}, k, l}R_{\bar{i}\bar{j}kl}c(f_i) c(f_j) \hat{c}(e_k) \hat{c}(e_l) + \cdots + \sum_{\bar{i}, \bar{j}, \bar{k}, l}R_{\bar{i}\bar{j}\bar{k}l}c(f_i) c(f_j) \hat{c}(f_k) \hat{c}(e_l) \\
        &+ \sum_{\bar{i}, \bar{j}, \bar{k}, \bar{l}}R_{\bar{i}\bar{j}\bar{k}\bar{l}}c(f_i) c(f_j) \hat{c}(f_k) \hat{c}(f_l).
    \end{align*}

\begin{lemma}\label{lem:BLW_formula}
Let $R^{E}$ denote the scalar curvature of $E$ with respect to $g_1 = g_E$. Then
    \begin{align*}
        D_{\tau h}^2
            = &-\sum_i\Big[ e_i + \frac{1}{8} \overline{\sum_{j, k, l}}R^{L}_{ijkl}(0)x^l \big(c(e^j)c(e^k)) - \hat{c}(e^j)\hat{c}(e^k)\big) \Big]^2 \\
            &-\sum_i\Big[ f_i + \frac{1}{8} \overline{\sum_{j, k, l}}R^{L}_{\bar{i}jkl}(0)x^l \big(c(e^j)c(e^k)) - \hat{c}(e^j)\hat{c}(e^k)\big) \Big]^2 \\
            & + \frac{R^{E}}{4} + \frac{1}{8}\overline{\sum_{i, j, k, l}}R^{L}_{ijkl}c(e^i) c(e^j) \hat{c}(e^k) \hat{c}(e^l) + 2\tau \sum_i c(f^i) \hat{c}(f^i) + \tau^2 \abs{y}^2.
    \end{align*}
\end{lemma}
\begin{proof}
    For the Witten Laplacian, since $h = \sum \frac{1}{2}(y^i)^2$, one can easily verify that 
    \begin{equation}
        D_{ \tau h}^2 = D^2 + 2\tau \sum_i c(f^i) \hat{c}(f^i) + \tau^2 \abs{y}^2.
    \end{equation}
    By this and the standard BLW formula, we complete the proof.
\end{proof}

Next we consider the conjugation action of $G_1$ on $-t D^2_{\tau h}$ and compute the limit.

\begin{lemma}\label{lem:G_1_Lap_limit}
    \begin{align*}
       L_1 :&= \lim_{t \to 0} G_{1} \circ (-t D^2_{\tau h}) \circ G_{1}^{-1}  \\
        &= - \Delta_0
        -\frac{1}{8} \overline{\sum_{i, j, k, l}}R^{L}_{ijkl}(x_0, y_0)e^i\wedge e^j\wedge \hat{e}^k\wedge \hat{e}^l\wedge  - 2 \tau \sum_i f^i \wedge \hat{f}^i - \tau^2 \abs{y_0}^2,
    \end{align*}
    where $\Delta_0 = - \sum_i e_i^2 - \sum_i f_i^2$ is the Euclidean Laplacian.
\end{lemma}
\begin{proof}
    To apply Getzler's rescaling technique, for any $p_0 = (x_0, y_0) \in E$, we trivialize $\Lambda T^*E$ over a normal coordinate neighborhood of $p_0$ via parallel transport, with respect to $\nabla^L$, along radical geodesics. Thus if $q\in E$ is sufficiently close to $p_0$, we can identify $T_q E$ with $T_{p_0} E$.
    
    Notice that under the conjugate action of $G_1$, each Clifford multiplication $c$ (or $\hat{c}$) comes with a factor $t^{-\frac{1}{4}}$. The rest of the proof follows from a straightforward computation.
\end{proof}

\begin{corollary} \label{coro:t_limit}
    \begin{align*}
        \lim_{t \to 0} \sqrt{t} G_{1} \circ \left(N e^{-t D^2_{\tau h}}\right) \circ G_{1}^{-1} = \left(\frac{1}{2} \sum_r e^r \wedge \hat{e}^r\wedge + \frac{1}{2} \sum_r f^r \wedge \hat{f}^r\wedge\right) e^{L_1}.
    \end{align*}
\end{corollary}

\begin{proof}
    Recall that the number operator decomposes as $N = N_M \otimes 1 + 1 \otimes N_Y$, where 
    \begin{align*}
        N_M &= e^i\wedge \iota_{e_i} = \frac{n}{2} + \frac{1}{2} \sum_i c(e^i) \hat{c}(e^i),\\
        N_Y &= f^i\wedge \iota_{f_i} = \frac{k}{2} + \frac{1}{2} \sum_i c(f^i) \hat{c}(f^i)
    \end{align*}
    are the horizontal and vertical number operators. Therefore by the definition of $G_1$ and Lemma \ref{lem:G_1_Lap_limit}, we complete the proof.
\end{proof}

The above limit should be understood in the sense that the coefficients of $\sqrt{t} G_{1} \circ \left(N e^{-t D^2_{\tau h}}\right) \circ G_{1}^{-1}$ converge to those of the limit operator as $t \to 0$. In order to see the effect on the supertraces, we recall that the heat kernel $K_\tau(t, x, y)$ (along the diagonal) of $N e^{-t D^2_{\tau h}}$ admits an asymptotic expansion of the form 
\begin{equation}
    K_\tau(t, x,y) \sim \sum_{j,l} t^j\tau^l A_{j,l;I,J,I',J'}(x,y) c(e^I f^J) \hat{c}(e^{I'} f^{J'}), 
\end{equation}
while for $\tr_s\left(N e^{-t D^2_{\tau h}}\right)$ we have 
\begin{equation}
    \tr_s \left(N e^{-t D^2_{\tau h}}\right) \sim \sum_{p \text{ odd}} \sum_q a_{p/2,q} t^{p/2}\tau^q.
\end{equation}

The corresponding heat kernel of the rescaled operator $G_1 \circ Ne^{-t D^2_{\tau h}} \circ G_1^{-1}$ is $t^{\frac{n+k}{2}} G_1 (K_\tau)$. By the definition of $G_1$, one has 
\begin{align*}
    &t^{\frac{n+k}{2}} G_1 (K_\tau)(t, x,y) \\
    &\sim t^{\frac{n+k}{2}} \sum_{j,l} t^j \cdot t^{-\frac{l}{2}} \tau^l A_{j,l;I,J,I',J'}(\sqrt{t} x,\sqrt{t} y) t^{-\frac{\abs{I} + \abs{I'} + \abs{J} + \abs{J'}}{4}} c_t(e^I f^J) \hat{c}_t(e^{I'} f^{J'}).
\end{align*}

Since only terms with $\abs{I} = \abs{I'} = n$ and $\abs{J} = \abs{J'} = k$ can survive under the supertrace, the two factors $t^{\frac{n+k}{2}}$ and $t^{-\frac{\abs{I} + \abs{I'} + \abs{J} + \abs{J'}}{4}}$ are canceled. This leads to 
\begin{align*}
    \tr_s \left(t^{\frac{n+k}{2}} G_1 (K_\tau)(t, 0, 0)\right) \sim (-1)^{\frac{(n+k)(n+k+1)}{2}} 2^{n+k}\sum_{j,l} t^{j - \frac{l}{2}} \tau^lA_{j,l;\text{top}}(0,0) .
\end{align*}
In particular, 
\begin{equation} \label{eq:tr_s_G_1_tau}
    \tr_s \left(G_1 \circ N e^{-t D^2_{\tau h}} \circ G_1^{-1}\right) \sim \sum_{p \text{ odd}} \sum_q a_{p/2,q} t^{\frac{p-q}{2}} \tau^q .
\end{equation}

Recall that by assumption, the base manifold $M$ is compact and odd-dimensional. It follows from \cite[Theorem 5.1]{DM} that 
\begin{equation} \label{equ:tr_s_expantion_on_M}
    \tr_s \left(N_M e^{-t D_M^2}\right) \sim b_{-1/2}t^{-\frac 12} + \sum_{j \ge 1, odd} b_{j/2} t^{\frac{j}{2}},
\end{equation}
and
\begin{equation} \label{equ:b_-1/2}
    b_{-1/2} e^1 \wedge \cdots \wedge e^n =  2\i (-1)^{\frac{n+1}{2}} (16\pi)^{-\frac{n}{2}} \sum_{k=1}^n (-1)^k \mathrm{Pf}(R_k) \wedge e^k,
\end{equation}
where $R_k = \left( \sum_{p, q \ne k} R_{ijkl} e^p \wedge e^k \right)_{i, j \ne k}$ is the matrix $2$-form obtained by deleting the $k$-th row and column from the curvature matrix $2$-form. 

Now we are ready to prove Theorem \ref{thm:a=b}. 
\begin{proof}[Proof of Theorem \ref{thm:a=b}]
    Let $\tilde{K}_\tau(x,y; x,y) \in \Gamma(\Lambda^*E \hat{\otimes} \Lambda^*\hat{E})$ denote the diagonal heat kernel of the limiting operator $e^{L_1}$ in Lemma \ref{lem:G_1_Lap_limit}. It is easy to see that 
    \begin{equation} \label{eq:tilde_K_tau}
        \tilde{K}_\tau = \frac{1}{(4\pi)^{\frac{n+k}{2}}} e^{-\tau^2\abs{y}^2} e^{-\frac{1}{8}\overline{\sum}_{ijkl}R^L_{ijkl}e^i \wedge e^j \wedge \hat{e}^k \wedge \hat{e}^l - 2\tau\sum_i f^i \wedge \hat{f}^i}.
    \end{equation}
    
    Recall that as an element of $\Lambda^*E \hat{\otimes} \Lambda^*\hat{E}$, the curvature $R^L:= (\nabla^L)^2$ is given by 
    \begin{equation}
        R^L = \frac{1}{4} \overline{\sum_{i,j,k,l}} R^L_{ijkl} e^i \wedge e^j \wedge \hat{e}^k \wedge \hat{e}^l.
    \end{equation}
    To deal with the mixed curvature terms in (\ref{eq:tilde_K_tau}), we consider another connection $\tilde{\nabla}$ on $TE$ defined by
    \begin{equation}
        \tilde{\nabla}_{e_i} e_j = \nabla^M_{e_i}e_j, ~ \tilde{\nabla}_{e_i} f_j = \tilde{\nabla}_{f_i} e_j = \tilde{\nabla}_{f_i} f_j = 0.
    \end{equation}
    Clearly $\tilde{\nabla}$ is the extension of the pullback connection of $\nabla^M$, therefore its curvature form $\tilde{R}$ is 
    \begin{equation}
        \tilde{R} = \frac{1}{4} \sum_{i,j,k,l} R^M_{ijkl} e^i \wedge e^j \wedge \hat{e}^k \wedge \hat{e}^l = R^M.
    \end{equation}

    By Chern-Weil theory, there is an $\omega \in \Omega^{n-1}(E)$ so that 
    \begin{equation}
        \int^Be^{-\frac{1}{2}R^L} = \int^Be^{-\frac{1}{2}R^M} + \d \omega.
    \end{equation}
     Moreover, $\omega$ can be written as polynomials of $g^E$ and $g^M$ as well as their derivatives. Since clearly $(E, g^E)$ has bounded geometry, both $e^{-\tau^2\abs{y}^2} \omega$ and $ \d \left(e^{-\tau^2\abs{y}^2} \omega\right)$ are integrable. It follows from a generalized Stokes theorem (see \cite{Gaffney_stokes}) that 
    \begin{equation}
        \int_E \d \left(e^{-\tau^2\abs{y}^2} \omega\right) = 0.
    \end{equation}
    
    Notice that 
    \begin{equation}
        \d \left(e^{-\tau^2\abs{y}^2} \omega\right) = \sum_{i=0}^k 2\tau^2y_i e^{-\tau^2\abs{y}^2} \d y^i \wedge \omega + e^{-\tau^2\abs{y}^2} \d \omega.
    \end{equation}
    By the construction above we see that $\omega$ is independent of $y$. Since $y_i e^{-\tau^2\abs{y}^2}$ is an odd function in $y$, we have $\int_E \sum_{i=0}^k 2y_i e^{-\tau^2\abs{y}^2} \d y^i \wedge \omega  = 0$, and 
    \begin{equation}
        \int_E e^{-\tau^2\abs{y}^2} \d \omega = \int_E \d \left(e^{-\tau^2\abs{y}^2} \omega\right) = 0.
    \end{equation}
    In particular, this shows that in (\ref{eq:tilde_K_tau}), only the curvature terms coming from the base manifold can make a contribution to the integration over $E$.
    
    It follows from Proposition \ref{prop:tr_s_Berezin_int}, Corollary \ref{coro:t_limit}, and the standard Mckean-Singer formula that 

    \begin{align} 
        &\lim_{t \to 0} \sqrt{t} \tr_s \left(G_1 \circ N e^{-t D^2_{\tau h}} \circ G_1^{-1}\right) \label{eq:a_-1/2_lim}\\
            &= (-1)^{\frac{(n+k)(n+k+1)}{2}} 2^{n+k} \int_E \int^B \left[\left(\frac{1}{2} \sum_r e^r \wedge \hat{e}^r\wedge + \frac{1}{2} \sum_r f^r \wedge \hat{f}^r\wedge\right) e^{L_1}\right] \d vol_E \notag\\
            &= (-1)^{\frac{n+k+1}{2}} 2^{n+k-1} \frac{1}{(4\pi)^{\frac{n+k}{2}}} \cdot \notag\\
            &\int_E e^{-\tau^2\abs{y}^2}\left(\int^B \sum_r e^r \wedge \hat{e}^r \wedge  e^{-\frac{1}{8}\overline{\sum}_{ijkl}R^L_{ijkl}e^i \wedge e^j \wedge \hat{e}^k \wedge \hat{e}^l - 2 \tau \sum_i f^i \wedge \hat{f}^i}\right) \d vol_E \notag\\
            &= (-1)^{\frac{n+k+1}{2}} 2^{n+k-1} \frac{1}{(4\pi)^{\frac{n+k}{2}}} \cdot \notag\\
            & \left( \int_{\R^k} (-1)^{\frac{k(k-1)}{2}} (2\tau)^k e^{-\tau^2\abs{y}^2} \d y \right) \cdot \int_M \left( \int^B \sum_r e^r \wedge \hat{e}^r \wedge  e^{-\frac{1}{2} R^M} \right) \d vol_M \notag\\
            &= b_{-1/2}.\notag
    \end{align}

    On the other hand, equation (\ref{eq:tr_s_G_1_tau}) implies 
    \begin{align}
        \lim_{t \to 0} \sqrt{t} \tr_s \left(G_1 \circ N e^{-t D^2_{\tau h}} \circ G_1^{-1}\right) &= \sum_{p-q+1 = 0} a_{p/2, q} \tau^q.
    \end{align}
    However by (\ref{eq:a_-1/2_lim}), the right hand side should be independent of $\tau$. Therefore $a_{p/2,q} = 0$ for all $q = p+1 \ne 0$ and 
\begin{equation}
    \begin{split}
        a_{-1/2,0} = \lim_{t \to 0} \sqrt{t} \tr_s \left(G_1 \circ N e^{-t D^2_{\tau h}} \circ G_1^{-1}\right) = b_{-1/2}.
    \end{split}
\end{equation}
\end{proof}

\subsection{A Vanishing result: Proof of Theorem \ref{thm:interm_result_I_4_1}} \label{subsect:proof_int_res_I_4_1}

We consider a new metric $g_{\sigma, T}$ on $E$ given by 
\begin{equation}
    g_{\sigma, T} := \frac{1}{\sigma^2}g_M + \frac{1}{T^2} g_Y.
\end{equation}
One should notice that this metric depends on two parameters $\sigma, T > 0$ in a non-trivial way.  Let $D_{\sigma, T, \tau h}$ denote the Witten-Dirac operator associated with $g_{\sigma, T}$. Observe that $g_{\sigma, T} = \frac{1}{\sigma^2} \tilde{g}_{T/\sigma}$. Therefore by Proposition \ref{prop:alpha},
\begin{equation} \label{eq:tr_sigma_T}
    \tr_s\left( N_Y e^{-D_{\sigma, T, \tau h}^2} \right) = \tr_s\left( N_Y e^{- \sigma^2 \tilde{D}_{T/\sigma, \tau h}^2} \right).
\end{equation}

Let $\{e^1, \cdots, e^n, f^1, \cdots, f^k\}$ be a locally orthonormal coframe for the metric $g_E = g_{1, 1}$. Set $\tilde{e}^i = \sigma e^i$ and $\tilde{f}^i = T f^i$. Then $\{\tilde{e}^1, \cdots, \tilde{e}^n, \tilde{f}^1, \cdots, \tilde{f}^k\}$ forms an orthonormal coframe for $g_{\sigma, T}$. 
We will use $c(\tilde{e}^{i_1 \cdots i_p} \tilde{f}^{j_1 \cdots j_q})$ to denote $c(\tilde{e}^{i_1})\cdots c(\tilde{e}^{i_p})c(\tilde{f}^{j_1})\cdots c(\tilde{f}^{j_p})$.

Inspired by \cite{BC}, we consider a new rescaling $G_2$ that only involves the base direction:
\begin{enumerate}
    \item For $f(x, y) \in C^{\infty}(E)$, $(G_{2}f)(x, y) = f(\sigma x, y)$. Thus  
    \[ G_{2} \circ \tilde{e}_i \circ G_{2}^{-1} = e_i.\]
    \item We only rescale the left Clifford multiplications along the base direction, i.e. 
    \begin{equation*}
        G_2 \circ c(\tilde{e}^i) \circ G_2^{-1} = \frac{1}{\sigma} \left(\tilde{e}^i \wedge - \sigma^2 \iota_{\tilde{e}_i}\right) := \frac{1}{\sigma}c_\sigma(\tilde{e}^i).
    \end{equation*}
\end{enumerate}

As an analogue of Lemma \ref{lem:BLW_formula}, we have:
\begin{lemma}\label{lem:BLW_formula_new}
Let $R^{E, \sigma, T}$ be the scalar curvature of $E$ with respect to $g_{\sigma, T}$. Then
    \begin{equation*}
        \begin{split}
            &D_{\sigma, T, \tau h}^2
            = - \sigma^2\sum_i\Big[ e_i + \frac{1}{8} \overline{\sum_{j, k, l}}R^{L, \sigma, T}_{ijkl}(0)x^l \big(c(\tilde{e}^j)c(\tilde{e}^k) - \hat{c}(\tilde{e}^j)\hat{c}(\tilde{e}^k)\big) \Big]^2 \\
            & \hspace{.3cm}-T^2\sum_i\Big[ \nabla^E_{f_i}  + \frac 12 \sum_{j,k} \inner{S^{\sigma, T}(f_i)\tilde{f}_j, \tilde{e}_k}_{\sigma, T} \left( c(\tilde{f}^j) c(\tilde{e}^k) - \hat{c}(\tilde{f}^j)\hat{c}(\tilde{e}^k) \right)\\
            & \hspace{.3cm}+ \frac 14 \sum_{j,k} \inner{S^{\sigma, T}(f_i)\tilde{e}_j, \tilde{e}_k}_{\sigma, T} \left( c(\tilde{e}^j)c(\tilde{e}^k) - \hat{c}(\tilde{e}^j)\hat{c}(\tilde{e}^k) \right)\Big]^2 + \frac{1}{4}R^{E, \sigma, T} \\
            &  \hspace{.3cm}+ \frac{1}{8}\overline{\sum_{i, j, k, l}}R^{L,\sigma, T}_{ijkl}c(\tilde{e}^i) c(\tilde{e}^j) \hat{c}(\tilde{e}^k) \hat{c}(\tilde{e}^l) + 2\tau \sum_i c(\tilde{f}^i) \hat{c}(\tilde{f}^i) + \tau^2 T^2 \abs{y}^2.
        \end{split}
    \end{equation*}
\end{lemma}

\begin{proof}
    Recall that $\nabla^{L, \sigma, T}_{f_i} = \nabla^E_{f_i} + S^{\sigma, T}(f_i)$. By \cite[(4.31)]{BZ92}, since $S^{\sigma, T}(f_i)$ is anti-symmetric as a derivation acting on $\Lambda^*(T^*E)$, its action is given by 
    \[ \frac{1}{4} \overline{\sum}_{j, k} \inner{S^{\sigma, T}(f_i)\tilde{e}_j, \tilde{e}_k}_{\sigma, T} \left( c(\tilde{e}^j)c(\tilde{e}^k) - \hat{c}(\tilde{e}^j)\hat{c}(\tilde{e}^k) \right) .\]
    Applying Proposition \ref{prop:S}, Proposition \ref{prop:S_epsilon}, and Lemma \ref{lem:BLW_formula}, we get the desired result.
\end{proof}

Let $S^{\sigma, T} = \nabla^{L, \sigma, T} - \nabla^E$, where $\nabla^{L, \sigma, T}$ is the Levi-Civita connection of $g_{\sigma, T}$ and $\nabla^E$ is the connection in Definition \ref{def:conn}.
\begin{lemma} \label{lem:S_sigma_T}
    $S^{\sigma, T} = \frac{\sigma^2}{T^2} P^H S + P^V S$.
\end{lemma}
\begin{proof}
    Let $U_i, U \in \Gamma(TM), W_i, W \in \Gamma(VE)$. Using Proposition \ref{prop:S}, one can verify easily that 
    \begin{align*}
        S^{\sigma, T}(W_1) W_2 &= \frac{\sigma^2}{T^2}S(W_1)W_2 , ~~~ P^VS^{\sigma, T}(W) U = P^VS(W)U, \\
    S^{\sigma, T}(U) W &= \frac{\sigma^2}{T^2}S(U)W, ~~~ S^{\sigma, T}(U_1) U_2 = S(U_1)U_2,\\
        P^H S^{\sigma, T}(W) U &= \frac{\sigma^2}{T^2}S(W)U,
    \end{align*}
    which completes the proof.
\end{proof}

Now we calculate the limit of the rescaled Dirac operator.

\begin{lemma}\label{lem:G_2_Lap_limit}
    Let $R^{E}$ denote the curvature tensor of $\nabla^E$. Then
    \begin{align*}
       L_2 &:= \lim_{\sigma \to 0} G_{2} \circ (-D^2_{\sigma, T,  \tau h}) \circ G_{2}^{-1}  \\
        &= \sum_i e_i^2 + T^2\sum_i\Big[\nabla^E_{f_i} + \frac{1}{2}\inner{S^T(f_i)f_j, e_k}_T c(\tilde{f}_j) e^k\wedge \\
        &+ \frac{1}{4}\inner{S^T(f_i)e_j, e_k}_T e^j \wedge e^k\wedge \Big]^2 - \frac{1}{8} \sum_{ijkl}R^M_{ijkl} e^i \wedge e^j \wedge \hat{e}^k \wedge \hat{e}^l \\
        & - \frac{1}{8} \sum_{ij\bar{k}\bar{l}} R^{E}_{ij\bar{k}\bar{l}} e^i\wedge e^j\wedge \hat{c}(\tilde{f}^k) \hat{c}(\tilde{f}^l) - \frac{T}{4} \sum_{i\bar{j}\bar{k}\bar{l}}R^{E}_{i\bar{j}\bar{k}\bar{l}} e^i\wedge c(\tilde{f}^j) \hat{c}(\tilde{f}^k) \hat{c}(\tilde{f}^l) \\
        &- 2\tau \sum_{i} c(\tilde{f}^i) \hat{c}(\tilde{f}^i) + \tau^2 T^2 \abs{y}^2.
    \end{align*}
\end{lemma}

\begin{proof}
    First, we calculate the $(\sigma, T)$-dependence of the curvature components. By Lemma \ref{lem:S_sigma_T}, the Levi-Civita connection decomposes as 
    \[ \nabla^{L, \sigma, T} = \nabla^E + \frac{\sigma^2}{T^2} P^H S + P^V S. \] 
    By a direct calculation, we have
        \begin{align*}
            &R^{L,\sigma, T}(e_i, e_j) e_k = \nabla^{L,\sigma, T}_{e_i} \nabla^{L,\sigma, T}_{e_j} e_k - \nabla^{L,\sigma, T}_{e_j} \nabla^{L,\sigma, T}_{e_i} e_k - \nabla^{L,\sigma, T}_{[e_i, e_j]}e_k \\
            &= \Big[R^M(e_i, e_j)e_k + \frac{\sigma^2}{T^2}\Big(S(e_i)S(e_j)e_k - S(e_j)S(e_i)e_k - P^HS([e_i,e_j])e_k\Big)\Big] \\
            &+ \Big[S(e_i) \nabla^M_{e_j}e_k + \nabla^E_{e_i}S(e_j)e_k - S(e_j) \nabla^M_{e_i}e_k - \nabla^E_{e_j}S(e_i)e_k - P^VS([e_i,e_j])e_k \Big], \\
            &R^{L,\sigma, T}(e_i, e_j) f_k = \nabla^{L,\sigma, T}_{e_i} \nabla^{L,\sigma, T}_{e_j} f_k - \nabla^{L,\sigma, T}_{e_j} \nabla^{L,\sigma, T}_{e_i} f_k - \nabla^{L,\sigma, T}_{[e_i, e_j]}f_k \\
            &= \frac{\sigma^2}{T^2}\Big[ S(e_i)\nabla^E_{e_j}f_k + \nabla^M_{e_i}S(e_j)f_k - S(e_j)\nabla^E_{e_i}f_k - \nabla^M_{e_j}S(e_i)f_k - S([e_i, e_j])f_k \Big] \\
            &~~~+ \Big[R^E(e_i, e_j)f_k + \frac{\sigma^2}{T^2} \Big( S(e_i)S(e_j)f_k - S(e_j)S(e_i)f_k \Big)\Big], \\
            &R^{L,\sigma, T}(e_i, f_j) e_k = \nabla^{L,\sigma, T}_{e_i} \nabla^{L,\sigma, T}_{f_j} e_k - \nabla^{L,\sigma, T}_{f_j} \nabla^{L,\sigma, T}_{e_i} e_k - \nabla^{L,\sigma, T}_{[e_i, f_j]}e_k \\ 
            &= \frac{\sigma^2}{T^2} \Big[ \nabla^M_{e_i} P^HS(f_j) e_k + S(e_i) P^V S(f_j) e_k - P^HS(f_j) \nabla^M_{e_i}e_k - S(f_j)S(e_i) e_k \\
            &- P^HS([e_i, f_j])e_k\Big] + \Big[ \frac{\sigma^2}{T^2} S(e_i)P^HS(f_j) e_k + \nabla^E_{e_i}P^VS(f_j)e_k \\
            &- P^VS(f_j)\nabla^M_{e_i}e_k - \nabla^E_{f_j}S(e_i)e_k - P^VS([e_i, f_j]) e_k\Big], \\
            &R^{L,\sigma, T}(e_i, f_j) f_k = \nabla^{L,\sigma, T}_{e_i} \nabla^{L,\sigma, T}_{f_j} f_k - \nabla^{L,\sigma, T}_{f_j} \nabla^{L,\sigma, T}_{e_i} f_k - \nabla^{L,\sigma, T}_{[e_i, f_j]}f_k \\ 
            &= \frac{\sigma^2}{T^2} \Big[ S(e_i)S(f_j)f_k + \nabla^M_{e_i}S(f_j)f_k - \frac{\sigma^2}{T^2} P^HS(f_j)S(e_i)f_k - S(f_j)\nabla^E_{e_i} f_k \\
            &- S([e_i, f_j]) f_k\Big] + \Big[ R^E(e_i, f_j)f_k + \frac{\sigma^2}{T^2} S(e_i) S(f_j) f_k - \frac{\sigma^2}{T^2} P^VS(f_j)S(e_i)f_k \Big], \\
            &R^{L,\sigma, T}(f_i, f_j) f_k = \nabla^{L,\sigma, T}_{f_i} \nabla^{L,\sigma, T}_{f_j} f_k - \nabla^{L,\sigma, T}_{f_j} \nabla^{L,\sigma, T}_{f_i} f_k - \nabla^{L,\sigma, T}_{[f_i, f_j]}f_k \\ 
            &= \frac{\sigma^2}{T^2} \Big[ S(f_i)\nabla^E_{f_j}f_k + \frac{\sigma^2}{T^2} P^HS(f_i)S(f_j)f_k - S(f_j)\nabla^E_{f_i}f_k \\
            &- \frac{\sigma^2}{T^2} P^HS(f_j)S(f_i)f_k - S([f_i, f_j])f_k \Big] + \Big[  R^E(f_i, f_j)f_k +\frac{\sigma^2}{T^2} P^V S(f_i, f_j)f_k\Big].
        \end{align*}
        Recall that for the metric $g_{\sigma, T}$, $\{\sigma e_1, \cdots, \sigma e_n, T f_1, \cdots, T f_k\}$ is a local orthonormal frame. We let $R^{L,\sigma, T}_{ijkl} = Rm^{L,\sigma, T}(\sigma e_i, \sigma e_j, \sigma e_k, \sigma e_l)$, $R^M_{ijkl} = Rm^M(e_i, e_j, e_k, e_l)$, etc. Then
        \begin{align*}
            &R^{L,\sigma, T}_{ijkl} = \sigma^2 R^M_{ijkl} + O\Big(\frac{\sigma^4}{T^2}\Big),~~~  R^{L,\sigma, T}_{ijk\bar{l}} = O\Big(\frac{\sigma^3}{T}\Big), ~~~ R^{L,\sigma, T}_{ij\bar{k}\bar{l}} = \sigma^2 R^E_{ij\bar{k}\bar{l}} + O\Big(\frac{\sigma^4}{T^2}\Big), \\
            &R^{L,\sigma, T}_{i\bar{j}k\bar{l}} = O(\sigma^2),  R^{L,\sigma, T}_{i\bar{j}\bar{k}\bar{l}} = \sigma T R^E_{i \bar{j} \bar{k} \bar{l}} + O\Big(\frac{\sigma^3}{T}\Big), ~~~ R^{L,\sigma, T}_{\bar{i}\bar{j}\bar{k}\bar{l}} = T^2 R^E_{\bar{i}\bar{j}\bar{k}\bar{l}} + O(\sigma^2).
        \end{align*}

    It is easy to see that if we apply the conjugate $G_2 \circ \bullet \circ G_2^{-1}$ to Lemma \ref{lem:BLW_formula_new}, the first two terms become $\sum_i e_i^2 + T^2 \sum_i\Big[\nabla^E_{f_i} + \frac{1}{2}\inner{S^T(f_i)f_j, e_k}_Tc(\tilde{f}_j) e^k\wedge + \frac{1}{4}\inner{S^T(f_i)e_j, e_k}_Te^j \wedge e^k\wedge \Big]^2 + o(\epsilon^2)$.
    On the other hand, the scalar curvature is only a function of $x$ and $y$, and the conjugation acts as a scaling on the variables. Notice that 
    \[ R^{E, \sigma, T} = T^2 \sum_{\bar{i}\bar{j}\bar{k}\bar{l}} R^E_{\bar{i}\bar{j}\bar{k}\bar{l}} + O(\sigma^2), \]
    and $R^E_{\bar{i}\bar{j}\bar{k}\bar{l}}$ is actually the Riemannian curvature of the flat fiber. Therefore the scalar curvature term drops out in the limit. 
    
    Now we calculate the rest of the curvature terms. Under the conjugate action, each left Clifford multiplication along the base direction, like $c(\tilde{e}^i)$, comes with a factor $\sigma^{-1}$. Therefore by the $(\sigma,T)$-dependence of the curvature, the only terms that can survive in the limit are $R^M_{ijkl}$, $R^E_{ij \bar{k} \bar{l}}$, and $R^E_{i \bar{j} \bar{k} \bar{l}}$. This completes the proof.
\end{proof}

As a consequence, we have

\begin{corollary} \label{coro:G_2_heat_op_limit}
    \begin{align}
        \lim_{\sigma \to 0} G_2 \circ N_Ye^{-D^2_{\sigma, T, \tau h}} \circ G_2^{-1} = N_Y  e^{L_2},
    \end{align}
    where $L_2$ is the limit operator in Lemma \ref{lem:G_2_Lap_limit}.
\end{corollary}

In the end, we prove Theorem \ref{thm:interm_result_I_4_1}.

\begin{proof}[Proof of Theorem \ref{thm:interm_result_I_4_1}]

    By (\ref{eq:tr_sigma_T}), Corollary \ref{coro:G_2_heat_op_limit}, and arguing as in \cite{G86}, we have 
    \[ \lim_{\sigma \to 0}\tr_s \left(N_Y e^{ - \sigma^2 \tilde{D}_{T/\sigma,\tau h}^2}\right) = \lim_{\sigma \to 0}\tr_s \left(N_Y e^{ -  D_{\sigma, T, \tau h}^2}\right) = \tr_s \left(N_Y e^{L_2}\right). \]
    
    Since the horizontal differential forms, i.e. those generated by $e^i$ and $\hat{e}^i$, in $L_2$ are all of even degree, but the dimension of $M$ is odd, we see that 
\[ \tr_s \left(N_Y e^{L_2}\right) = 0. \]
    This completes the proof.
\end{proof}

\subsection{Remainder estimate: Proof of Theorem \ref{thm:interm_result_I_4_2}} \label{subsect:proof_int_res_I_4_2}

Inspired by \cite[section 8]{BB1994}, we consider a new parameter $a = \frac{T}{\sigma}$. Then 
\[ \sigma^2\tilde{D}_{T/\sigma, \tau h}^2 = \frac{T^2}{a^2} \tilde{D}_{a, \tau h}^2. \]
Since the ranges for $\sigma$ and $T$ are $\sigma \in (0,1]$ and $T \in [\sigma, 1]$, $a \in [1, \frac{1}{\sigma}]$ and thus Theorem \ref{thm:interm_result_I_4_2} amounts to a uniform upper bound when $a \to \infty$.

\begin{proof}[Proof of Theorem \ref{thm:interm_result_I_4_2}]
    Using notations in section \ref{subsect:proof_int_res_I_4_1}, we see that 
\begin{equation} \label{eq:7.3_1}
    \begin{split}
        \tr_s\left(N_Y e^{-\sigma^2\tilde{D}_{T/\sigma, \tau h}^2}\right) &= \tr_s\left(N_Y e^{-\frac{T^2}{a^2} \tilde{D}_{a, \tau h}^2}\right) \\
    &= \tr_s\left(N_Y e^{-T^2 D_{1/a, 1, \tau h}^2}\right),
    \end{split}
\end{equation}
where $D_{1/a, 1, \tau h}$ is the Witten-Dirac operator of 
\[ g_{1/a, 1} = a^2 g_M + g_Y. \]

Consider another rescaling $G_3$ which is the same as $G_2$ with $\sigma$ replaced by $\frac{1}{a}$. A similar argument as in the previous section shows that $G_3 \circ N_Y e^{-T^2 D^2_{1/a, 1, \tau h}} \circ G_3^{-1}$ approaches a finite limit $N_Y e^{T^2 L_3}$ as $a \to \infty$. Moreover, it is easy to see that $L_3$ is uniformly elliptic. 

By a similar argument as in Section \ref{subsect:a=b} we know for $T$ small, $\tr_s\left(N_Y e^{-T^2 D_{1/a, 1, \tau h}^2}\right)$ admits an expansion of the form 
\begin{equation} \label{eq:7.3_2}
    \tr_s\left(N_Y e^{-T^2 D_{1/a, 1, \tau h}^2}\right) \sim  d_1(a) T  + d_3(a) T^3 + \cdots .
\end{equation}

Since $\tr_s\left(N_Y e^{T^2 L_3}\right) = 0$ and the asymptotic expansion depends only on the local symbol, all the coefficients $d_i(a)$ converge to $0$ as $a \to \infty$. 

On the other hand, by an argument as in \cite[P.~64]{BC} (see also \cite[P.~194-197]{Cheeger87eta}), each $d_i(a)$ admits an expansion like $\sum d_{i, j} a^j$ for $a$ sufficiently large. In particular, the powers of $a$ can only be integers. This implies that $d_i(a) = O(\frac{1}{a})$ for each $i$. Thus by (\ref{eq:7.3_1}) and (\ref{eq:7.3_2}) we have 
\[ \left|\frac{1}{T}\tr_s\left(N_Y e^{-\sigma^2\tilde{D}_{T/\sigma, \tau h}^2}\right) \right| = \left|\frac{1}{T}\tr_s\left(N_Y e^{-T^2 D_{1/a, 1, \tau h}^2}\right) \right| \le \frac{C}{a} \le C.\]
\end{proof}

\section{Proof of Theorem \ref{thm:interm_result_I_4_3}} \label{sect:proof_last_thm} 

The proof essentially follows the argument in \cite[section 9]{BB1994} and \cite[section 11, section 13]{BL}.  

\subsection{Localization of the problem}
Let $\alpha > 0$ be a constant with $\alpha \ll \frac{1}{4} \operatorname{injrad}(M)$, $f: \R \to [0,1]$ be a smooth even cut-off function such that 
\begin{equation*}
    f(t) = \begin{cases}
        1, & \text{if } \abs{t} \le \frac{\alpha}{2}, \\
        0, & \text{if } \abs{t} \ge \alpha.
    \end{cases}
\end{equation*}
Let $g(t) = 1 - f(t)$. For any $v > 0$, let 
\begin{align*}
    F_v(a) &= \frac{1}{\sqrt{2\pi}}\int_{-\infty}^\infty e^{\i t\sqrt{2}a - t^2/2} f(vt) \d t, \\
    G_v(a) &= \frac{1}{\sqrt{2\pi}}\int_{-\infty}^\infty e^{\i t\sqrt{2}a - t^2/2} g(vt) \d t, \\
    H_v(a) &= \frac{1}{\sqrt{2\pi}v}\int_{-\infty}^\infty e^{\i t\sqrt{2}a - t^2/2v^2} g(t) \d t.
\end{align*}

In the following we collect properties about these functions. The proof is standard and is omitted here.
\begin{proposition}\label{prop:F_and_G}
    \begin{enumerate}
        \item $G_v(a) = H_v(\frac{a}{v})$. 
        \item $F_v(a) + G_v(a) = e^{-a^2}$. Therefore 
        \[ F_\sigma(\sigma \tilde{D}_{T/\sigma, \tau h}) + G_\sigma(\sigma \tilde{D}_{T/\sigma, \tau h}) = e^{ - \sigma^2 \tilde{D}^2_{T/\sigma, \tau h}} .\]
        \item There exist holomorphic functions $\tilde{F}_v(a)$ and $\tilde{H}_v(a)$ such that 
        \[ F_v(a) = \tilde{F}_v(a^2), ~~ H_v(a) = \tilde{H}_v(a^2). \]
        \item For any $c > 0 $, $m \in \N$, there exist constants $C_1, C_2 > 0$ such that for any $v \in (0,1], a \in \C$ with $\abs{\operatorname{Im} a} \le c$, 
        \[ \sup_a \abs{a^m H_v(a)} \le C_1 e^{-\frac{C_2}{v^2}}. \]
        \item For any $c > 0 \in$, $m \in \N$, there exists a constant $C> 0$ such that for any $v \in (0,1], a \in \C$ with $\operatorname{Re} a \ge \frac{(\operatorname{Im} a)^2}{4c^2} - c^2$, 
        \[ \sup_a \abs{a^m F_v(a)} \le C. \]
    \end{enumerate}
\end{proposition}

It is useful to observe that for the metric $g_{\sigma, T} = \frac{1}{\sigma^2} g_M + \frac{1}{T^2} g_Y$ considered in section \ref{subsect:proof_int_res_I_4_1}, we have
\[  \tr_s\left( N_Y e^{-D_{\sigma, T, \tau h}^2} \right) = \tr_s\left( N_Y e^{- \sigma^2 \tilde{D}_{T/\sigma, \tau h}^2} \right). \] 
It turns out that for $G_\sigma(\sigma \tilde{D}_{T/\sigma, \tau h})$, the computation would be easier if we consider the original metric $\tilde{g}_{T/\sigma, \tau h}$, while for $F_\sigma(\sigma \tilde{D}_{T/\sigma, \tau h})$ it is better to use the metric $g_{\sigma, T}$.

Using Lemma \ref{lem:S_sigma_T}, we obtain the following analogue of (\ref{eq:D_epsilon_Witten}).
\begin{lemma} \label{lem:D_sigma_T}
\begin{align*}
    D_{\sigma, T} &= \sigma c(\tilde{e}^i) \nabla^u_{e_i} + T c(\tilde{f}^j) \nabla^E_{f_j}  + T \tau \tilde{\hat{c}}(\d h) -  \frac{\sigma^2}{4} c(\tilde{e}^\alpha) c(\tilde{e}^\beta) \tilde{c}(\mathrm{Tor}) (e_\alpha, e_\beta) \\
        & := \sigma \tilde{D}_M + T \tilde{D}_{Y, T} - \frac{\sigma^2}{4}\mathrm{Tor},
\end{align*}
    where $\tilde{\hat{c}}(\d h) = \frac{1}{T}\d h \wedge + T \iota_{\nabla h}$.
\end{lemma}

One can easily verify that $T^2 \tilde{D}_{Y, T}^2 = T^2 \tilde{D}_Y^2$. Thus (\ref{eq:ker_D_Y}) implies that 
\[ \ker T^2 \tilde{D}_{Y,T}|_{\Lambda^* Y} = \ker \tilde{D}_{Y}|_{\Lambda^* Y} = \mathrm{span}_\R \left\{ e^{-\frac{\tau \abs{y}^2}{2}} \right\} \]
is independent of $\sigma$ and $T$.

Proposition \ref{prop:F_and_G} shows that 
\begin{align} \label{eq:tr_s_F+G}
        \tr_s&\left( N_Y e^{- \sigma^2 \tilde{D}_{T/\sigma, \tau h}^2} \right) \\
        &~~~~~~~~~~~~~~~~~= \tr_s\left( N_Y F_\sigma(\sigma \tilde{D}_{T/\sigma, \tau h}) \right) + \tr_s\left( N_Y G_\sigma(\sigma \tilde{D}_{T/\sigma, \tau h}) \right) \nonumber \\
        &~~~~~~~~~~~~~~~~~= \tr_s\left( N_Y F_\sigma(D_{\sigma, T, \tau h}) \right) + \tr_s\left( N_Y H_\sigma(\tilde{D}_{T/\sigma, \tau h}) \right). \nonumber
\end{align}
First we deal with the second term. Let $H^p(\Lambda^*E)$ be the $p$-th Sobolev space of sections, and $\norm{~}_p$ be the corresponding Sobolev norm.
\begin{proposition} \label{prop:tr_s_H}
    There exist constants $C > 0$ and $T_0 \ge 1$ such that for any $\sigma \in (0,1], T \in [T_0,\infty)$, 
    \[ \left|\tr_s\left( N_Y H_\sigma(\tilde{D}_{T/\sigma, \tau h}) \right)\right| \le \frac{C}{T}. \]
\end{proposition}
\begin{proof}
    It follows from Lemma \ref{lem:D_sigma_T} that 
    \begin{align*}
        D_{\sigma, T} &= \tilde{D}_M - \frac{1}{4}\mathrm{Tor} + \frac{T}{\sigma} \tilde{D}_{Y, T/\sigma}  \\
        &:= E + \frac{T}{\sigma} \tilde{D}_{Y, T/\sigma},
    \end{align*}
    Obviously we have 
    \[ D_{\sigma, T}^2 = E^2 + \frac{T}{\sigma} [E, \tilde{D}_{Y, T/\sigma}] + \frac{T^2}{\sigma^2} \tilde{D}_{Y, T/\sigma}^2. \]
    
    Since the pair $(E, \tilde{D}_{Y, T/\sigma})$ is uniformly elliptic, \cite[Lemma 2.1]{DY23} and the standard elliptic estimate imply for any $s \in H^1(\Lambda^* E)$, 
    \begin{align*}
        \inner{(E^2 + \frac{T^2}{\sigma^2}\tilde{D}_{Y, T/\sigma}^2)s, s}_0 &= \inner{(E^2 + \tilde{D}_{Y, T/\sigma}^2)s, s}_0 + \left(\frac{T^2}{\sigma^2} - 1\right) \inner{\tilde{D}_{Y, T/\sigma}^2s, s}_0 \\
        &\ge C \left( \norm{s}_1 - \norm{s}_0\right) + \left(\frac{T^2}{\sigma^2} - 1\right) \norm{\tilde{D}_{Y, T/\sigma} s}^2_0.
    \end{align*}
    On the other hand, from the proof of \cite[Proposition 4.41]{BC} one has 
    \begin{align*}
        \frac{T}{\sigma} \inner{[E, \tilde{D}_{Y, T/\sigma}]s, s}_0 \le \frac{CT}{\sigma} \norm{\tilde{D}_{Y, T/\sigma} s}^2_0 + \frac{CT}{\sigma} \norm{s}_0^2.
    \end{align*}
    Here we have an extra term $\norm{s}_0^2$ since we no longer assume $s$ to be perpendicular to $\ker \tilde{D}_{Y, T/\sigma}$. This implies
    \begin{align*}
        \norm{D_{\sigma, T} s}_0^2 \ge C \left(\norm{s}_1^2 - \norm{s}_0^2 - \frac{T}{\sigma} \norm{s}_0^2\right) + \left(\frac{T^2}{\sigma^2} - 1 - \frac{CT}{\sigma}\right) \norm{\tilde{D}_{Y, T/\sigma} s}^2_0.
    \end{align*}
    Since we can always choose $T$ large enough so that $\frac{T^2}{\sigma^2} - 1 - \frac{CT}{\sigma} \ge T^2 - 1 - CT > 0$, it follows that 
    \begin{equation} \label{eq:s}
        \norm{s}_1^2 \le C \left( \norm{D_{\sigma, T} s}_0^2 + \frac{T}{\sigma}\norm{s}_0^2 \right).
    \end{equation}
    Using (\ref{eq:s}), Proposition \ref{prop:F_and_G} (4) and arguing as in \cite[Section 8]{BB1994} (see also \cite[Section 5]{BB1994}), we complete the proof.
\end{proof}

For the rest of the section, we aim to show that the estimate of the first term in (\ref{eq:tr_s_F+G}) can be localized. Let $p_0 \in \pi^{-1}(x_0) \subset E$ be a fixed point. Notice that $e^{\i t\sqrt{2} D_{\sigma, T, \tau h}}$ is a wave type operator. By the finite propagation speed of solutions of hyperbolic equations and the fact that $f(t) \equiv 0$ for $\abs{t} \ge \alpha$, the heat kernel of $F_\sigma(D_{\sigma, T, \tau h})$ depends only on the restriction of $D_{\sigma, T, \tau h}$ to $B_\alpha(p_0)$. For this reason, we introduce the following construction (see \cite[section 4.2]{DLM06}, \cite[section 13(e)]{BL}): 
\begin{enumerate}
    \item Set $M_0 := T_{x_0}M \cong \R^{n}$. Let $\{e_i\}$ be an orthonormal basis of $M_0$, $X = \sum_{i=1}^n x_i e_i$ be the radial vector field on $M_0$. We will also use $X$ to denote an arbitrary point in $M_0$. 
    \item Pick $\epsilon > 0$ such that $4\epsilon < \alpha$. We identify $B_{4\epsilon}^{M_0}(0) \subset M_0$ with $B_{4\epsilon}^{M}(0) \subset M$ by the exponential map $\exp_{x_0}$. Then locally $E$ can be viewed as a vector bundle over $B_{4\epsilon}^{M_0}(0)$. Moreover, we trivialize $E|_{B_{4\epsilon}^{M_0}(0)}$ by identifying $E_X$ for $X \in B_{4\epsilon}^{M_0}(0)$ with $E_{x_0}$ by parallel transport along radial geodesics.
    \item Let $\rho: \R \to [0,1]$ be a smooth cutoff function such that 
    \begin{equation*}
        \rho(t) = \begin{cases}
            1, &\text{if } \abs{t}<2; \\
            0, &\text{if } \abs{t}>4.
        \end{cases}
    \end{equation*} 
    We define a map $\rho_{\epsilon}: M_0 \to M_0$ by $\rho_{\epsilon}(X) = \rho\left(\frac{\abs{X}}{\epsilon}\right)X$. Let $g_{M_0}$ be a Riemannian metric on $M_0$ defined by $g_{M_0}(X) = g_M(\rho_\epsilon(X))$. Since $g_{M_0}$ coincides with $g_M$ on $B_{2\epsilon}^{M_0}(0)$, in the following we will replace $(M, g_M)$ with $(M_0, g_{M_0})$. 
    \item Using $\rho_\epsilon$ again, we obtain the extended trivial vector bundle $(E_0 = M_0 \times Y, g_{E_0})$ over $M_0$, satisfying $(E_0, g_{E_0})|_{B_{2\epsilon}^{M_0}(0)} = (E, g_{E})|_{B_{2\epsilon}^{M_0}(0)}$. The family of metrics $g_{\sigma, T}$ also induces metrics $g_{E_0, \sigma, T}$ on $E_0$. Similarly, using parallel transport and $\rho_\epsilon$, for any $(X, y) \in E_0$ we can identify $(\Lambda^*E_0)_{(X, y)} = (\Lambda^*M_0)_X \otimes (\Lambda^*Y)_y$ with $(\Lambda^*M)_{x_0} \otimes (\Lambda^*Y)_0$.
    \item Similarly, we extend $\ker \tilde{D}_Y$ to $K  := \Lambda^*M_0 \otimes \mathrm{span}_\R \Big\{ e^{-\frac{\tau \abs{y}^2}{2}} \Big\}$. Then $K$ is a subbundle of $\Lambda^*E_0$ satisfying 
    \begin{align*}
        K = \begin{cases}
            \Lambda^*M \otimes \ker \tilde{D}_Y|_{\Lambda^* Y}, & \text{on } B_{2 \epsilon}^{M_0}(0) \times Y; \\
            (\Lambda^*M)_{x_0} \otimes \ker \tilde{D}_{Y}|_{\Lambda^* Y} \cong (\Lambda^*M)_{x_0}, & \text{on } E_0 - B_{4 \epsilon}^{M_0}(0) \times Y.
        \end{cases}
    \end{align*}
    Let $\tilde{P}: L^2(\Lambda^* E_0) \to L^2(K)$ denote the orthogonal projection with respect to the inner product induced by $g_{E_0}$. Set $\tilde{P}^\perp = 1 - \tilde{P}$.
    \item Set $\varphi(X) = \rho(\frac{\abs{X}}{\epsilon})$. Define 
    \begin{equation} \label{eq:L_sigma_T}
        L_{\sigma, T} := \varphi^2 D_{\sigma, T, \tau h}^2 + \left(1 - \varphi^2\right) \left( \sigma^2 D_{M_0}^2 + T^2 \tilde{P}^\perp\tilde{D}_{Y}^2\tilde{P}^\perp\right),
    \end{equation}
    where $\tilde{D}_{Y, x_0}^2$ is the Witten Laplacian on the fiber $\{x_0\} \times Y$.
    \item From Lemma \ref{lem:BLW_formula_new}, we see that $L_{\sigma, T}$ fails to be uniformly elliptic as $\sigma \to 0$. To correct this, we consider
    \begin{equation} \label{eq:L_sigma_T^tilde}
        \tilde{L}_{\sigma, T} := G_2 \circ L_{\sigma, T} \circ G_2^{-1},
    \end{equation}
    where $G_2$ is the rescaling defined in Section \ref{subsect:proof_int_res_I_4_1}.
\end{enumerate}
As a result of the construction above, 
\begin{equation} \label{eq:tr_F_D_and_F_L}
    \begin{split}
        \tr_s\left( N_Y F_\sigma(D_{\sigma, T, \tau h}) \right) &= \tr_s\left( N_Y \tilde{F}_\sigma(D_{\sigma, T, \tau h}^2) \right) \\
        &= \tr_s\left( N_Y \tilde{F}_\sigma(\tilde{L}_{\sigma, T}) \right).
    \end{split}
\end{equation}

Again we decompose $\tilde{L}_{\sigma, T}$ with respect to the splitting $L^2(\Lambda^* E_0) = L^2(K)^\perp \oplus L^2(K)$ as 
\begin{equation}
    \tilde{L}_{\sigma, T} = \begin{pmatrix}
        \tilde{L}_{\sigma, T, 1} & \tilde{L}_{\sigma, T, 2} \\
        \tilde{L}_{\sigma, T, 2}' & \tilde{L}_{\sigma, T, 3}
    \end{pmatrix}.
\end{equation}

\begin{theorem} \label{thm:L_1_L_2_L_3}
    Let $D_{Y,\sigma}$ denote the standard Dirac operator on the fiber $Y_{(\sigma X, y)}$. Then as $T \to \infty$, 
    \begin{align*}
        \tilde{L}_{\sigma, T, 1} &= T^2 L_1 + O \left(T\right), \\
        \tilde{L}_{\sigma, T, 2} &= T L_2 + O \left(1\right), \\
        \tilde{L}_{\sigma, T, 2}' &= T L_2' + O \left(1\right), \\
        \tilde{L}_{\sigma, T, 3} &=  L_3 + O \left(\frac 1T\right),
    \end{align*}
    where 
    \begin{align*}
        L_1 &=  \tilde{P}^\perp  \tilde{D}_{Y}^2 \tilde{P}^\perp ,\\
        L_2 &=  \tilde{P}^\perp \Big\{ - \frac{\varphi^2(\sigma X)}{2} \Big[ \nabla^E_{f_i},  \frac{\sigma}{2} \sum_{j,k} \inner{S(f_i)f_j, e_k} \left( c(\tilde{f}^j) c(\tilde{e}^k) - \hat{c}(\tilde{f}^j)\hat{c}(\tilde{e}^k) \right)\Big] \\
        &\hspace{1cm}+ \frac{\varphi^2(\sigma X)}{4} R^{E}_{i\bar{j}\bar{k}\bar{l}} c(\tilde{e}^i) c(\tilde{f}^j) \hat{c}(\tilde{f}^k) \hat{c}(\tilde{f}^l) \Big\} \tilde{P},\\
        L_2' &= L_2 \text{ with } \tilde{P} \text{ and } \tilde{P}^\perp \text{ interchanged},\\
        L_3 &=  \tilde{P} \Big\{D_{M_0}^2  + B(\sigma) \Big\} \tilde{P}.
    \end{align*}
    Here $B(\sigma)$ is a first order differential operator, for which the differential is only along the base direction.
\end{theorem}
\begin{proof}
    The proof follows directly from Lemma \ref{lem:BLW_formula_new}, Lemma \ref{lem:S_sigma_T}, (\ref{eq:L_sigma_T}), and (\ref{eq:L_sigma_T^tilde}).
\end{proof}

\subsection{Uniform estimate on heat kernel}

Let $H^p(\Lambda^*M_0\otimes\Lambda^*Y) = H^p(\Lambda^* E_0)$ denote the $p$-th Sobolev space of sections of $\Lambda^* E_0$. Obviously $H^0(\Lambda^*M_0 \otimes \Lambda^*Y) = L^2(\Lambda^*E_0)$.

To keep track of the dependence of $\sigma$ and $ T$, we introduce weighted Sobolev norms that fit into our construction. 
\begin{definition} \label{def:norm_sigma_T}
    Let $\nabla = \nabla^{M_0} + \nabla^Y$ be a connection on $E_0$, where $\nabla^{M_0}$ and $\nabla^Y$ are the pullback of the Levi-Civita connections on $M_0$ and $Y$ respectively.
    Set 
    \begin{equation} \label{eq:psi_sigma}
        \psi_\sigma(X) := 1 + (1 + \abs{X}^2)^{\frac{1}{2}} \varphi\left(\frac{\sigma X}{2}\right).
    \end{equation}
    For any $s \in H^0(\Lambda^qM_0 \otimes \Lambda^*Y)$ and $s', s'' \in H^1(\Lambda^*E_0)$, we define 
    \begin{align*}
        \abs{s}_{\sigma; 0}^2 &= \frac{1}{(2\pi)^n}\int_{E_0} \abs{s}^2 \psi_\sigma^{2(n-q)} dvol_{E_0}, \\
        \abs{s'}_{\sigma,T;1}^2 &= T^2 \abs{\tilde{P}^\perp s'}_{\sigma;0}^2 + \abs{\tilde{P} s'}_{\sigma;0}^2 + \sum_{i=1}^n \abs{\nabla_{e_i}s'}_{\sigma; 0}^2 + T^2 \sum_{j=1}^k \abs{\nabla_{f_j} \tilde{P}^\perp s'}_{\sigma; 0}^2,\\
        \abs{s''}_{\sigma, T; -1} &= \sup_{v \ne 0 \in H^1(\Lambda^* E_0) } \frac{\abs{\inner{s'', v}_{\sigma; 0}}}{\abs{v}_{\sigma, T; 1}},
    \end{align*}
    where $\inner{~,~}_{\sigma; 0}$ denotes the Hermitian inner product associated with $\abs{~}_{\sigma; 0}^2$.

    If $L : \big(H^m(\Lambda^* E_0), \abs{~}_{\sigma, T; m}\big) \to \big(H^{m'}(\Lambda^* E_0), \abs{~}_{\sigma, T; m'}\big)$ is a bounded linear operator, then we use $\norm{L}_{\sigma,T;m,m'}$ to denote its operator norm.
\end{definition}
The next result follows from \cite[Theorem 13.27]{BL}.
\begin{lemma} \label{lem:Re_Im_bd}
    There exist constants $C, C' > 0$ such that for any $\sigma \in (0,1]$, $T \in [1, \infty)$, and $s, s' \in H^1(\Lambda^*E_0)$, we have
    \begin{align*}
        \mathrm{Re}\inner{\tilde{L}_{\sigma, T} s, s}_{\sigma; 0} &\ge C \abs{s}_{\sigma,T;1}^2 - C' \abs{s}_{\sigma;0}^2, \\
        \left|\mathrm{Im}\inner{\tilde{L}_{\sigma, T} s, s}_{\sigma; 0}\right| &\le C \abs{s}_{\sigma,T;1} \abs{s}_{\sigma;0}, \\
        \abs{\inner{\tilde{L}_{\sigma, T} s, s'}_{\sigma; 0}} &\le C \abs{s}_{\sigma,T;1} \abs{s'}_{\sigma,T;1}.
    \end{align*}
    In particular, $\norm{\tilde{L}_{\sigma, T}}_{\sigma, T; 1, -1} \le C$.
\end{lemma}

\begin{proposition} \label{prop:lambda-L_-1_to_1}
    Set $\tilde{\Gamma} := \{\lambda = a + \i b \in \C \mid a = \delta b^2 - A\}$. Then there exist constants $C > 0$, $0 < \delta \ll 1$, $A \gg 1$ such that for any $\sigma \in (0,1]$, $T \in [1, \infty)$, and $\lambda \in \tilde{\Gamma}$, 
    \begin{enumerate}
        \item the resolvent $(\lambda - \tilde{L}_{\sigma, T})^{-1}$ exists;
        \item \(\begin{aligned}[t]
            &\norm{(\lambda - \tilde{L}_{\sigma, T})^{-1}}_{\sigma,T;0,0} \le C, \\
            &\norm{(\lambda - \tilde{L}_{\sigma, T})^{-1}}_{\sigma,T;-1,1} \le C (1 + \abs{\lambda}^2).
        \end{aligned}\)
    \end{enumerate}
\end{proposition}
\begin{proof}
    By Lemma \ref{lem:Re_Im_bd}, we find that for any $s \in H^1(\Lambda^* E_0)$ and $\lambda = a + \i b \in \tilde{\Gamma}$,
    \begin{align*}
        \Big| \inner{(\tilde{L}_{\sigma, T} &- \lambda) s, s }_{\sigma; 0} \Big| \\
        &=  \left| \left(\mathrm{Re}\inner{\tilde{L}_{\sigma, T} s, s }_{\sigma; 0} - a \abs{s}_{\sigma; 0}^2\right) + \i \left(\mathrm{Im}\inner{\tilde{L}_{\sigma, T} s, s }_{\sigma; 0} - b \abs{s}_{\sigma; 0}^2\right) \right| \\
        &\ge \max \left\{ \mathrm{Re}\inner{\tilde{L}_{\sigma, T} s, s }_{\sigma; 0} - a \abs{s}_{\sigma; 0}^2,  \left|  \mathrm{Im}\inner{\tilde{L}_{\sigma, T} s, s }_{\sigma; 0} - b \abs{s}_{\sigma; 0}^2 \right|\right\} \\
        &\ge \max \left\{ C\abs{s}_{\sigma, T; 1}^2 - (C' + a)\abs{s}_{\sigma; 0}^2,  -C \abs{s}_{\sigma, T;1} \abs{s}_{\sigma;0} + \abs{b} \abs{s}_{\sigma;0}^2\right\} \\
        &= \max \left\{  C \alpha^2  - (C' + a),   -C \alpha + \abs{b}  \right\} \cdot \abs{s}_{\sigma; 0}^2,
    \end{align*}
    where $\alpha= \frac{\abs{s}_{\sigma, T;1}}{\abs{s}_{\sigma;0}} \ge 1$. It is easy to see that if $\delta>0$ is sufficiently small and $A>1$ is sufficiently large, there exists a constant $C'' > 0$ such that 
    \[ \max \left\{  C \alpha^2  - (C' + a),   -C \alpha + \abs{b}  \right\} \ge C'' , ~~~ \forall \lambda \in \tilde{\Gamma}.\]
    
    Using this and the Cauchy-Schwarz inequality, we see that 
    \begin{equation} \label{eq:0,0}
        \abs{(\tilde{L}_{\sigma, T} - \lambda) s}_{\sigma; 0} \ge C'' \abs{s}_{\sigma; 0}.
    \end{equation}
    This proves (1) as well as the first inequality in (2). 

    Now, by the first inequality in Lemma \ref{lem:Re_Im_bd}, for $\lambda_0 \in \R$ with $\lambda_0 \le -2 C'$ and $s\in H^{1}(\Lambda^*E)$, we have 
    \begin{align*}
        \Big| \inner{(\tilde{L}_{\sigma, T} - \lambda_0) s, s }_{\sigma; 0} \Big| 
        &\ge \mathrm{Re}\inner{\tilde{L}_{\sigma, T} s, s }_{\sigma; 0} - \lambda_0 \abs{s}_{\sigma; 0}^2 \\
        &\ge C \abs{s}_{\sigma,T;1}^2 - (C' + \lambda_0) \abs{s}_{\sigma;0}^2 \\
        &\ge C \abs{s}_{\sigma,T;1}^2.
    \end{align*}
    As a consequence,
    \begin{equation} \label{eq:-1,1}
        \norm{(\lambda_0 - \tilde{L}_{\sigma, T})^{-1}}_{\sigma, T; -1, 1} \le \frac{1}{C}.
    \end{equation}

    For any $\lambda \in \tilde{\Gamma}$, the first resolvent identity yields
    \begin{equation} \label{eq:res_ide}
        (\lambda - \tilde{L}_{\sigma, T})^{-1} = (\lambda_0 - \tilde{L}_{\sigma, T})^{-1} - (\lambda - \lambda_0) (\lambda - \tilde{L}_{\sigma, T})^{-1} (\lambda_0 - \tilde{L}_{\sigma, T})^{-1}.
    \end{equation}

    Notice that (\ref{eq:0,0}) implies 
    \begin{equation} \label{eq:1,0}
        \norm{(\lambda - \tilde{L}_{\sigma, T})^{-1}}_{\sigma, T; 1, 0} \le \frac{1}{C''}.
    \end{equation}
    Using (\ref{eq:-1,1}), (\ref{eq:res_ide}) and (\ref{eq:1,0}), we find that 
    \begin{equation} \label{eq:-1, 0}
        \norm{(\lambda - \tilde{L}_{\sigma, T})^{-1}}_{\sigma, T; -1, 0} \le \frac{1}{C} + \frac{1}{C C''} \abs{\lambda - \lambda_0}.
    \end{equation}

    Now we interchange the last two factors $(\lambda - \tilde{L}_{\sigma, T})^{-1}$ and $(\lambda_0 - \tilde{L}_{\sigma, T})^{-1}$ in (\ref{eq:res_ide}) and apply (\ref{eq:-1,1}) and (\ref{eq:-1, 0}) to obtain 
    \begin{equation}
        \norm{(\lambda - \tilde{L}_{\sigma, T})^{-1}}_{\sigma, T; -1, 1} \le \frac{1}{C} + \frac{\abs{\lambda - \lambda_0}}{C} \left( \frac{1}{C} + \frac{1}{C C''}\abs{\lambda - \lambda_0} \right) \le C''' (1 + \abs{\lambda}^2).
    \end{equation}
\end{proof}

\begin{proposition} \label{prop:QQQQQ}
    Let $m \in \N$, $\mathcal{Q} = \{ \tilde{P} \nabla_{e_i} \tilde{P}, \tilde{P}^\perp \nabla_{e_i} \tilde{P}^\perp, \tilde{P}^\perp \nabla_{f_j} \tilde{P}^\perp \}^{i = 1, \cdots, n }_{j = 1, \cdots, k}$. Then there exists a constant $C_m > 0$ such that for any $\sigma \in (0,1]$, $ T \in [1, \infty)$, and $ Q_1, \cdots, Q_m \in \mathcal{Q}$, 
   \begin{equation} \label{eq:QQQQQQ}
       \left\|[Q_1,[Q_2, \cdots, [Q_m, \tilde{L}_{\sigma, T}] \cdots]]\right\|_{\sigma, T; 1, -1} \le C_m.
   \end{equation}
\end{proposition}
\begin{proof}
    We first assume $m=1$. By multiplying by a constant factor if necessary, we may assume that 
    \begin{equation*}
        \int_Y e^{-\tau \abs{y}^2} \d y = 1.
    \end{equation*}
    Then for $s \in L^2(\Lambda^*E_0)$, the projection $\tilde{P}s$ is given by 
    \begin{equation} \label{eq:P^tilde}
        \tilde{P}s = e^{-\frac{\tau \abs{y}^2}{2}}\int_Y e^{-\frac{\tau \abs{y}^2}{2}} s(X, y) \d vol_Y.
    \end{equation}
    By Definition \ref{def:norm_sigma_T} and (\ref{eq:P^tilde}), we see that 
    \begin{equation}\label{eq:vanish_nabla_e_f_P}
        [\nabla_{e_i}, \tilde{P}] = 0, ~~~ \tilde{P} \nabla_{f_i}\tilde{P} = 0, ~~~ [\varphi^2, \tilde{P}] = 0.
    \end{equation}

    Let $Q_1 = \tilde{P} \nabla_{e_i} \tilde{P}, Q_2 = \tilde{P}^\perp \nabla_{e_i} \tilde{P}^\perp, Q_3 = \tilde{P}^\perp \nabla_{f_j} \tilde{P}^\perp$. It follows from (\ref{eq:vanish_nabla_e_f_P}) that $[Q_1, \tilde{L}_{\sigma, T}]$ can be expressed as 
    \begin{equation}
        \sum_{k,l} a_{k,l}(\sigma, T) \nabla_{e_k} \nabla_{e_l} + \sum_{k} b_{k}(\sigma, T) \nabla_{e_k} + T \sum_{\bar{l}} b'_{l}(\sigma, T) \nabla_{f_l} \tilde{P}^\perp + c(\sigma, T),
    \end{equation}
    and $a_{k,l}, b_{k}, b'_{l}, c$, together with their derivatives, are uniformly bounded for any $(X, y) \in E_0, \sigma \in (0,1]$ and $T \in [1, \infty)$. Therefore by Definition \ref{def:norm_sigma_T} and integration by parts, (\ref{eq:QQQQQQ}) is verified for $[Q_1, \tilde{L}_{\sigma, T}]$.

    Similarly, for $i = 2, 3$ we have 
    \begin{equation} \label{eq:[Q,L]_type}
        \begin{split}
            [Q_i, \tilde{L}_{\sigma, T}] &= \sum_{k,l} a_{k,l}(\sigma, T) \nabla_{e_k} \nabla_{e_l} + T^2\sum_{\bar{k}, \bar{l}} a'_{k,l}(\sigma, T)  \tilde{P}^\perp \nabla_{f_k} \nabla_{f_l} \tilde{P}^\perp \\
            & \hspace{.5cm}+ \sum_{k} b_{k}(\sigma, T) \nabla_{e_k} + T \sum_{l} b'_{l}(\sigma, T) \nabla_{f_l} \tilde{P}^\perp + c(\sigma, T).
        \end{split}
    \end{equation}
    The coefficients are also uniformly bounded as before. Thus we have proved (\ref{eq:QQQQQQ}) for $m=1$.
    
        


    In general, using an iteration argument as in the proof of \cite[Proposition 4.9]{DLM06}, $[Q_1,[Q_2, \cdots, [Q_m, \tilde{L}_{\sigma, T}] \cdots]]$ has the same type (\ref{eq:[Q,L]_type}) as $[Q_i, \tilde{L}_{\sigma, T}]$. This completes the proof.
\end{proof}

\begin{definition} \label{def:norm_QQQQQ}
    Let $m \in \N$. We define a norm $\norm{~}^2_{\sigma; m}$ on $H^m(\Lambda^*E_0)$ by 
    \[ \norm{s}^2_{\sigma; m} =  \sum_{l = 0}^m \sum_{Q_i \in \mathcal{Q}} \abs{Q_1 \cdots Q_l s}_{\sigma; 0}^2.\]
    If $A: (H^m(\Lambda^*E_0), \norm{~}^2_{\sigma; m}) \to (H^{m'}(\Lambda^*E_0), \norm{~}^2_{\sigma; m'})$ is a bounded linear operator, we use $\vertiii{A}_{\sigma;m,m'}$ to denote the corresponding operator norm.
\end{definition}

\begin{proposition} \label{prop:rel_m_m+1}
    For any $m \in \N$, $\sigma \in (0,1]$, $T \in [1, \infty)$, and $\lambda \in \tilde{\Gamma}$, the resolvent $(\lambda - \tilde{L}_{\sigma, T})^{-1}$ maps $H^m(\Lambda^*E_0)$ to $H^{m+1}(\Lambda^*E_0)$. Moreover, there exist $N \in \N$ and $C_m > 0$ such that 
    \[ \vertiii{(\lambda - \tilde{L}_{\sigma, T})^{-1}}_{\sigma; m,m+1} \le C_m (1 + \abs{\lambda})^N. \]
\end{proposition}
\begin{proof}
    First, we notice that for $Q_i \in \mathcal{Q}$, the operator $Q_1 \cdots Q_m (\lambda - \tilde{L}_{\sigma, T})^{-1}$ can be expressed, using commutators, as a linear combination of operators of the type 
    \[ [Q_1, [Q_2, \cdots [Q_{l}, (\lambda - \tilde{L}_{\sigma, T})^{-1}]\cdots]]Q_{l+1} \cdots Q_m \]
    for some $l \le m$. Let $\mathcal{C} = \{ [Q_{i_1}, [Q_{i_1}, \cdots [Q_{i_p}, (\lambda - \tilde{L}_{\sigma, T})^{-1}]\cdots]] \}$ be the set of commutators. Then each $[Q_1, [Q_2, \cdots [Q_{l}, (\lambda - \tilde{L}_{\sigma, T})^{-1}]\cdots]]$ has the form 
    \[ (\lambda - \tilde{L}_{\sigma, T})^{-1} R_1 (\lambda - \tilde{L}_{\sigma, T})^{-1} R_2 \cdots R_l (\lambda - \tilde{L}_{\sigma, T})^{-1}, \]
    where $R_i \in \mathcal{C}$.

    By Proposition \ref{prop:lambda-L_-1_to_1} and Proposition \ref{prop:QQQQQ}, there exists $C > 0$ such that 
    \[ \norm{R_i}_{\sigma, T; 1, -1} \le C,~~ \norm{(\lambda - \tilde{L}_{\sigma, T})^{-1}}_{\sigma, T; -1, 1} \le C(1 + \abs{\lambda}^2), \norm{Q_i}_{\sigma, T; 1, 0} \le C. \]
    Since we obviously have $\norm{s}_{\sigma; 1} \le C \abs{s}_{\sigma, T; 1}$, the assertion follows.
\end{proof}

Let $\tilde{K}_{\sigma,T}$ denote the heat kernel of $\tilde{F}_\sigma(\tilde{L}_{\sigma, T})$. Using the results above, we obtain the following uniform estimate for $\tilde{K}_{\sigma,T}$.

\begin{theorem} \label{thm:unif_bound}
    For any $\sigma \in (0,1]$ and $T \in [1,\infty)$, the heat kernel $\tilde{K}_{\sigma,T}$, together with all its derivatives, is uniformly bounded on compact subsets. 
\end{theorem}
\begin{proof}
    In the whole proof, $C$ (or $C_m$) will denote a constant that is independent of $\sigma, T$, and may vary from line to line.

    Let $\lambda \in \tilde{\Gamma}$. By Proposition \ref{prop:lambda-L_-1_to_1}, we can express $\tilde{F}_\sigma(\tilde{L}_{\sigma, T})$ as 
    \begin{equation} \label{eq:F_L_m}
        \tilde{F}_\sigma(\tilde{L}_{\sigma, T}) = \frac{(-1)^{m-1} (m-1)! 2^{m-1}}{2\pi\i} \int_{\tilde{\Gamma}}\tilde{F}_\sigma(\lambda) (\lambda - \tilde{L}_{\sigma, T})^{-m} \d \lambda
    \end{equation}
    for any $m \in \N$. Using (\ref{eq:F_L_m}) and applying Proposition \ref{prop:rel_m_m+1} for $m$ times, we find that 
    \[ \vertiii{\tilde{F}_\sigma(\tilde{L}_{\sigma, T})}_{\sigma, T; 0, m} \le C_m .\]

    This implies
    \begin{equation}
        \vertiii{\Delta^l\tilde{F}_\sigma(\tilde{L}_{\sigma, T})}_{\sigma, T; 0, 0} \le C, ~~ \forall l \in \N.
    \end{equation}
    By an adjoint argument, we find that 
    \begin{equation}
        \vertiii{\tilde{F}_\sigma(\tilde{L}_{\sigma, T}) \Delta^{l'}}_{\sigma, T; 0, 0} \le C, ~~ \forall l' \in \N.
    \end{equation}
    As a result, for any $l, l' \in \N$ we have 
    \begin{equation} \label{eq:l_F_l'_0,0}
        \vertiii{\Delta^{l}\tilde{F}_\sigma(\tilde{L}_{\sigma, T}) \Delta^{l'}}_{\sigma, T; 0, 0} \le C.
    \end{equation}

    Now we restrict our discussion to $D = \{\abs{X} \le 1, \abs{y} \le 1\}$. Let $\abs{~}_D$ denote the $L^2$-norm of $L^2(\Lambda^*E_0|_{D})$, and $\norm{~}_D$ denote the corresponding operator norm. Then using the trivial embedding $L^2(\Lambda^*E_0|_{D}) \hookrightarrow L^2(\Lambda^*E_0)$, we see that for $s \in L^2(\Lambda^*E_0|_{D})$, 
    \begin{equation} \label{eq:norm_D_0}
        \abs{s}_D \le \abs{s}_{\sigma; 0} \le C \abs{s}_D.
    \end{equation}

    (\ref{eq:l_F_l'_0,0}) and (\ref{eq:norm_D_0}) imply an $L^2$-bound in the standard sense: 
    \begin{equation}
        \norm{\Delta^{l}\tilde{F}_\sigma(\tilde{L}_{\sigma, T}) \Delta^{l'}}_D \le C.
    \end{equation}
    Thus by Sobolev's inequalities, we find that for any $q, q' \in \N$, 
    \begin{equation}
        \sup_{(p, p') \in D \times D} \abs{\Delta^q_p \Delta^{q'}_{p'} \tilde{K}_{\sigma, T}(p, p')} \le C.
    \end{equation}
    This completes the proof.
    
\end{proof}

\subsection{Uniform estimate on supertrace}
Now we are ready to prove Theorem \ref{thm:interm_result_I_4_3}.

In view of (\ref{eq:tr_s_F+G}), Proposition \ref{prop:tr_s_H}, and (\ref{eq:tr_F_D_and_F_L}), Theorem \ref{thm:interm_result_I_4_3} is a direct consequence of the following theorem, for which the proof will be delayed to the end of this section.
\begin{theorem} \label{thm:last_one}
    There exist constants $\delta > 0$, $C > 0$, such that for any $\sigma \in (0,1]$, $T \in [1.\infty)$, 
    \begin{equation}
        \left| \frac{1}{T}\tr_s \left( N_Y \tilde{F}_\sigma(\tilde{L}_{\sigma, T})\right) \right| \le \frac{C}{T^{1 + \delta}}.
    \end{equation}
\end{theorem}

Recall that under the splitting $L^2(\Lambda^*E_0) = L^2(K)^\perp \oplus L^2(K)$, the operator $\tilde{L}_{\sigma, T}$ decomposes as 
\begin{equation}
    \tilde{L}_{\sigma, T} = \begin{pmatrix}
        \tilde{L}_{\sigma, T, 1} & \tilde{L}_{\sigma, T, 2} \\
        \tilde{L}_{\sigma, T, 2}' & \tilde{L}_{\sigma, T, 3}
    \end{pmatrix}.
\end{equation}
If we write $(\lambda - \tilde{L}_{\sigma, T})^{-1} = \begin{pmatrix}
    U_1 & U_2 \\
    U_2' & U_3
\end{pmatrix}$, then 
\begin{equation}
    U_3 = \left((\lambda - \tilde{L}_{\sigma, T, 3}) - \tilde{L}_{\sigma, T, 2}' (\lambda - \tilde{L}_{\sigma, T, 1})^{-1} \tilde{L}_{\sigma, T, 2}\right)^{-1}.
\end{equation}

\begin{definition}
    Let $\Theta_\sigma: L^2(\Lambda^*E_0) \to L^2(\Lambda^*E_0)$ be the operator 
    \begin{equation}
        \Theta_\sigma = \begin{pmatrix}
            0 & 0 \\
            0 & L_3 - L_2' L_1^{-1} L_2
        \end{pmatrix},
    \end{equation}
    where $L_i$'s are the operators given in Theorem \ref{thm:L_1_L_2_L_3}.
\end{definition}

\begin{remark} \label{rmk:ident_Theta_D_M}
    By Lemma \ref{lem:D_sigma_T}, if we let $E = \tilde{D}_M - \frac{1}{4}\mathrm{Tor}$, then 
    \begin{align*}
        \tilde{L}_{\sigma, T}^2 &= \varphi^2(\sigma X) \left[E^2 + T (E \tilde{D}_{Y, T} + \tilde{D}_{Y, T} E ) + T^2 \tilde{D}_{Y, T}^2\right] \\
        &+ (1 - \varphi^2(\sigma X)) \left(D_{M_0}^2 + T^2 \tilde{P}^\perp\tilde{D}_{Y}^2\tilde{P}^\perp\right).
    \end{align*}

    It is clear that $\tilde{D}_{Y, T} = \tilde{D}_{Y}^2 + O(T)$. Suppose $E = E_1 + O(T)$. Therefore, compared to Theorem \ref{thm:L_1_L_2_L_3}, we find that
    \begin{align*}
        L_1 &= \tilde{P}^\perp \tilde{D}_{Y}^2 \tilde{P}^\perp, \\
        L_2 &= \tilde{P}^\perp \Big[ \varphi^2(\sigma X)\tilde{D}_{Y} E_1  \Big] \tilde{P}, \\
        L_2' &= \tilde{P} \Big[ \varphi^2(\sigma X)E_1 \tilde{D}_{Y}   \Big] \tilde{P}^\perp, \\
        L_3 &= \tilde{P}\Big[ \varphi^2(\sigma X)E_1^2  \Big] \tilde{P}.
    \end{align*}

    When $\abs{X} \le \frac{\epsilon}{\sigma}$, $\varphi^2(\sigma X) \equiv 1$, and thus 
    \begin{equation}
        L_3 - L_2' L_1^{-1} L_2 = \tilde{P} E_1^2  \tilde{P} - \tilde{P} E_1 \tilde{P}^\perp E_1  \tilde{P} = (\tilde{P} E_1  \tilde{P})^2.
    \end{equation}
    Moreover, if we view $D_M^2$ as an operator acting on $\Lambda^*M \otimes \ker \tilde{D}_Y|_{\Lambda^* Y}$, then $(\tilde{P} E_1  \tilde{P})^2 = G_2 \circ (\sigma D_M)^2 \circ G_2^{-1}$ for $X \in B^{M_0}_{\epsilon/\sigma}(0)$.
\end{remark}

\begin{proposition} \label{prop:last_prop}
    There exists a constant $C > 0$ such that for any $\sigma \in (0,1]$ and $T \in [1, \infty)$, we have 
    \begin{enumerate}
        \item $\norm{\tilde{P}^\perp \tilde{F}_\sigma(\tilde{L}_{\sigma, T}) \tilde{P}^\perp}_{\sigma;0, 0} \le \frac{C}{T}$;
        \item $\norm{\tilde{P}^\perp\tilde{F}_\sigma(\tilde{L}_{\sigma, T})\tilde{P}}_{\sigma;0, 0} \le \frac{C}{T}$;
        \item $\norm{\tilde{P}\tilde{F}_\sigma(\tilde{L}_{\sigma, T})\tilde{P}^\perp}_{\sigma;0, 0} \le \frac{C}{T}$;
        \item $\norm{\tilde{P}\tilde{F}_\sigma(\tilde{L}_{\sigma, T})\tilde{P} - \tilde{P} \tilde{F}_\sigma (\Theta_\sigma) \tilde{P}}_{\sigma;0, 0} \le \frac{C}{T^{1/2}}$.
    \end{enumerate}
    In particular, $\norm{\tilde{F}_\sigma(\tilde{L}_{\sigma, T}) - \tilde{P} \tilde{F}_\sigma (\Theta_\sigma) \tilde{P}}_{\sigma;0, 0} \le \frac{C}{T^{1/2}}$.
\end{proposition}
\begin{proof}
    By Definition \ref{def:norm_sigma_T} and Proposition \ref{prop:lambda-L_-1_to_1}, we find that for any $\lambda \in \tilde{\Gamma}$ and $s \in L^2(\Lambda^*E_0)$, 
    \begin{align} \label{eq:p_perp_rel}
            \abs{\tilde{P}^\perp (\lambda - \tilde{L}_{\sigma, T})^{-1}s}_{\sigma;0} &\le \frac{C}{T} \abs{(\lambda - \tilde{L}_{\sigma, T})^{-1}s}_{\sigma,T;1} \\
            &\le \frac{C}{T} (1 + \abs{\lambda}^2) \abs{s}_{\sigma,T;-1} \nonumber \\
            &\le \frac{C}{T} (1 + \abs{\lambda}^2) \abs{s}_{\sigma; 0}. \nonumber
    \end{align}
    Using Proposition \ref{prop:F_and_G}(5), (\ref{eq:p_perp_rel}), and the fact that $\tilde{F}_\sigma(\tilde{L}_{\sigma, T}) = \frac{1}{2\pi \i} \int_{\tilde{\Gamma}} \tilde{F}_\sigma(\lambda) (\lambda  - \tilde{L}_{\sigma, T})^{-1}\d \lambda$, we obtain the first three inequalities.

    To derive the last inequality, notice that 
    \begin{equation} \label{eq:L_Theta_1}
        \begin{split}
        &\tilde{P}\tilde{F}_\sigma(\tilde{L}_{\sigma, T})\tilde{P} - \tilde{P}\tilde{F}_\sigma (\Theta_\sigma)\tilde{P} \\
        & \hspace{1cm}= \frac{1}{2\pi \i}\int_{\tilde{\Gamma}} \tilde{F}_\sigma(\lambda) \left[\tilde{P}(\lambda - \tilde{L}_{\sigma, T})^{-1}\tilde{P} - \tilde{P}(\lambda - \Theta_\sigma)^{-1}\tilde{P}\right] d \lambda,
        \end{split}
    \end{equation}
    and 
    \begin{equation} \label{eq:U_3_Theta}
        \begin{split}
            &\tilde{P}(\lambda - \tilde{L}_{\sigma, T})^{-1}\tilde{P} - \tilde{P}(\lambda - \Theta_\sigma)^{-1}\tilde{P} \\
            &= U_3 - \tilde{P}(\lambda - \Theta_\sigma)^{-1}\tilde{P} \\
            &= - U_3 \left[ \tilde{L}_{\sigma, T, 3} + \tilde{L}_{\sigma, T, 2}' (\lambda - \tilde{L}_{\sigma, T, 1})^{-1} \tilde{L}_{\sigma, T, 2} - L_3 + L_2' L_1^{-1} L_2 \right] \tilde{P}(\lambda - \Theta_\sigma)^{-1}\tilde{P}.
        \end{split}
    \end{equation}

    Now we estimate the operators in the last line of (\ref{eq:U_3_Theta}) separately.
    \begin{enumerate}
        \item[(a)] Since $U_3 = \tilde{P}(\lambda - \tilde{L}_{\sigma, T})^{-1}\tilde{P}$, by Proposition \ref{prop:lambda-L_-1_to_1} we deduce that 
    \end{enumerate} \label{eq:U_3}
        \begin{equation} \label{eq:est_1}
            \norm{U_3}_{\sigma,T;0,0} \le C, ~\norm{U_3}_{\sigma,T;-1,1} \le C (1 + \abs{\lambda}^2).
        \end{equation}
    \begin{enumerate}
        \item[(b)] By a similar argument as in the proof of Theorem \ref{prop:lambda-L_-1_to_1}, if $\delta, A, \sigma, T, \lambda$ 
    \end{enumerate}    
        satisfy the conditions therein, then 
        \begin{equation} \label{eq:est_2}
            \begin{split}
                &\norm{\tilde{P}(\lambda - \Theta_\sigma)^{-1}\tilde{P}}_{\sigma,T;0,0} \le C, \\
                &\norm{\tilde{P}(\lambda - \Theta_\sigma)^{-1}\tilde{P}}_{\sigma,T;-1,1} \le C (1 + \abs{\lambda}^2).
            \end{split}
        \end{equation}
    \begin{enumerate}
        \item[(c)] It is easy to see that $\tilde{L}_{\sigma, T, 3} - L_3 = \frac{1}{T}R$, where $R$ is a first order 
    \end{enumerate}   
        differential operator for which the differential only involves $\nabla_{e_i}$. Thus for any $s, s' \in L^2(\Lambda^*E_0)$, we have 
        \begin{align*}
            \left|\inner{(\tilde{L}_{\sigma, T, 3} - L_3)s, s'}_{\sigma; 0}\right| \le \frac{1}{T} \abs{s}_{\sigma, T;1} \abs{s'}_{\sigma, T; 1},
        \end{align*}
        which implies
        \begin{equation} \label{eq:est_3}
            \lVert\tilde{L}_{\sigma, T, 3} - L_3\rVert_{\sigma,T;1,-1} \le \frac{C}{T}
        \end{equation}
    \begin{enumerate}
        \item[(d)] Let $\tilde{P}_\sigma : L^2(\Lambda^*E_0) \to L^2(K)^\perp$ be the projection map with respect  
    \end{enumerate}   
        to the inner product $\inner{~,~}_{\sigma; 0}$. Define $\tilde{L}'_{\sigma, T, 1} = \tilde{P}_\sigma^\perp \tilde{L}_{\sigma, T} \tilde{P}^\perp$. Then clearly we have 
        \begin{equation}
            \tilde{L}'_{\sigma, T, 1} - \tilde{L}_{\sigma, T, 1} = (\tilde{P}_\sigma^\perp - \tilde{P}^\perp) \tilde{L}_{\sigma, T} \tilde{P}^\perp.
        \end{equation}

        For any $s, s' \in L^2(\Lambda^*E_0)$, 
        \begin{equation} \label{eq:L'_1}
            \inner{ \tilde{L}'_{\sigma, T, 1} \tilde{P}^\perp  s,  \tilde{P}^\perp s'}_{\sigma; 0} = \inner{ \tilde{L}_{\sigma, T} \tilde{P}^\perp s,  \tilde{P}^\perp s'}_{\sigma; 0}.
        \end{equation}
        Using (\ref{eq:L'_1}) and arguing as in Proposition \ref{prop:lambda-L_-1_to_1}, one can show that for $\lambda \in \tilde{\Gamma}$, $(\lambda - \tilde{L}'_{\sigma, T, 1})^{-1}$ maps $H^{-1}(\Lambda^* E_0)$ to $H^{1}(\Lambda^* E_0)$, and
        \begin{equation} \label{eq:lam-L'^-1}
            \norm{(\lambda - \tilde{L}'_{\sigma, T, 1})^{-1}}_{\sigma, T; -1, 1} \le C (1 + \abs{\lambda}^2).
        \end{equation}

        On the other hand, if we let $\tilde{P}^{*,\sigma}$ denote the adjoint of $\tilde{P}$ with respect to $\inner{~,~}_{\sigma; 0}$, then by Proposition \ref{lem:Re_Im_bd} we have 
    \begin{align}  \label{eq:L'-L}
            \bigg|\inner{(\tilde{L}'_{\sigma, T, 1} &- \tilde{L}_{\sigma, T, 1})\tilde{P}^\perp s, \tilde{P}^\perp s'}_{\sigma;0}\bigg| \\
            &= \bigg|\inner{(1 - \tilde{P}^\perp)\tilde{L}_{\sigma, T}\tilde{P}^\perp s, \tilde{P}^\perp s'}_{\sigma;0}\bigg| \nonumber\\
            &= \bigg|\inner{\tilde{P}\tilde{L}_{\sigma, T}\tilde{P}^\perp s, \tilde{P}^\perp s'}_{\sigma;0}\bigg| \nonumber\\
            &= \bigg|\inner{\tilde{L}_{\sigma, T}\tilde{P}^\perp s, \tilde{P}^{*,\sigma} \tilde{P}^\perp s'}_{\sigma;0} \bigg|\nonumber\\
            &= \bigg|\inner{\tilde{L}_{\sigma, T}\tilde{P}^\perp s, (\tilde{P}^{*,\sigma} - \tilde{P}) \tilde{P}^\perp s'}_{\sigma;0}\bigg| \nonumber\\
            &\le C\big|\tilde{P}^\perp s\big|_{\sigma, T; 1} \big|(\tilde{P}^{*,\sigma} - \tilde{P})\tilde{P}^\perp s\big|_{\sigma, T; 1}. \nonumber
    \end{align}

    In view of (\ref{eq:psi_sigma}) and Definition \ref{def:norm_sigma_T}, one has $\tilde{P}^{*,\sigma} = \psi_\sigma^{-2(n-q)} \tilde{P} \psi_{\sigma}^{2(n-q)}$. Since 
    \begin{equation*}
     \begin{split}
         \Bigg| \frac{\psi_{\sigma}^{2(n-q)}(X + X')}{\psi_{\sigma}^{2(n-q)}(X)} - 1 \Bigg| &= \Bigg| \Bigg(\frac{1 + \sqrt{1 + \abs{X + X'}^2} \varphi(\frac{\sigma(X+ X')}{2})}{1 + \sqrt{1 + \abs{X}^2} \varphi(\frac{\sigma X}{2})}\Bigg)^{2(n-q)} - 1 \Bigg| \\
         &\le C (1 + \abs{X'})^{2(n-q)},
     \end{split}
    \end{equation*}
    we find that 
    \begin{equation} \label{eq:P*-P}
        \big|(\tilde{P}^{*,\sigma} - \tilde{P})\tilde{P}^\perp s\big|_{\sigma, T; 1} \le C \big|\tilde{P}^\perp s\big|_{\sigma, T; 1}.
    \end{equation}
    Combining (\ref{eq:lam-L'^-1}), (\ref{eq:L'-L}) and (\ref{eq:P*-P}), we see that if $\abs{\lambda} \le T^{\frac{1}{8}}$, 
    \begin{equation} \label{eq:lam-L_1^-1}
            \norm{(\lambda - \tilde{L}_{\sigma, T, 1})^{-1}}_{\sigma, T; -1, 1} \le C (1 + \abs{\lambda}^2) \le C T^{\frac{1}{4}}.
    \end{equation}
    Since $(\lambda - \tilde{L}_{\sigma, T, 1})^{-1}$ maps $L^2(K)^\perp$ to $L^2(K)^\perp$, (\ref{eq:lam-L_1^-1}) implies 
    \begin{equation}
        \norm{(\lambda - \tilde{L}_{\sigma, T, 1})^{-1}}_{\sigma, T; -1, 0} \le \frac{C}{T^{3/4}}.
    \end{equation}

    Now, we notice that 
    \begin{equation}
        \begin{split}
            &\tilde{L}_{\sigma, T, 2}'  (\lambda - \tilde{L}_{\sigma, T, 1})^{-1} \tilde{L}_{\sigma, T, 2} - L_2' (-L_1)^{-1} L_2 \\
            &= TL_2'(\lambda - \tilde{L}_{\sigma, T, 1})^{-1} (\lambda - \tilde{L}_{\sigma, T, 1} + T^2 L_1) (-T^2  L_1)^{-1} TL_2  \\
            &+ T L_2' (\lambda - \tilde{L}_{\sigma, T, 1})^{-1} (\tilde{L}_{\sigma, T, 2} - T L_2) +  (\tilde{L}_{\sigma, T, 2}' - T L_2')(\lambda - \tilde{L}_{\sigma, T, 1})^{-1}  T L_2 \\
            & + (\tilde{L}_{\sigma, T, 2}' - T L_2')(\lambda - \tilde{L}_{\sigma, T, 1})^{-1} (\tilde{L}_{\sigma, T, 2} - T L_2)
        \end{split}
    \end{equation}

    By proceeding as above, we find that 
    \begin{enumerate}
        \item[i)]   $\norm{\tilde{L}_{\sigma, T, 2}}_{\sigma, T; 1, -1} \le C$, $\norm{TL_2}_{\sigma, T; 1, -1} \le C$, $\norm{TL_2}_{\sigma, T; 0, -1} \le C$.
        \item[ii)]  $\norm{\tilde{L}_{\sigma, T, 1}^{-1}}_{\sigma, T; -1, 1} \le C$, $\norm{(T^2L_1)^{-1}}_{\sigma, T; -1, 1} \le C$.
        \item[iii)] $\norm{\tilde{L}_{\sigma, T, 2} - TL_2}_{\sigma, T; 1,0} \le C$, $\norm{\tilde{L}_{\sigma, T, 2} - TL_2}_{\sigma, T; 1,-1} \le \frac{C}{T}$.
        \item[iv)]  For $\lambda \in \tilde{\Gamma}$ with $ \abs{\lambda}< T^{\frac{1}{8}}$, $\norm{\lambda - \tilde{L}_{\sigma, T, 1} + L_1}_{\sigma, T; 1,-1} \le C (1 + \abs{\lambda}) \le C T^{\frac{1}{8}}$.
        \item[v)] $\tilde{L}_{\sigma, T, 2}'$ and $L_2'$ satisfy the exact same estimate as $\tilde{L}_{\sigma, T, 2}$ and $L_2$.
    \end{enumerate}
    
    As a consequence, we have 
    \begin{equation} \label{eq:L_Theta_2}
        \norm{\tilde{L}_{\sigma, T, 2}'  (\lambda - \tilde{L}_{\sigma, T, 1})^{-1} \tilde{L}_{\sigma, T, 2} - L_2' (-L_1)^{-1} L_2}_{\sigma, T; 1, -1} \le \frac{C}{T^{1/2}}.
    \end{equation}
    Therefore by (\ref{eq:L_Theta_1}), (\ref{eq:U_3_Theta}), and (\ref{eq:L_Theta_2}), 
    \begin{equation}
        \Big\|\frac{1}{2\pi \i}\int_{\tilde{\Gamma} \cap \{\abs{\lambda} \le T^{\frac 18}\}} \tilde{F}_\sigma(\lambda) \left[\tilde{P}(\lambda - \tilde{L}_{\sigma, T})^{-1}\tilde{P} - \tilde{P}(\lambda - \Theta_\sigma)^{-1}\tilde{P}\right] d \lambda\Big\|_{\sigma; 0,0} \le \frac{C}{T^{1/2}}.
    \end{equation}

    On the other hand, using Proposition \ref{prop:F_and_G}, (\ref{eq:p_perp_rel}) and (\ref{eq:U_3}), it is not hard to see 
    \begin{equation}
        \begin{split}
            &\Big\|\frac{1}{2\pi \i}\int_{\tilde{\Gamma} \cap \{\abs{\lambda} \ge T^{\frac 18}\}} \tilde{F}_\sigma(\lambda) \left[\tilde{P}(\lambda - \tilde{L}_{\sigma, T})^{-1}\tilde{P} - \tilde{P}(\lambda - \Theta_\sigma)^{-1}\tilde{P}\right] d \lambda\Big\|_{\sigma; 0,0} \\
            &\le \Big\|\frac{1}{2\pi \i}\int_{\tilde{\Gamma} \cap \{\abs{\lambda} \ge T^{\frac 18}\}} \tilde{F}_\sigma(\lambda) \left[\tilde{P}(\lambda - \tilde{L}_{\sigma, T})^{-1}\tilde{P} \right] d \lambda\Big\|_{\sigma; 0,0} \\
            &+ \Big\|\frac{1}{2\pi \i}\int_{\tilde{\Gamma} \cap \{\abs{\lambda} \ge T^{\frac 18}\}} \tilde{F}_\sigma(\lambda) \left[\tilde{P}(\lambda - \Theta_\sigma)^{-1}\tilde{P}\right] d \lambda\Big\|_{\sigma; 0,0} \\
            & \le \frac{C}{T},
        \end{split}
    \end{equation}
    from which the last inequality follows.
\end{proof}

\begin{proof}[Proof of Theorem \ref{thm:last_one}]
    In the proof we will use $\d p$ to denote the volume element $\d vol_{E_0}(p)$.
    Recall that we use $\tilde{K}_{\sigma, T}$ to denote the heat kernel of $\tilde{F}_\sigma(\tilde{L}_{\sigma, T})$. Let $K'_\sigma$ be the heat kernel of $\tilde{P} \tilde{F}_\sigma (\Theta_\sigma) \tilde{P}$. Set $f_{\sigma, T}(p, p') = \tilde{K}_{\sigma, T}(p, p') - K'_\sigma(p, p')$. 
    Let $\xi$ be a smooth nonnegative function on $E_0$ such that 
    \begin{equation} \label{eq:xi}
        \operatorname{supp} \xi \subset B^{E_0}_1(0), \text{ and }\int_{E_0} \xi(p) \d p= 1.
    \end{equation}

    Let $a \in (0,1]$ be a constant. Since $f_{\sigma, T}$ is $\End(\Lambda^*E_0)$-valued, we take $U, U' \in \Lambda^*E_0$. By Proposition \ref{thm:unif_bound}, for any $p, p' \in E_0$, 
    \begin{align}
        \label{eq:f_0-f_p}
            \bigg| &\inner{f_{\sigma, T}(0,0)U, U'} 
        - \iint_{E_0\times E_0} \inner{f_{\sigma, T}(p, p')U, U'}\frac{\xi(p/a)}{a^{2(n+k)}} \frac{\xi(p'/a)}{a^{2(n+k)}} \d p \d p' \bigg| \\
        &= \bigg| \iint_{E_0\times E_0} \inner{(f_{\sigma, T}(p, p') - f_{\sigma, T}(0,0))U, U'}\frac{\xi(p/a)}{a^{2(n+k)}} \frac{\xi(p'/a)}{a^{2(n+k)}} \d p \d p' \bigg|\nonumber \\
        &= \bigg| \iint_{E_0\times E_0} \inner{(f_{\sigma, T}(ap, ap') - f_{\sigma, T}(0,0))U, U'}\xi(p) \xi(p') \d p \d p' \bigg|\nonumber \\
        &\le C a \cdot \iint_{E_0\times E_0} \xi(p) \xi(p') \d p \d p' = Ca.\nonumber
    \end{align}

    On the other hand, using Proposition \ref{prop:last_prop} we deduce that 
    \begin{align}  \label{eq:f_p}
            &\bigg|\iint_{E_0\times E_0} \inner{f_{\sigma, T}(p, p')U, U'}\frac{\xi(p/a)}{a^{2(n+k)}} \frac{\xi(p'/a)}{a^{2(n+k)}} \d p \d p'\bigg| \\
            &\le \bigg[ \iint_{E_0\times E_0} \inner{f_{\sigma, T}(p, p')U, U'}^2 \d p \d p'\bigg]^{\frac 12} \cdot \norm{\xi}_0^2 \le \frac{C}{T^{1/2}}. \nonumber
    \end{align}

    Take $a = \frac{1}{T^{1/2}} \le 1$. By (\ref{eq:f_0-f_p}) and (\ref{eq:f_p}), we find that 
\begin{equation}
    \abs{\inner{f_{\sigma, T}(0,0)U, U'}} \le \frac{C}{T^{1/2}}. 
\end{equation}
The exact same estimate holds for any point $(p,p) \in E_0 \times E_0$ on the diagonal. Using the McKean-Singer formula, we have 
\begin{equation}
    \left| \tr_s \left( N_Y \tilde{F}_\sigma(\tilde{L}_{\sigma, T})\right) - \tr_s \left( N_Y \tilde{P} \tilde{F}_\sigma (\Theta_\sigma) \tilde{P}\right) \right| \le \frac{C}{T^{1/2}}.
\end{equation}

By Remark \ref{rmk:ident_Theta_D_M} and the finite propagation speed property, we find that 
\begin{equation}
    \tr_s \left( N_Y \tilde{P} \tilde{F}_\sigma (\Theta_\sigma) \tilde{P}\right) = \tr_s \left( N_Y G_2 \circ (\sigma D_M)^2 \circ G_2^{-1}\right) = 0,
\end{equation}
where the second equality is due to a parity argument as in the proof of Theorem \ref{thm:interm_result_I_4_1}. The proof is completed.
\end{proof}

\bigskip

\printbibliography

\end{document}